\documentclass[11pt,reqno]{amsart}

\newtheorem{thm}{Theorem}[section]
\newtheorem*{thm*}{Theorem}
\newtheorem{lem}[thm]{Lemma}
\newtheorem*{lem*}{Lemma}

\newtheorem{cor}[thm]{Corollary}

\newtheorem{prop}[thm]{Proposition}

\theoremstyle{definition}

\newtheorem{assump}[thm]{Assumption}
\newtheorem*{case*}{Case}

\newtheorem{defn}[thm]{Definition}
\newtheorem*{defn*}{Definition}

\newtheorem*{exmp*}{Example}

\newtheorem{hyp}[thm]{Hypothesis}

\renewcommand{\thestep}{}

\theoremstyle{remark}

\renewcommand{\thecase}{}

\newtheorem{rmk}[thm]{Remark}
\newtheorem*{rmk*}{Remark}

\makeatletter
\def\alphenumi{
  \def\theenumi{\alph{enumi}}
  \def\p@enumi{\theenumi}
  \def\labelenumi{(\@alph\c@enumi)}}
\makeatother

\makeatletter
\def\thecase{\@arabic\c@case}
\makeatother

\makeatletter
\def\thestep{\@arabic\c@step}
\makeatother

\newcount\hh
\newcount\mm
\mm=\time
\hh=\time
\divide\hh by 60
\divide\mm by 60
\multiply\mm by 60
\mm=-\mm
\advance\mm by \time
\def\hhmm{\number\hh:\ifnum\mm<10{}0\fi\number\mm}

\let\oldmarginpar\marginpar
\renewcommand\marginpar[1]{\-\oldmarginpar[\raggedleft\footnotesize #1]%
{\raggedright\footnotesize #1}}

\renewcommand\emptyset{\varnothing}

\newcommand\CC{\mathbb{C}}

\newcommand\HH{\mathbb{H}}

\newcommand\NN{\mathbb{N}}

\newcommand\RR{\mathbb{R}}

\newcommand\sA{{\mathscr{A}}}

\newcommand\sD{{\mathscr{D}}}

\newcommand\sO{{\mathscr{O}}}

\newcommand\sS{{\mathscr{S}}}

\newcommand\eps{\varepsilon}

\newcommand\less{\setminus}

\DeclareMathOperator{\Arg}{Arg}

\newcommand\diam{\operatorname{diam}}

\newcommand\dist{\operatorname{dist}}

\DeclareMathOperator{\height}{height}

\DeclareMathOperator{\Int}{int}
\newcommand\Imag{\operatorname{Im}}

\newcommand\loc{\operatorname{loc}}
\DeclareMathOperator{\Log}{Log}

\newcommand\Real{\operatorname{Re}}

\newcommand\supp{\operatorname{supp}}

\newcommand\tr{\operatorname{tr}}

\numberwithin{equation}{section}

\usepackage{amssymb}
\usepackage{graphicx}
\usepackage{hyperref}
\usepackage{mathrsfs}
\usepackage[usenames]{color}
\usepackage{slashed}
\usepackage{url}
\usepackage{verbatim}

\hypersetup{pdftitle={A classical Perron method for existence of smooth solutions to boundary value and obstacle problems for degenerate-elliptic operators via holomorphic maps}}
\hypersetup{pdfauthor={Paul M. N. Feehan}}

\begin{document}

\title[Boundary value and obstacle problems for degenerate-elliptic operators]{A classical Perron method for existence of smooth solutions to boundary value and obstacle problems for degenerate-elliptic operators via holomorphic maps}

\author[Paul M. N. Feehan]{Paul M. N. Feehan}
\address{Department of Mathematics, Rutgers, The State University of New Jersey, 110 Frelinghuysen Road, Piscataway, NJ 08854-8019}
\email{feehan@math.rutgers.edu}
\address{Department of Mathematics, Columbia University, New York, NY 10027}

\date{April 18, 2013}
%\date{\today{ }\hhmm}

\begin{abstract}
We prove existence of solutions to boundary value problems and obstacle problems for degenerate-elliptic, linear, second-order partial differential operators with partial Dirichlet boundary conditions using a new version of the Perron method. The elliptic operators considered have a degeneracy along a portion of the domain boundary which is similar to the degeneracy of a model linear operator identified by Daskalopoulos and Hamilton \cite{DaskalHamilton1998} in their study of the porous medium equation or the degeneracy of the Heston operator \cite{Heston1993} in mathematical finance. Existence of a solution to the partial Dirichlet problem on a half-ball, where the operator becomes degenerate on the flat boundary and a Dirichlet condition is only imposed on the spherical boundary, provides the key additional ingredient required for our Perron method. Surprisingly, proving existence of a solution to this partial Dirichlet problem with ``mixed'' boundary conditions on a half-ball is more challenging than one might expect. Due to the difficulty in developing a global Schauder estimate and due to compatibility conditions arising where the ``degenerate'' and ``non-degenerate boundaries'' touch, one cannot directly apply the continuity or approximate solution methods. However, in dimension two, there is a holomorphic map from the half-disk onto the infinite strip in the complex plane and one can extend this definition to higher dimensions to give a diffeomorphism from the half-ball onto the infinite ``slab''. The solution to the partial Dirichlet problem on the half-ball can thus be converted to a partial Dirichlet problem on the slab, albeit for an operator which now has exponentially growing coefficients. The required Schauder regularity theory and existence of a solution to the partial Dirichlet problem on the slab can nevertheless be obtained using previous work of the author and C. Pop \cite{Feehan_Pop_elliptichestonschauder}. Our Perron method relies on weak and strong maximum principles for degenerate-elliptic operators, concepts of continuous subsolutions and supersolutions for  boundary value and obstacle problems for degenerate-elliptic operators, and maximum and comparison principle estimates previously developed by the author \cite{Feehan_maximumprinciple_v1}.
\end{abstract}

% AMS 2010 subject classifications (used in AMS journals)
% Primary
% 35J70  	Degenerate elliptic equations
% 35J86  	Linear elliptic unilateral problems and linear elliptic variational inequalities
% 49J40  	Variational methods including variational inequalities
% 35R45  	Partial differential inequalities
%
% Secondary
% 35R35  	Free boundary problems
% 60J60  	Diffusion processes
% 49J20  	Optimal control problems involving partial differential equations

% ``2010'' not recognized and defaults to 1991

% AMS keywords (used in AMS journals)

% Acknowledge support
\thanks{The author was partially supported by NSF grant DMS-1237722, a visiting faculty appointment in the Department of Mathematics at Columbia University, and the Max Planck Institut f\"ur Mathematik in der Naturwissenschaft, Leipzig.}

\maketitle
\tableofcontents
\listoffigures

\section{Introduction}
\label{sec:Introduction}
Suppose $\sO\subseteqq\HH$ is a domain (possibly unbounded) in the open upper half-space $\HH := \RR^{d-1}\times\RR_+$, where $d\geq 2$ and $\RR_+ := (0,\infty)$, and $\partial_1\sO := \partial\sO\cap\HH$ is the portion of the boundary $\partial\sO$ of $\sO$ which lies in $\HH$, and $\partial_0\sO$ is the interior of $\partial\HH\cap\partial\sO$, where $\partial\HH = \RR^{d-1}\times\{0\}$ is the boundary of $\bar\HH := \RR^{d-1}\times\bar\RR_+$ and $\bar\RR_+ := [0,\infty)$. We assume $\partial\HH\cap\partial\sO$ is non-empty throughout this article, though $\partial_0\sO$ may be empty. We denote $\underline\sO := \sO\cup\partial_0\sO$, while $\bar \sO=\sO\cup\partial\sO$ denotes the usual topological closure of $\sO$ in $\RR^d$.

We consider the elliptic equation with \emph{partial Dirichlet boundary condition},
\begin{align}
\label{eq:Elliptic_equation}
Au &= f \quad \hbox{on }\sO,
\\
\label{eq:Elliptic_boundary_condition}
u &= g \quad \hbox{on } \partial_1\sO,
\end{align}
for a suitably regular function $u$ on $\underline\sO\cup\partial_1\sO$, given a suitably regular source function $f$ on $\underline\sO$ and boundary data $g$ on $\partial_1\sO$, and the obstacle problem,
\begin{equation}
\label{eq:Elliptic_obstacle_problem}
\min\{Au-f, \ u-\psi\} = 0 \quad \hbox{a.e. on }\sO,
\end{equation}
with partial Dirichlet boundary condition \eqref{eq:Elliptic_boundary_condition}, given a suitably regular obstacle function $\psi$ on $\underline\sO\cup\partial_1\sO$ which is compatible with $g$ in the sense that
\begin{equation}
\label{eq:Boundarydata_obstacle_compatibility}
\psi\leq g \quad\hbox{on } \partial_1\sO.
\end{equation}
In particular, \emph{no boundary condition} is prescribed along $\partial\sO\less\partial_1\sO$, or even $\partial_0\sO$, provided a solution $u$ is sufficiently regular up to $\partial_0\sO$ and the coefficients of $A$ have suitable properties. See Figure \ref{fig:domain}. This can be important in applications (whether to the porous medium equation \cite{DaskalHamilton1998}, mathematical biology \cite{Epstein_Mazzeo_2011}, or mathematical finance \cite{Heston1993}), since a Dirichlet condition along $\partial_0\sO$ is often unmotivated by applications. If one elects to impose a Dirichlet condition along $\partial_0\sO$ anyway, regardless of motivation, simple examples (using the Kummer equation \cite{Feehan_Pop_elliptichestonschauder}) illustrate that this limits the regularity of the solution up to $\partial_0\sO$ to being at most continuous. In particular, the boundary value problem would become ill-posed if we attempted to seek a solution which is smooth up to $\partial_0\sO$ while simultaneously imposing a Dirichlet boundary condition on the full boundary, $\partial\sO$.

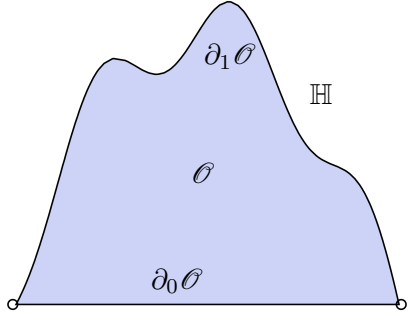
\begin{figure}
 \centering
 \begin{picture}(210,210)(0,0)
 \put(0,0){\includegraphics[scale=0.6]{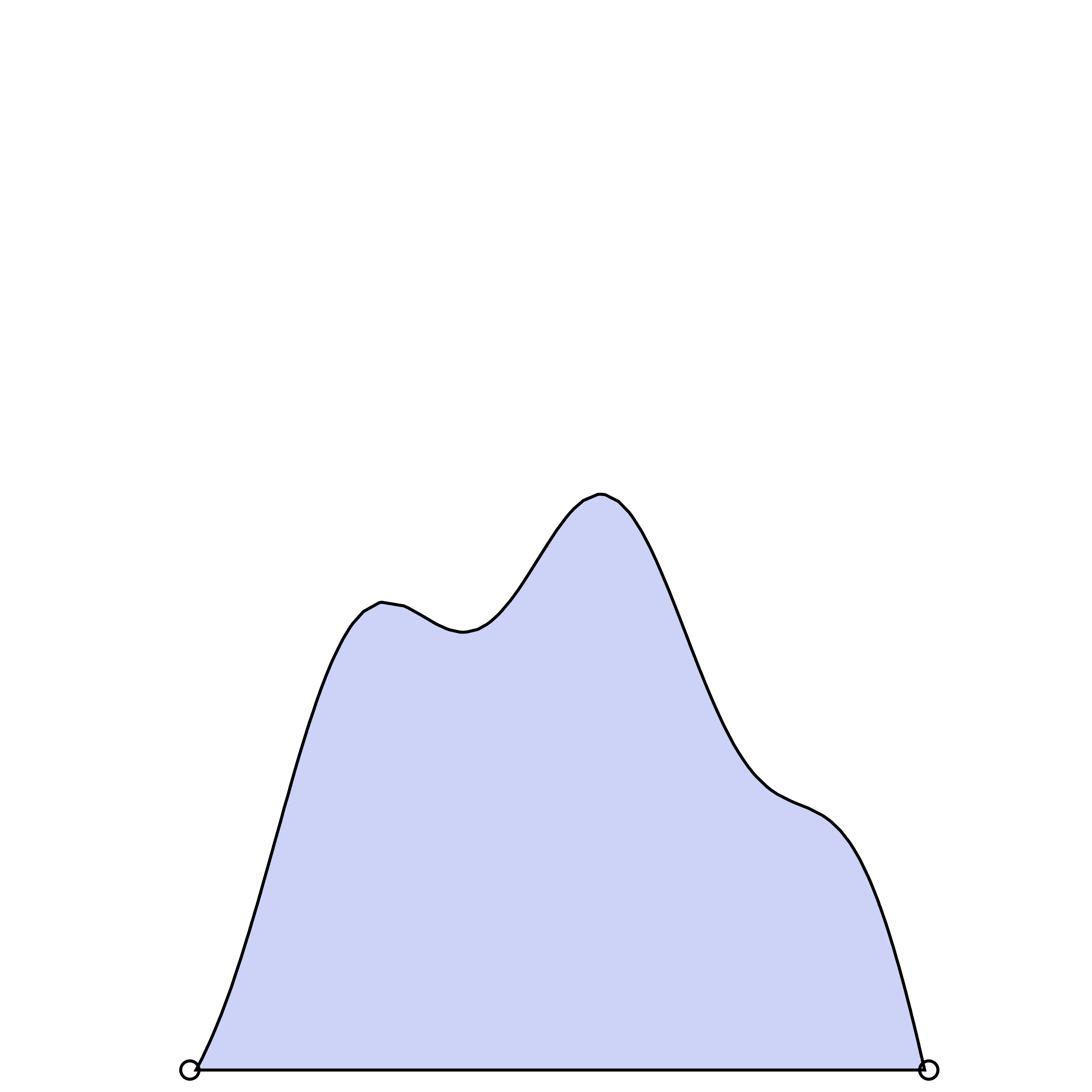}}
 \put(90,10){$\partial_0\sO$}
 \put(105,50){$\sO$}
 \put(110,95){$\partial_1\sO$}
 \put(150,80){$\HH$}
 \end{picture}
 \caption[A domain and its ``degenerate'' and ``non-degenerate'' boundaries.]{A subdomain, $\sO$, of the upper half-space, $\HH$, and the ``degenerate'' and ``non-degenerate'' boundaries, $\partial_0\sO$ and $\partial_1\sO$.}
 \label{fig:domain}
\end{figure}

We allow for the possibility that $\sO=\HH$, in which case $\partial_1\sO$ is empty and the boundary condition \eqref{eq:Elliptic_boundary_condition} and compatibility condition \eqref{eq:Boundarydata_obstacle_compatibility} are omitted.

The operator $A$ will have the form
\begin{equation}
\label{eq:Generator}
Av :=  -x_d\tr(aD^2v) - \langle b, Dv\rangle + cv \quad\hbox{on }\sO, \quad v\in C^\infty(\sO),
\end{equation}
where $x=(x_1,\ldots,x_d)$ are the standard coordinates on $\RR^d$ and the coefficients of $A$ are given by a matrix-valued function $a:\underline\sO\to\RR^{d\times d}$, a vector field $b:\underline\sO\to\RR^d$, and a function $c:\underline\sO\to \RR$. We shall endow these coefficients with additional properties, some of which we enumerate further below for convenience though we only assume them when explicitly invoked.
They will include a requirement that $a:\underline\sO\to\RR^{d\times d}$ be locally bounded on a neighborhood of $\partial_0\sO$ in $\underline\sO$ and that the vector field $b$ obey $\langle b, \vec n\rangle \geq 0$ along $\partial_0\sO$, where $\vec n$ is the \emph{inward}-pointing unit normal vector field. Clearly, $A$ becomes degenerate when $x_d=0$, that is, along $\partial_0\sO$, and thus we refer to $\partial_0\sO$ as the ``degenerate boundary'' portion and, because $A$ will typically be elliptic when $x_d>0$, we refer to $\partial_1\sO$ as the ``non-degenerate boundary'' portion.

When $u\in C^2(\sO)\cap C^1(\underline\sO)$ obeys a \emph{second-order boundary condition} along $\partial_0\sO$, in the sense that $x_dD^2u$ extends continuously to $\partial_0\sO$ with limit zero on $\partial_0\sO$, and solves \eqref{eq:Elliptic_equation}, then $u$ obeys the equivalent \emph{implicit oblique derivative boundary condition},
\begin{equation}
\label{eq:Equivalent_elliptic_oblique_derivative_boundary_condition}
- \langle b, Du\rangle + cu = f  \quad\hbox{on }\partial_0\sO,
\end{equation}
and so we may consider \eqref{eq:Elliptic_equation} as a degenerate-elliptic boundary value problem with accompanying ``mixed'' boundary conditions, \eqref{eq:Elliptic_boundary_condition} along $\partial_1\sO$ and \eqref{eq:Equivalent_elliptic_oblique_derivative_boundary_condition} along $\partial_0\sO$.

\subsection{Connections with previous research}
\label{subsec:Survey}
Before turning to a discussion of the detailed properties of the operator $A$ in \eqref{eq:Generator} and a description of our main results, we first summarize some of previous research on problems of this type and highlight a few of the difficulties which give the ``mixed boundary value'' problems addressed in this article their distinctive character. When the elliptic Dirichlet problem \eqref{eq:Elliptic_equation}, \eqref{eq:Elliptic_boundary_condition} is replaced by a parabolic terminal value problem
\begin{equation}
\label{eq:Parabolic_terminal_value_problem}
-u_t + Au = f \quad\hbox{on }(0,T)\times\sO, \quad u(T,\cdot) = h \quad\hbox{on }\bar\sO,
\end{equation}
on the half-space $\sO=\HH$, where now $d=2$ and $A$ has constant coefficients, $f\in C^\alpha_s([0,T]\times\bar\HH)$ and $h\in C^\alpha_s(\bar\HH)$ have \emph{compact support}, then Daskalopoulos and Hamilton prove existence of a unique solution $u\in C^{2+\alpha}_s([0,T]\times\bar\HH)$ \cite[Theorem I.1.1]{DaskalHamilton1998} (see \S \ref{subsec:Holder} for definitions of the elliptic analogues of their H\"older spaces). Closely related results for this problem are obtained by Koch using variational methods and carefully chosen weighted Sobolev spaces \cite{Koch}. For this purpose, the Daskalopoulos-Hamilton analysis requires an important (and difficult) a priori interior Schauder estimate for a solution $u\in C^{2+\alpha}_s([0,T]\times\underline U)$ for a bounded open subset $U\Subset\underline\HH$ \cite[Theorem I.1.3]{DaskalHamilton1998}, where points in $\partial\HH$ are regarded as ``interior'' (an essential distinction throughout \cite{DaskalHamilton1998}), However, their analysis does \emph{not} use or require an a priori \emph{global} Schauder estimate for a solution $u\in C^{2+\alpha}_s([0,T]\times\bar U)$ with a ``mixed'' boundary condition where $u$ obeys a Dirichlet condition along $(0,T)\times\partial_1U$ and the equivalent implied oblique derivative condition,
\begin{equation}
\label{eq:Equivalent_parabolic_oblique_derivative_boundary_condition}
-u_t - \langle b, Du\rangle + cu = f  \quad\hbox{on }(0,T)\times\partial_0U.
\end{equation}
%COMMENT Add reference to smoothing operators defined in section I.11 of DH
This is a significant point because, as far as we can tell, the development of such an a priori \emph{global} Schauder estimate --- whether in the parabolic setting of \cite{DaskalHamilton1998} or the elliptic setting of our article --- does not appear to be straightforward, except in the special case that the boundary portions $\partial_0\sO$ and $\partial_1\sO$ lie a positive distance apart and thus $\partial\sO = \partial_0\sO\cup\partial_1\sO$ is smooth if this is true of $\partial_0\sO$ and $\partial_1\sO$. An important example of such a domain, one which we exploit in this article, is the case of an infinite ``slab'', $S = \RR^{d-1}\times(0,\nu)$, with $\partial_0S = \RR^{d-1}\times\{0\}$ and $\partial_1S = \RR^{d-1}\times\{\nu\}$. In the typical applications of our article, however, the closures of the boundary portions $\partial_0\sO$ and $\partial_1\sO$ will touch at corner points as in Figure \ref{fig:domain}. Part of the reason for the difficulties in proofs of existence of solutions caused by these corner points is explained in \S \ref{subsec:Corner_point_compatibility}; see also \cite{Feehan_Pop_higherregularityweaksoln}.

When $\sO\subset\RR^d$ is a bounded domain and $A$ has variable coefficients and takes the form \eqref{eq:Generator} in local coordinates on neighborhoods of boundary points in $\partial\sO$ (see Figure \ref{fig:half_space_solution_and_patching_solutions}) and is elliptic in the interior of $\sO$, then analogues of standard covering arguments, constructions of approximate solutions, and a version of the method of continuity allow Daskalopoulos and Hamilton to obtain an a priori global Schauder estimate together with existence and uniqueness of a solution $u\in C^{2+\alpha}_s([0,T]\times\bar\sO)$ \cite[Theorem II.1.1]{DaskalHamilton1998}. Indeed, they prove \cite[Theorem II.1.1]{DaskalHamilton1998} by exploiting their a priori interior Schauder estimate together with existence and uniqueness of solutions to a degenerate-parabolic terminal value problem on a half-space with compactly supported source function and initial data and standard results for strictly parabolic terminal value problems on bounded domains \cite{Krylov_LecturesHolder}. However, in this setting, the boundary condition is not mixed, since there is no Dirichlet condition along any portion of $\partial\sO$, only the condition \eqref{eq:Equivalent_parabolic_oblique_derivative_boundary_condition} implied by the regularity of $u$ up to $\partial\sO$.

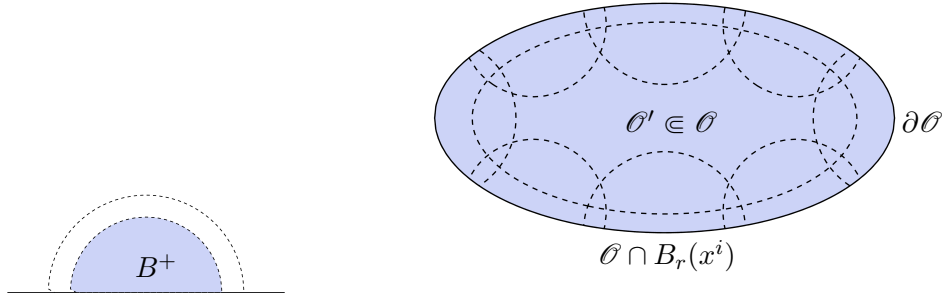
\begin{figure}
 \centering
 \begin{picture}(450,180)(0,0)
 \put(40,-20){\includegraphics[scale=0.3]{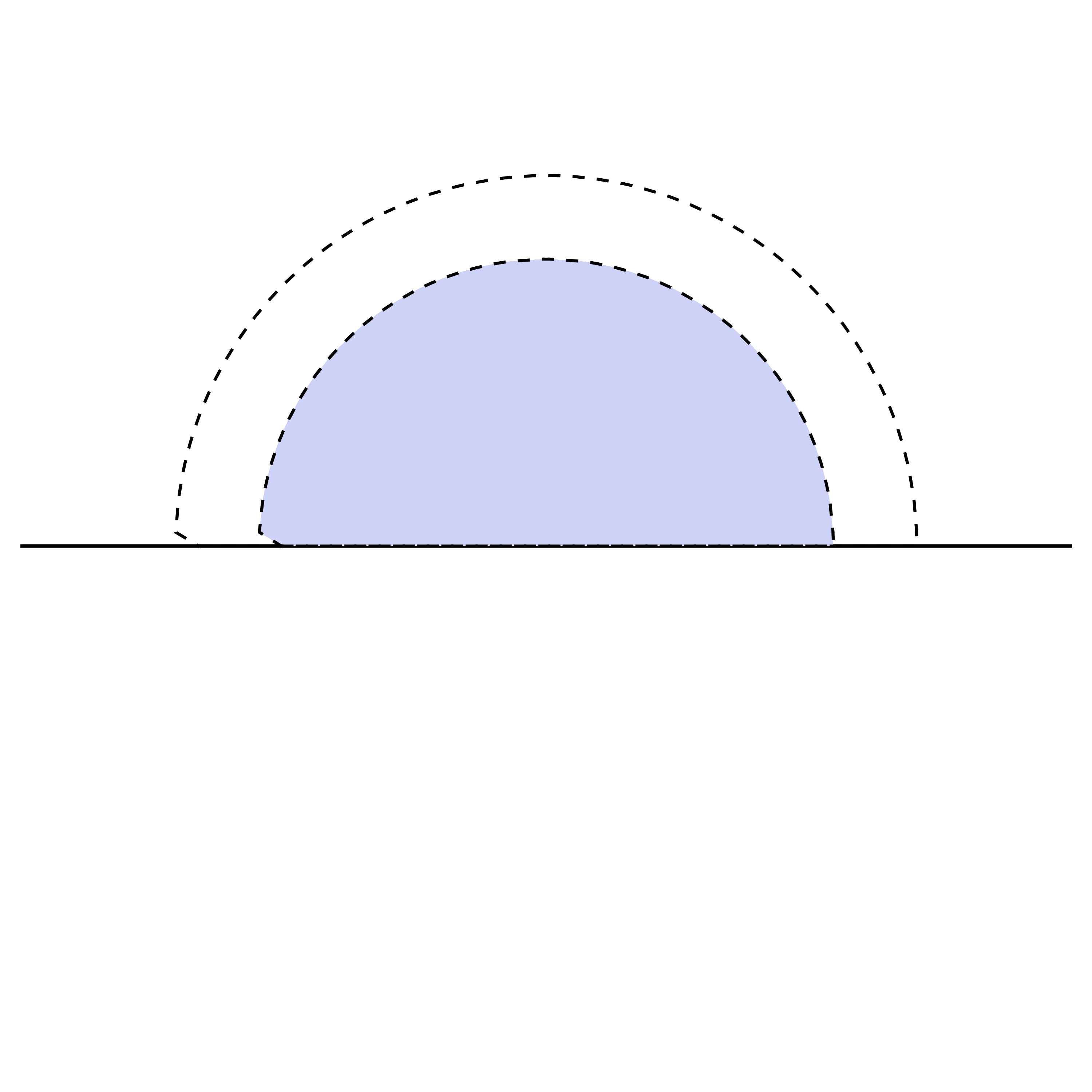}}
 \put(90,40){$B^+$}
 \put(200,10){\includegraphics[scale=0.5]{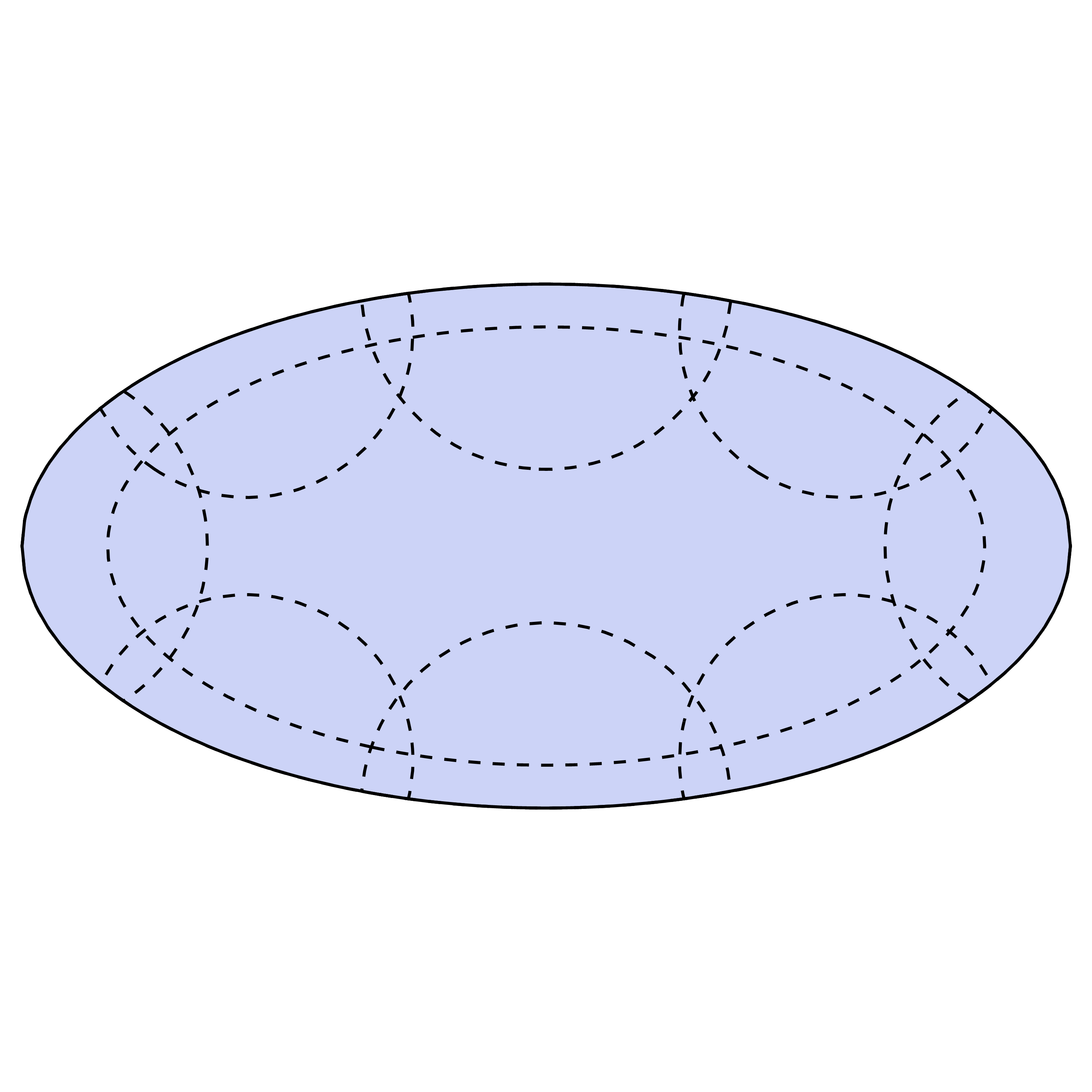}}
 \put(275,95){$\sO'\Subset\sO$}
 \put(265,45){$\sO\cap B_r(x^i)$}
 \put(380,95){$\partial\sO$}
 \end{picture}
 \caption[Construction of approximate solutions on a domain with a fully ``degenerate'' boundary.]{An approximate solution on a domain, $\sO$, with ``fully degenerate'' boundary, $\partial_0\sO=\partial\sO$, obtained by patching solutions on half-balls, $\sO\cap B_r(x^i) \cong B^+\subset\HH$, and a precompact $C^{2,\alpha}$ subdomain $\sO'\Subset\sO$.}
 \label{fig:half_space_solution_and_patching_solutions}
\end{figure}

By extending the methods of Koch \cite{Koch} and restricting to the case where $A$ is the Heston operator \eqref{eq:Heston_generator} and $d=2$, Daskalopoulos, Pop, and the author succeeded in proving existence, uniqueness, and regularity os solutions to the equation \eqref{eq:Elliptic_equation} and obstacle problem \eqref{eq:Elliptic_obstacle_problem} with partial Dirichlet boundary condition \eqref{eq:Elliptic_boundary_condition} by solving the associated variational equation and inequality for solutions, $u$, in weighted Sobolev spaces \cite{Daskalopoulos_Feehan_statvarineqheston}; appealing to \cite{Feehan_maximumprinciple_v1} to obtain uniqueness; proving that the solutions are continuous up to the boundary \cite{Feehan_Pop_regularityweaksoln} using a Moser iteration technique; proving Schauder regularity when $f\in C^\infty_0(\underline\sO)\cap C_b(\sO)$ using a variational method \cite{Feehan_Pop_higherregularityweaksoln}; proving the expected Schauder regularity, $u\in C^{2+\alpha}_s(\underline\sO)\cap C_b(\sO\cup\partial_1\sO)$ for a solution to the equation \eqref{eq:Elliptic_equation}, using elliptic a priori interior Schauder estimates and regularity results in \cite{Feehan_Pop_elliptichestonschauder}; and proving that $u\in C^{1,1}_s(\underline\sO)\cap C_b(\sO\cup\partial_1\sO)$ for a solution to the obstacle problem \eqref{eq:Elliptic_obstacle_problem} by adapting arguments of Caffarelli \cite{Caffarelli_jfa_1998} in \cite{Daskalopoulos_Feehan_optimalregstatheston}.

However, the Perron methods which we use in this article, and which we outline in \S \ref{subsec:Guide}, provide a more direct and elegant path to the desired results, even though they do not establish continuity of the solution at the corner points, although such continuity properties are proved by Pop and the author by a variational method in \cite{Feehan_Pop_regularityweaksoln} for the case of the Heston operator \eqref{eq:Heston_generator}. It is important to note that the Perron methods developed in this article are analogues of their classical counterpart in \cite[Chapters 2 and 6]{GilbargTrudinger} for the existence of smooth solutions to a Dirichlet problem for a linear, second-order, strictly elliptic operator. While one could envisage first proving existence of viscosity solutions to the boundary value and obstacle problems considered in this article by adapting previous work of Barles \cite{Barles_1993} and Ishii \cite{Ishii_1991} for existence of viscosity solutions to fully nonlinear boundary value problems with fully nonlinear boundary conditions and then proving the expected Schauder regularity by adapting the methods of \cite{Feehan_Pop_elliptichestonschauder}, this approach is less straightforward and less direct than it might appear at first glance. In particular, standard comparison theorems \cite{Crandall_Ishii_Lions_1992} for viscosity solutions do not immediately apply to problems such as \eqref{eq:Elliptic_equation} and \eqref{eq:Elliptic_obstacle_problem} with the partial Dirichlet boundary condition \eqref{eq:Elliptic_boundary_condition}, or mixed boundary condition \eqref{eq:Elliptic_boundary_condition} and \eqref{eq:Equivalent_elliptic_oblique_derivative_boundary_condition}.

\subsection{Properties of the coefficients of the operator $A$}
\label{subsec:Properties_A}
We now discuss some of the conditions we shall impose on the coefficients $A$ in \eqref{eq:Generator} in this article, emphasizing that in any specific result we only impose those conditions which are explicitly referenced. We shall often require that $a(x)$ be \emph{symmetric} and \emph{strictly elliptic} for $x\in\underline\sO$ in the sense that \cite[p. 31]{GilbargTrudinger}
\begin{equation}
\label{eq:Strict_ellipticity_domain}
\langle a\xi, \xi\rangle \geq \lambda_0 |\xi|^2\quad\hbox{on } \underline\sO,\quad \forall\,\xi \in \RR^d \quad\hbox{(strict ellipticity)},
\end{equation}
for some positive constant $\lambda_0$. The component $b^d$ of the vector field $b=(b^1,\ldots,b^d)$ or the coefficient $c$ may be constrained by requiring that
\begin{align}
\label{eq:Nonnegative_bd_boundary}
b^d &\geq 0 \quad \hbox{on } \partial_0 \sO  \quad\hbox{(non-negativity of $\langle b, e_d\rangle$ on boundary), or}
\\
\label{eq:Nonnegative_c}
c &\geq 0 \quad \hbox{on } \underline\sO \quad\hbox{(non-negativity of $c$ on domain)},
\end{align}
or that
\begin{align}
\label{eq:Positive_bd_boundary}
b^d &> 0 \quad \hbox{on } \partial_0 \sO  \quad\hbox{(positivity of $\langle b, e_d\rangle$ on boundary), or}
\\
\label{eq:Positive_c_boundary}
c &> 0 \quad \hbox{on }  \partial_0\sO \quad\hbox{(positivity of $c$ on boundary)},
\end{align}
or that there are positive constants $b_0$ or $c_0$ such that
\begin{align}
\label{eq:Positive_lower_bound_bd_boundary}
b^d &\geq b_0 \quad \hbox{on } \partial_0 \sO  \quad\hbox{(positive lower bound for $\langle b, e_d\rangle$ on boundary), or}
\\
\label{eq:Positive_lower_bound_c_domain}
c &\geq c_0 \quad \hbox{on } \underline\sO \quad\hbox{(positive lower bound for $c$ on domain)}.
\end{align}
We may also impose the conditions\footnote{Note that the condition \eqref{eq:Positive_lower_bound_bd_domain} gives a uniform positive lower bound on $b^d$ on $\underline\sO$ and not just $\partial_0\sO$ as in \eqref{eq:Positive_lower_bound_bd_boundary}. While we could alternatively ask that $b^i/a^{ii}$ is bounded below on bounded subdomains of $\sO$ for some $i\in\{1,\ldots,d\}$, the condition is most commonly obeyed when $i=d$; see \S \ref{subsec:Heston} for a well-known example when $d=2$.}
\begin{align}
\label{eq:Finite_upper_bound_add_domain}
a^{dd} &\leq \Lambda \quad\hbox{on }\underline\sO \quad\hbox{(upper bound for $\langle ae_d, e_d\rangle$ on domain), or}
\\
\label{eq:Positive_lower_bound_bd_domain}
b^d &\geq b_0 \quad\hbox{on }\underline\sO   \quad\hbox{(positive lower bound for $\langle b, e_d\rangle$ on domain)},
\end{align}
for some positive constants $b_0$ and $\Lambda$. We say that $\sO$ has finite height,
\begin{equation}
\label{eq:Finite_height_domain}
\height(\sO)\leq \nu,
\end{equation}
when $\sO \subset\RR^d\times[0,\nu]$ for some positive constant $\nu$. When the domain $\sO$ is \emph{unbounded}, we will occasionally appeal to the growth condition,
\begin{equation}
\label{eq:Quadratic_growth}
\tr \left(x_da(x)\right) + \langle b(x),x\rangle \leq K(1+|x|^2), \quad\forall\, x \in \underline\sO  \quad\hbox{(quadratic growth for $a, b$)},
\end{equation}
for some\footnote{The actual values of constants such as $b_0, c_0, K, \lambda_0, \Lambda, \nu$ have no impact on our existence results.} positive constant $K$. We may also require the constants $c_0$ and $K$ in \eqref{eq:Positive_lower_bound_c_domain} and \eqref{eq:Quadratic_growth} to obey
\begin{equation}
\label{eq:Positive_lower_bound_c_geq_2K}
c_0 \geq 2K \quad\hbox{(strong quadratic growth for $a, b$)},
\end{equation}
further strengthening \eqref{eq:Quadratic_growth}. Connectedness of the domain $\sO$ plays an important role in applications of the strong maximum principle but, when $\sO$ is unbounded, we may also need to strengthen this condition to
\begin{equation}
\label{eq:Connectedness_domain_and_ball}
\sO\cap B_R \quad\hbox{is connected for all $R\geq R_0(\sO)$},
\end{equation}
where $B_R$ is the open ball with radius $R$ and center at the origin.

\subsection{Summary of main results}
\label{subsec:Summary}
We shall prove existence of solutions to the elliptic equation \eqref{eq:Elliptic_equation} and obstacle problem \eqref{eq:Elliptic_obstacle_problem} with partial Dirichlet boundary condition \eqref{eq:Elliptic_boundary_condition} using a version of the classical Perron method \cite[\S 2.8 \& 6.3]{GilbargTrudinger}. For that purpose, a crucial role is played by the following analogue for a half-ball $B_r^+(x_0)=\HH\cap B_r(x_0)$, with center $x_0\in\partial\HH$ and radius $r>0$, of the solution to the Dirichlet problem on a ball $B_r(x_0)\subset\RR^d$ \cite[Lemma 6.10]{GilbargTrudinger} for a strictly elliptic operator with coefficients in $C^\alpha(\bar B_r(x_0))$. We refer the reader to Definitions \ref{defn:Calphas} and \ref{defn:DH2spaces} for definitions of the indicated Daskalopoulos-Hamilton-Koch H\"older spaces.

\begin{thm}[Existence of a solution to the partial Dirichlet boundary value problem on a half-ball]
\label{thm:Existence_uniqueness_elliptic_Dirichlet_halfball}
Let $x_0\in\partial\HH$ and $r>0$, and $\alpha\in(0,1)$. Let $A$ be as in \eqref{eq:Generator} with coefficients belonging to $C^\alpha_s(\bar B_r^+(x_0))$, and obeying \eqref{eq:Strict_ellipticity_domain} and \eqref{eq:Nonnegative_c} on $\underline B^+_r(x_0)$ and \eqref{eq:Nonnegative_bd_boundary} on $\partial_0B^+_r(x_0)$. Assume in addition that one of the conditions, \eqref{eq:Positive_lower_bound_c_domain} or \eqref{eq:Positive_lower_bound_bd_domain}, holds on $\underline B^+_r(x_0)$.
\begin{enumerate}
\item (Smooth boundary data.) If
\begin{align*}
f &\in C^\alpha_s(\underline B_r^+(x_0)\cup\partial_1 B_r^+(x_0))\cap C_b(B_r^+(x_0)),
\\
g &\in C^{2+\alpha}_s(\underline B_r^+(x_0)\cup\partial_1 B_r^+(x_0))\cap C_b(B_r^+(x_0)),
\end{align*}
then there is a unique solution,
$$
u \in C^{2+\alpha}_s(\underline B_r^+(x_0)\cup\partial_1 B_r^+(x_0))\cap C_b(B_r^+(x_0)),
$$
to the partial Dirichlet boundary value problem \eqref{eq:Elliptic_equation}, \eqref{eq:Elliptic_boundary_condition} on $B_r^+(x_0)$.

\item (Continuous boundary data.) If
\begin{align*}
f &\in C^\alpha_s(\underline B_r^+(x_0))\cap C_b(B_r^+(x_0)),
\\
g &\in C_b(\partial_1 B_r^+(x_0)),
\end{align*}
then there is a unique solution,
$$
u \in C^{2+\alpha}_s(\underline B_r^+(x_0))\cap C_b(B_r^+(x_0)\cup\partial_1 B_r^+(x_0)),
$$
to the partial Dirichlet boundary value problem \eqref{eq:Elliptic_equation}, \eqref{eq:Elliptic_boundary_condition} on $B_r^+(x_0)$.
\end{enumerate}
\end{thm}

\begin{rmk}[Application of Theorem \ref{thm:Existence_uniqueness_elliptic_Dirichlet_halfball}]
\label{rmk:Boundary_property_bd}
In practice, we shall only need to apply Theorem \ref{thm:Existence_uniqueness_elliptic_Dirichlet_halfball} for $r>0$ small enough that the condition \eqref{eq:Positive_lower_bound_bd_domain} on $\underline B^+_r(x_0)$, that is, $b^d\geq b_0$ on $\underline B^+_r(x_0)$, is implied by the weaker condition \eqref{eq:Positive_lower_bound_bd_boundary} on $\partial_0B^+_r(x_0)$, that is, $b^d\geq b_0$  on $\partial_0B^+_r(x_0)$, and the fact that $b^d$ is continuous on $\underline \sO$ and $B^+_r(x_0)\Subset\underline \sO$.
\end{rmk}

Before turning to the partial Dirichlet problem on a domain $\sO\subseteqq \HH$, we provide the

\begin{defn}[Regular boundary point]
\label{defn:Regular_boundary_point_elliptic_equation}
If $u\in C^{2+\alpha}_s(\underline\sO)$ is a bounded solution to the elliptic equation \eqref{eq:Elliptic_equation}, we say that a point $x_0\in \partial_1\sO$ is \emph{regular with respect to $A$, $f$, and $g$} if the point $x_0$ admits a local barrier in the sense of \cite[p. 105]{GilbargTrudinger}; if $x_0$ is a regular point, then \cite[Lemma 6.12]{GilbargTrudinger} implies (see Remark 2 in \cite[p. 105]{GilbargTrudinger}) that
$$
\lim_{\sO\ni x\to x_0}u(x)=g(x_0).
$$
\end{defn}

For example, when $A$ obeys the hypotheses of Theorem \ref{thm:Existence_uniqueness_elliptic_Dirichlet}, a point $x_0\in\partial_1\sO$ will be regular if $\sO$ obeys an exterior sphere condition at $x_0$ \cite[p. 106]{GilbargTrudinger}, or even an exterior cone condition at $x_0$ \cite[Problem 6.3]{GilbargTrudinger}. We have the following analogue of \cite[Theorem 6.13]{GilbargTrudinger}.

\begin{thm}[Existence of a smooth solution to the partial Dirichlet boundary value problem]
\label{thm:Existence_uniqueness_elliptic_Dirichlet}
Let $\sO \subseteqq \HH$ be a domain and $\alpha\in(0,1)$. Let $A$ be as in \eqref{eq:Generator} with coefficients belonging to $C^\alpha(\underline\sO)$, and obeying \eqref{eq:Strict_ellipticity_domain}, \eqref{eq:Nonnegative_c}, and \eqref{eq:Positive_bd_boundary}. Moreover, the coefficients of $A$ should obey either
\begin{enumerate}
\item Condition \eqref{eq:Positive_lower_bound_c_domain}, or
\item Conditions \eqref{eq:Finite_upper_bound_add_domain} and \eqref{eq:Positive_lower_bound_bd_domain},
\end{enumerate}
and, in addition when $\sO$ is unbounded, \eqref{eq:Quadratic_growth}, \eqref{eq:Positive_lower_bound_c_geq_2K}, and $\sO$ should obey \eqref{eq:Connectedness_domain_and_ball}. If
$$
f \in C^\alpha_s(\underline\sO)\cap C_b(\sO) \quad\hbox{and}\quad g \in C_b(\partial_1\sO),
$$
and each point of $\partial_1\sO$ is regular with respect to $A$, $f$, and $g$ in the sense of Definition \ref{defn:Regular_boundary_point_elliptic_equation}, then there is a unique solution,
$$
u \in C^{2+\alpha}_s(\underline\sO)\cap C_b(\sO\cup\partial_1\sO),
$$
to the partial Dirichlet boundary value problem \eqref{eq:Elliptic_equation}, \eqref{eq:Elliptic_boundary_condition}.
\end{thm}

\begin{rmk}[Continuity up to $\partial_1\sO$ and  $\overline{\partial\sO}$]
More generally, given a solution $u \in C^{2+\alpha}_s(\underline\sO)$ to \eqref{eq:Elliptic_equation}, continuity up to $\partial_1\sO$ is assured\footnote{Continuity up to $\partial_0\sO$ is already implied by the fact that $u \in C^{2+\alpha}_s(\underline\sO)$.}
by the existence of a local barrier at each point of $\partial_1\sO$ \cite[pp. 104--106]{GilbargTrudinger}. However, because $A$ becomes degenerate when $x_d=0$, it is unclear how to construct a local barrier at the ``corner points'', $\overline{\partial_0\sO}\cap\overline{\partial_1\sO}$, so regularity of the solution at those points will be considered by a different method in a separate article. However, when $A$ is specialized to the Heston operator \eqref{eq:Heston_generator} and $\sO$ obeys an exterior an interior cone condition along $\overline{\partial_0\sO}\cap\overline{\partial_1\sO}$, then solutions to the boundary value problem \eqref{eq:Elliptic_equation}, \eqref{eq:Elliptic_boundary_condition} are continuous up to boundary $\overline{\partial_1\sO}$ when $g\in C(\overline{\partial_1\sO})$, by consequence of \cite[Theorem 1.12]{Feehan_Pop_regularityweaksoln}.
\end{rmk}

\begin{rmk}[Regularity up to $\partial_1\sO$]
When the coefficients of $A$ and $f$ in Theorem \ref{thm:Existence_uniqueness_elliptic_Dirichlet} belong to $C^\alpha_s(\underline\sO\cup\partial_1\sO)$, and $g$ belongs to $C^{2+\alpha}_s(\underline\sO\cup\partial_1\sO)$, and $\partial_1\sO$ is of class $C^{2,\alpha}$, then standard regularity theory for boundary value problems for strictly elliptic operators \cite[Lemma 6.18]{GilbargTrudinger} implies that the solution $u$ belongs to $C^{2+\alpha}_s(\underline\sO\cup\partial_1\sO)$.
\end{rmk}

We have the following analogue of \cite[Theorems 1.3.2, 1.3.4, 1.4.1, \& 1.4.3]{Friedman_1982}, \cite[Theorems 4.7.7 \& 5.6.1 and Corollary 5.6.3]{Rodrigues_1987}, \cite[4.27 \& 4.38]{Troianiello}.

\begin{thm}[Existence of a smooth solution to the obstacle problem]
\label{thm:Existence_uniqueness_elliptic_obstacle}
Let $\sO \subseteqq \HH$ be a domain, and $d<p<\infty$, and $\alpha\in(0,1)$. Assume the hypotheses for $A$, $f$, $g$, and $\sO$ in Theorem \ref{thm:Existence_uniqueness_elliptic_Dirichlet}. If
$$
\psi \in C^2(\underline\sO)\cap C(\sO\cup\partial_1\sO)
$$
obeys \eqref{eq:Boundarydata_obstacle_compatibility}, so $\psi\leq g$ on $\partial_1\sO$, and
$$
\sup_\sO\psi < \infty,
$$
and each point of $\partial_1\sO$ is regular with respect to $A$, $f$, and $g$ in the sense of Definition \ref{defn:Regular_boundary_point_elliptic_equation}, then there is a unique solution,
$$
u \in C^{2+\alpha}_s(\underline\Omega)\cap W^{2,p}_{\loc}(\sO)\cap C^{1,\alpha}_s(\underline\sO)\cap C_b(\sO\cup\partial_1\sO),
$$
to the obstacle problem \eqref{eq:Elliptic_obstacle_problem}, \eqref{eq:Elliptic_boundary_condition}, where $\Omega = \{x\in\sO:(u-\psi)(x)>0\}$.
\end{thm}

By analogy with \cite[Theorem 1.3.4]{Friedman_1982}, we observe that the condition $\psi \in C^2(\underline\sO\cup\partial_1\sO)$ in Theorem \ref{thm:Existence_uniqueness_elliptic_obstacle} may be weakened. We say that an obstacle function $\psi \in C(\underline\sO\cup\partial_1\sO)$ is \emph{Lipschitz} in $\sO$ and has \emph{locally finite concavity} in $\sO$ (in the sense of distributions) if
\begin{gather}
\label{eq:Elliptic_obstaclefunction_lipschitz_condition}
\psi \in C^{0,1}(\sO),
\\
\label{eq:Elliptic_obstaclefunction_locally_finite_concavity_condition}
D^2_\xi\psi \geq - C \quad\hbox{in }\sD'(U), \quad\forall\, \xi\in \RR^d, \ |\xi|=1,
\end{gather}
for every open subset $U\Subset\sO$ and some positive constant $C=C(U)$. We say that
$$
D^2_\xi\psi \geq 0 \quad\hbox{in }\sD'(U), \quad\forall\, \xi\in \RR^d, \ |\xi|=1,
$$
if and only if
$$
\int_U\psi D^2_\xi\varphi\,dx \geq 0, \quad\forall\, \xi\in \RR^d, \ |\xi|=1,
$$
for all $\varphi\in\sD(U) \equiv C^\infty_0(U)$ with $\varphi\geq 0$ on $U$. Condition \eqref{eq:Elliptic_obstaclefunction_locally_finite_concavity_condition} means that
$$
D^2_\xi\left(\psi + \frac{1}{2}C|x|^2\right) \geq 0 \quad\hbox{in }\sD'(U),
$$
and we call $\psi$ \emph{convex} in $U$ (in the sense of distributions) when $C(U)=0$.

\begin{thm}[Existence of a smooth solution to the partial Dirichlet obstacle problem with Lipschitz obstacle function with locally finite concavity]
\label{thm:Existence_uniqueness_elliptic_obstacle_lipschitz}
Assume the hypotheses of Theorem \ref{thm:Existence_uniqueness_elliptic_obstacle}, except that the assumption $\psi\in C^2(\underline\sO)\cap C(\sO\cup\partial_1\sO)$ is relaxed to $\psi \in C(\underline\sO\cup\partial_1\sO)$ obeying \eqref{eq:Elliptic_obstaclefunction_lipschitz_condition} and \eqref{eq:Elliptic_obstaclefunction_locally_finite_concavity_condition}. Then all the conclusions of Theorem \ref{thm:Existence_uniqueness_elliptic_obstacle} again hold, except that now
$$
u \in C^{2+\alpha}_s(\underline\Omega) \cap W^{2,p}_{\loc}(\sO) \cap C_b(\underline\sO).
$$
\end{thm}

\begin{rmk}[Regularity up to $\partial_1\sO$ and optimal regularity in $\sO$ and $\sO\cup\partial_1\sO$]
When the coefficients of $A$ and $f$ in Theorem \ref{thm:Existence_uniqueness_elliptic_obstacle} belong to $C^\alpha_s(\underline\sO\cup\partial_1\sO)$, and $g$ belongs to $C^{2+\alpha}_s(\underline\sO\cup\partial_1\sO)$, and $\psi$ belongs to $C^2(\underline\sO\cup\partial_1\sO)$, and $\partial_1\sO$ is of class $C^{2,\alpha}$, and then standard regularity theory for obstacle problems for strictly elliptic operators \cite[Theorem 1.3.2 or 1.3.5]{Friedman_1982} implies that the solution $u$ also belongs to $W^{2,p}_{\loc}(\sO\cup\partial_1\sO)$. If in addition, $a^{ij} \in C^{2,\alpha}(\sO)$ (respectively, $C^{2,\alpha}(\sO\cup\partial_1\sO)$), then $u\in W^{2,\infty}_{\loc}(\sO)$ (respectively, $u\in W^{2,\infty}_{\loc}(\sO\cup\partial_1\sO)$) \cite[Theorems 1.4.1 and 1.4.3]{Friedman_1982}, \cite{Jensen_1980}, \cite[Theorem 4.38]{Troianiello}. Similar remarks apply to Theorem \ref{thm:Existence_uniqueness_elliptic_obstacle_lipschitz}.
\end{rmk}

\begin{rmk}[Optimal regularity of a solution to the obstacle problem up to $\partial_0\sO$]
Optimal regularity up to the degenerate boundary, $\partial_0\sO$, that is, $u\in C^{1,1}_s(\underline\sO)$ in the sense of \cite[Definition 2.2]{Daskalopoulos_Feehan_optimalregstatheston}, for a solution $u$ to \eqref{eq:Elliptic_obstacle_problem} is proved by Daskalopoulos and the author in \cite{Daskalopoulos_Feehan_optimalregstatheston}. (Recall that $W^{2,\infty}_{\loc}(\sO) = C^{1,1}(\sO)$.)
\end{rmk}

\begin{rmk}[Extension of the results of this article to the case of unbounded functions]
For simplicity, we have stated the main results of this article for case of a bounded solution, $u$, to the boundary value or obstacle problem, given a bounded source function, $f$, partial Dirichlet boundary data, $g$, and obstacle function, $\psi$, which is bounded above. However, these results may be easily extended to case of $f$, $g$, and $\psi$ having controlled growth (including exponential growth) using the simple device described in \cite[Theorem 2.20]{Feehan_maximumprinciple_v1} to convert a boundary value or obstacle problem with controlled growth for $f$, $g$, or $\psi$ to an equivalent problem for a bounded solution, $\hat u$, with bounded $\hat f$, $\hat g$, and $\hat \psi$ bounded above.
\end{rmk}

\subsection{Application to the elliptic Heston operator}
\label{subsec:Heston}
The elliptic Heston operator \cite{Heston1993}
\begin{equation}
\label{eq:Heston_generator}
Av := -\frac{x_2}{2}\left(v_{x_1x_1} + 2\varrho\sigma v_{x_1x_2} + \sigma^2 v_{x_2x_2}\right) - \left(c_0-q-\frac{x_2}{2}\right)v_{x_1} - \kappa(\theta-x_2)v_{x_2} + c_0v,
\end{equation}
where $v \in C^\infty(\HH)$, provides an example of an operator of the form \eqref{eq:Generator} and which has important applications in mathematical finance, where a solution to the obstacle problem \eqref{eq:Elliptic_obstacle_problem}, \eqref{eq:Elliptic_boundary_condition} can be interpreted as the price of a perpetual American-style put option with payoff function $\psi$ and barrier condition $g$. The coefficients defining $A$ in \eqref{eq:Heston_generator} are constants obeying
\begin{gather}
\label{eq:EllipticHeston}
\sigma \neq 0 \quad\hbox{and}\quad -1< \varrho < 1,
\\
\notag
\kappa > 0 \quad\hbox{and}\quad \theta > 0,
\end{gather}
while $c_0, q \in \RR$, though these constants are typically non-negative in financial applications. The financial and probabilistic interpretations of the preceding coefficients are provided in \cite{Heston1993}. Note that the condition \eqref{eq:EllipticHeston} implies that $A$ in \eqref{eq:Heston_generator} is uniformly but not strictly elliptic on $\HH$ in the sense\footnote{The terminology is not universal.} of \cite[p. 31]{GilbargTrudinger}. Indeed, the matrix
$$
(a^{ij}) = \begin{pmatrix}1 & \varrho\sigma \\ \varrho\sigma & \sigma^2\end{pmatrix}
$$
%By Mathematica
has eigenvalues $(1 + \sigma^2 \pm \sqrt{1 - 2\sigma^2 + 4\varrho^2\sigma^2 + \sigma^4})/2$, which are both bounded below by
\begin{equation}
\label{eq:Defn_Heston_lambda_0}
\lambda_0 = \frac{1}{2}\left(1 + \sigma^2 - \sqrt{1 - 2\sigma^2 + 4\varrho^2\sigma^2 + \sigma^4}\right),
\end{equation}
and $\lambda_0$ is positive when \eqref{eq:EllipticHeston} holds and therefore the strict ellipticity condition \eqref{eq:Strict_ellipticity_domain} is satisfied.

When $\phi(x) = (K-e^{x_1})^+$, for a positive constant $K>0$, then the solution $u$ to \eqref{eq:Elliptic_obstacle_problem}, \eqref{eq:Elliptic_boundary_condition} (with obstacle function $\phi$) can be interpreted as the price of the \emph{perpetual American-style put option} with strike $K$ and asset price $S=e^{x_1}$. Since
$$
\phi_{x_1} = \phi_{x_1x_1} = \begin{cases} -e^{x_1} &\hbox{if } x_1<\ln K, \\ 0 &\hbox{if } x_1 > \ln K, \end{cases}
$$
one can show that the put option payoff satisfies the locally finite concavity condition \eqref{eq:Elliptic_obstaclefunction_locally_finite_concavity_condition} in $\HH$. (The put option payoff, $\Phi(S)=(K-S)^+$, is convex as a function of $S=e^{x_1} \in(0,\infty)$.)

When $\psi(x) = (e^{x_1}-K)^+$, the solution $u$ to \eqref{eq:Elliptic_obstacle_problem}, \eqref{eq:Elliptic_boundary_condition} (with obstacle function $\psi$) can be interpreted as the price of the perpetual American-style \emph{call} option with strike $K$. Since
$$
\psi_{x_1} = \psi_{x_1x_1} = \begin{cases} 0 &\hbox{if } x_1<\ln K, \\ e^{x_1} &\hbox{if } x_1 > \ln K, \end{cases}
$$
one can show that the call option payoff also satisfies the locally finite concavity condition \eqref{eq:Elliptic_obstaclefunction_locally_finite_concavity_condition} in $\HH$. (The call option payoff, $\Psi(S)=(K-S)^+$, is also convex as a function of $S=e^{x_1} \in(0,\infty)$.)

\subsection{Compatibility of mixed boundary conditions where the degenerate and non-degenerate boundary portions touch}
\label{subsec:Corner_point_compatibility}
When $\overline{\partial_0\sO}\cap\overline{\partial_1\sO}$ is non-empty, it appears to be a more challenging problem than one might expect to prove existence of solutions, $u$, in $C^{2+\alpha}_s(\underline\sO)\cap C(\bar\sO)$ or $C^{2+\alpha}_s(\bar\sO)$ to the elliptic equation \eqref{eq:Elliptic_equation} with partial Dirichlet boundary condition \eqref{eq:Elliptic_boundary_condition}. Indeed, complications emerge when one attempts to apply the continuity method to prove existence of solutions $u \in C^{2+\alpha}_s(\bar\sO)$ even in the simplest case, $g=0$ on $\partial_1\sO$. We illustrate the issue when $d=2$ with the aid of the Heston operator $A$ in \eqref{eq:Heston_generator}.

While the reflection principle (across the axis $x=0$) does not hold for the Heston operator $A$ in \eqref{eq:Heston_generator}, it does hold for the simpler model operator,
$$
A_0v := -\frac{x_2}{2}\left(v_{x_1x_1} + \sigma^2 v_{x_2x_2}\right) - \kappa(\theta-x_2)v_{x_2} + c_0v, \quad v\in C^\infty(\HH),
$$
since the $v_{x_1x_2}$ and $v_{x_1}$ terms are absent. If $f(-x_1,x_2)=-f(x_1,x_2)$ for $(x_1,x_2)\in\HH$ (and thus $f(0,\cdot)=0$), one can solve
$$
A_0u_0=f \quad\hbox{on }\sO, \quad u=0 \quad\hbox{on }\partial_1\sO,
$$
for a solution, $u_0$, when the domain, $\sO$, is the quadrant $\RR_+\times\RR_+$.

The continuity method would proceed by showing that, given $f\in C^\alpha_s(\bar\sO)$, the set of $t\in [0,1]$ such that
$$
A_tu = f \quad\hbox{on }\sO, \quad u = g \quad\hbox{on }\partial_1\sO,
$$
has a solution $u\in C^{2+\alpha}_s(\bar\sO)$ is non-empty, open, and closed, where
\begin{align*}
A_tv &:= -\frac{x_2}{2}\left(v_{x_1x_1} + 2t\varrho\sigma v_{x_1x_2} + \sigma^2 v_{x_2x_2}\right) - t\left(c_0-q-\frac{x_2}{2}\right)v_{x_1}
\\
&\qquad - \kappa(\theta-x_2)v_{x_2} + c_0v, \quad v\in C^\infty(\HH),
\end{align*}
is a family of operators, $C^{2+\alpha}_s(\bar\sO) \to C^\alpha_s(\bar\sO)$, connecting $A_0$ to $A$. However, if $u\in C^{2+\alpha}_s(\bar\sO)$ solves $A_0u=f$ on $\sO$, $u=0$ on $\partial_1\sO$, then, letting $x_2\to 0$ in \eqref{eq:Elliptic_equation}, we find
$$
-(c_0-q)u_{x_1}(0,0) = f(0,0),
$$
since $u_{x_2}(0,0)=0$ (because $u(0,\cdot)=0$) and as $u\in C^{2+\alpha}_s(\bar\sO)$ implies $\lim_{(x_1,x_2)\to(0,0)}x_2D^2u=0$. When $t=0$, we see that we can only solve $A_tu=f$ on $\sO$, $u=0$ on $\partial_1\sO$ when $f$ obeys the \emph{compatibility condition}, $f(0,0)=0$, which is not present when $t \neq 0$.

\subsection{Extensions and future work}
\label{subsec:Extensions}
In \cite{Feehan_classical_perron_parabolic}, we describe parabolic analogues of the results discussed in the present article.

We cannot conclude from the analysis in this article that our solutions will obey an a priori global Schauder estimate or be smooth, let alone continuous, up to the domain corner points regardless of the geometry of the domain or regularity of the coefficients of $A$, source $f$, and boundary data, $g$. This appears to us to be an important, albeit overlooked problem, and one to which we plan to return in a subsequent article.

\subsection{Outline and mathematical highlights}
\label{subsec:Guide}
For the convenience of the reader, we provide an outline of the article.

In \S \ref{sec:Maximum_principles}, we review the weak and strong maximum principles in \cite{Feehan_maximumprinciple_v1} and specialize them to the case of an operator $A$ of the form \eqref{eq:Generator} on a subdomain $\sO$ of the upper-half space $\HH\subset\RR^d$. These maximum principles, just as in the case of their classical analogues in \cite{GilbargTrudinger}, play an essential role in our versions of the Perron method.

In \S \ref{sec:Schauder}, we review the construction of the Daskalopoulos-Hamilton-Koch families of H\"older norms and Banach spaces \cite{DaskalHamilton1998}, \cite{Koch} and recall the a priori interior Schauder estimates and the solution to the partial Dirichlet problem on a slab developed with C. Pop in \cite{Feehan_Pop_elliptichestonschauder}.

\begin{figure}
\centering
\begin{picture}(360,180)(0,0)
\put(0,0){\includegraphics[width=360pt,height=180pt,clip=true,trim=0pt 0pt 0pt 270pt]{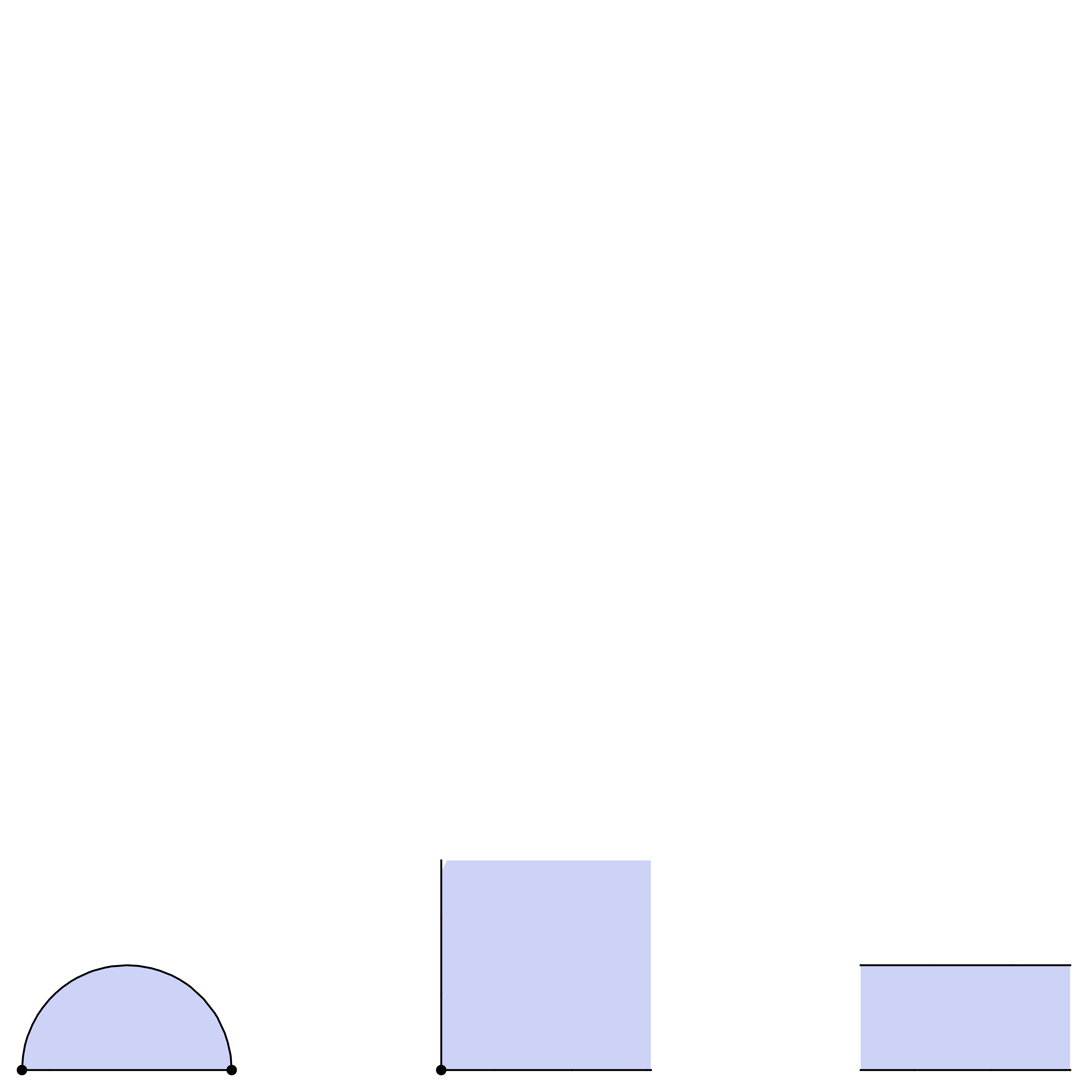}}
\put(-5,-3){$\scriptstyle z=-1$}
\put(16,17){$\scriptstyle |z|<1, \ \Imag z > 0$}
\put(70,-3){$\scriptstyle z=1$}
\put(138,-3){$\scriptstyle \xi=0$}
\put(152,35){$\scriptstyle \Real \xi > 0, \ \Imag \xi > 0$}
\put(294,17){$\scriptstyle 0 < \Imag w < \pi/2$}
\end{picture}
\caption[Conformal equivalence between the half-disk, quadrant, and infinite strip.]{The conformal maps $z\mapsto \xi = (z+1)/(z-1)$ and $\xi\mapsto w = \Log\xi$ identify the half-disk with the quadrant and the quadrant with the infinite horizonal strip. The conformal map $z\mapsto w = \Log((z+1)/(z-1))$ identifies the corner point $z=-1$ with $w=-\infty$, the corner point $z=1$ with $w=\infty$, the interval $\{z\in\CC:|z|<1, \ \Imag z=0\}$ with the line $\{w\in\CC:\Imag w=0\}$, and the semicircle $\{z\in\CC:|z|=1, \ \Imag z>0\}$ with the line $\{w\in\CC:\Imag w=\pi/2\}$.}
\label{fig:conformalmaps}
\end{figure}
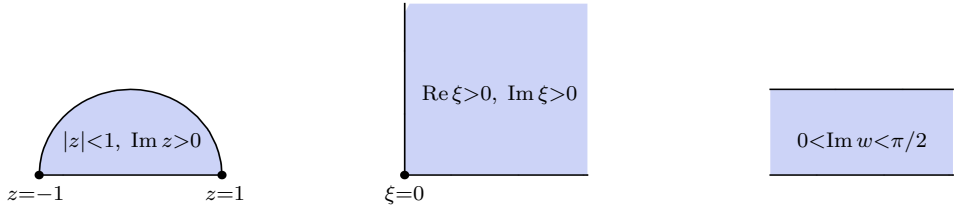

Our goal in \S \ref{sec:Diffeomorphism} is to prove existence of a solution to the partial Dirichlet problem on a half-ball centered along the boundary of the upper half-space (Theorem \ref{thm:Existence_uniqueness_elliptic_Dirichlet_halfball}). We accomplish this by generalizing the fact that a function which is harmonic on a region in the plane remains harmonic after composition with a conformal map and apply this observation when the region is a half-disk. One knows from elementary complex analysis that the half-disk is conformally equivalent to the quadrant and that the quadrant is conformally equivalent to the infinite horizontal strip. See Figure \ref{fig:conformalmaps}.

In \S \ref{subsec:Diffeomorphism_halfball_slab}, we generalize the definition of the conformal map from a half-disk to a horizontal strip in the complex plane $\CC$ to a map from the half-ball in to a slab in $\RR^d$. We explore the properties of this map and, in particular, its effect on the coefficients of the operator $A$ in \eqref{eq:Generator}, taking note of the effect of the map on finite cylinders in the slab illustrated in Figure \ref{fig:conformalmaprectangle}.

In \S \ref{subsec:Existence_dirichlet_boundary_value_problem_halfball}, we exploit the existence of a solution to the partial Dirichlet problem on a slab \cite{Feehan_Pop_elliptichestonschauder} and properties of the map from a half-ball to a slab to prove Theorem \ref{thm:Existence_uniqueness_elliptic_Dirichlet_halfball}. The solution to the partial Dirichlet problem on a half-ball with Dirichlet condition on the spherical portion of the boundary (and no boundary condition on the flat portion) provides the key missing ingredient required for ``local solvability'' of the partial Dirichlet problem in our version of the Perron method.

Section \ref{sec:Perron} contains the heart of our proofs of existence of solutions to the boundary value and obstacle problems described in \S \ref{subsec:Summary} using a modification of the classical Perron method \cite{GilbargTrudinger} for smooth solutions, as distinct from Ishii's version of the Perron method for existence of viscosity solutions \cite{Ishii_1989}. In \S \ref{sec:Perron_boundary_value_problem}, we define continuous subsolutions and supersolutions to the boundary value problem \eqref{eq:Elliptic_equation}, \eqref{eq:Elliptic_boundary_condition} and develop their properties using our version of the strong maximum principle \cite{Feehan_maximumprinciple_v1}. We then adapt the classical Perron method for solving the Dirichlet problem for a harmonic function or a solution to a boundary value problem for a strictly elliptic operator to show that the Perron function,
$$
u(x) := \sup_{w \in \sS_{f,g}^-} w(x), \quad x \in \underline\sO,
$$
is a solution to the boundary value problem \eqref{eq:Elliptic_equation}, \eqref{eq:Elliptic_boundary_condition}, where $\sS_{f,g}^-$ is the set of continuous subsolutions. This yields Theorem \ref{thm:Existence_uniqueness_elliptic_Dirichlet}.

\begin{figure}
 \centering
\includegraphics[scale=0.5]{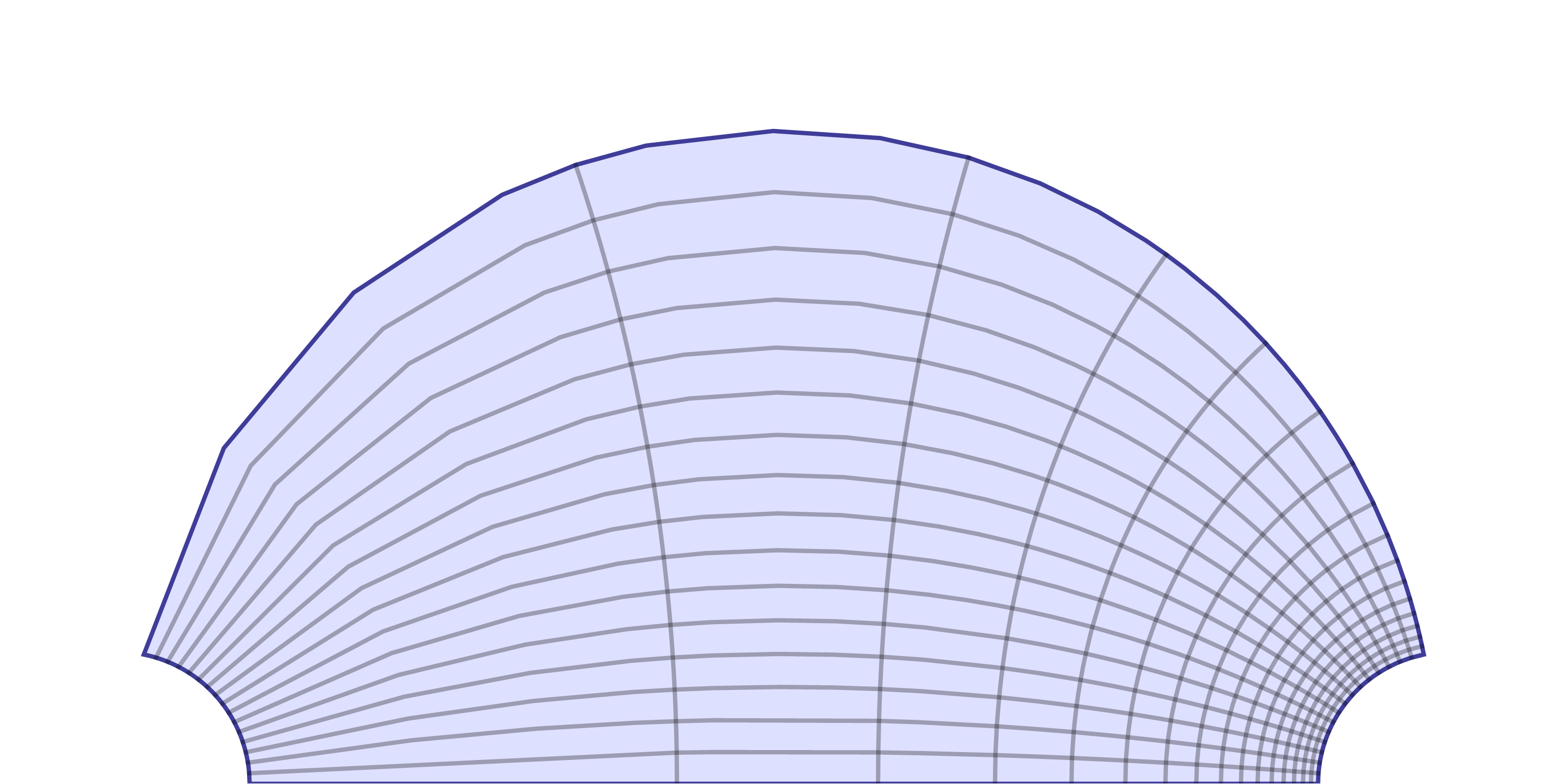}
 \caption[Image of a finite rectangular strip by a conformal map.]{Image of a finite rectangular strip, $\{w\in\CC: -a\leq \Real w\leq a, \ 0 \leq \Imag w\leq \pi/2\}$, of width $2a>0$ in the complex $z$-plane under the conformal map $w\mapsto z = (e^w-1)/(e^w+1)$ from the infinite strip, $\{w\in\CC: 0 \leq \Imag w\leq \pi/2\}$, in the complex $w$-plane to the unit half-disk in the complex $z$-plane without the corner points, $\{z\in\CC: |z| \leq 1, \ \Imag z \geq 0, \ z\neq \pm 1\}$.}
 \label{fig:conformalmaprectangle}
\end{figure}

In \S \ref{sec:Perron_obstacle_problem}, we define continuous supersolutions to the obstacle problem \eqref{eq:Elliptic_obstacle_problem}, \eqref{eq:Elliptic_boundary_condition} and develop their properties via the strong maximum principle. We then prove the essential comparison principle for a solution and continuous supersolution to the obstacle problem (Theorem \ref{thm:Elliptic_obstacle_problem_comparison_principle}). Finally, we adapt the classical Perron method to show that the Perron function,
$$
u(x) := \inf_{w \in \sS_{f,g,\psi}^+} w(x), \quad x \in \underline\sO,
$$
is a solution to the obstacle problem \eqref{eq:Elliptic_obstacle_problem}, \eqref{eq:Elliptic_boundary_condition}, where $\sS_{f,g,\psi}^+$ is the set of continuous supersolutions. Our version of the Perron method for the obstacle problem requires us to solve the boundary value problem \eqref{eq:Elliptic_equation}, \eqref{eq:Elliptic_boundary_condition} for a \emph{degenerate} elliptic operator on half-balls $B^+\Subset\underline\Omega$, where $\Omega = \{x\in\sO:(u-\psi)(x)>0\}$, and the obstacle problem \eqref{eq:Elliptic_obstacle_problem}, \eqref{eq:Elliptic_boundary_condition} for a \emph{strictly} elliptic operator on balls $B\Subset\sO$. See Figure \ref{fig:domain_continuation_region}. This approach requires us to first show that the Perron function, $u$, is at least continuous on $\sO$ (Lemma \ref{lem:Perron_elliptic_problem_solution_continuity}), so we can solve the required obstacle problem on interior balls, $B$, with continuous partial Dirichlet boundary conditions. The Perron method for solving the obstacle problem on $\sO$ is more delicate than its counterpart in the case of the boundary value problem, but ultimately yields Theorems \ref{thm:Existence_uniqueness_elliptic_obstacle} and \ref{thm:Existence_uniqueness_elliptic_obstacle_lipschitz}.

In Appendix \ref{sec:Obstacle_strictly_elliptic_operator}, we develop the results we need for obstacle problems defined by a \emph{strictly} elliptic operator in our application to the Perron of solution to the obstacle problem for a degenerate-elliptic operator \eqref{eq:Generator} and for which there is no complete reference.

In dimension two, the effect of the map in \S \ref{subsec:Diffeomorphism_halfball_slab} from the half-disk to the strip on an operator $A$ of the form \eqref{eq:Generator} can be made explicit and this provides many useful insights, so we describe the calculation in Appendix \ref{sec:Holomorphic_map_model_operator}.

%TODO Add flat and curved boundary labeling
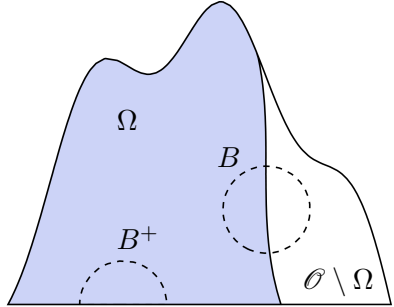
\begin{figure}
 \centering
 \begin{picture}(210,210)(0,0)
 \put(0,0){\includegraphics[scale=0.6]{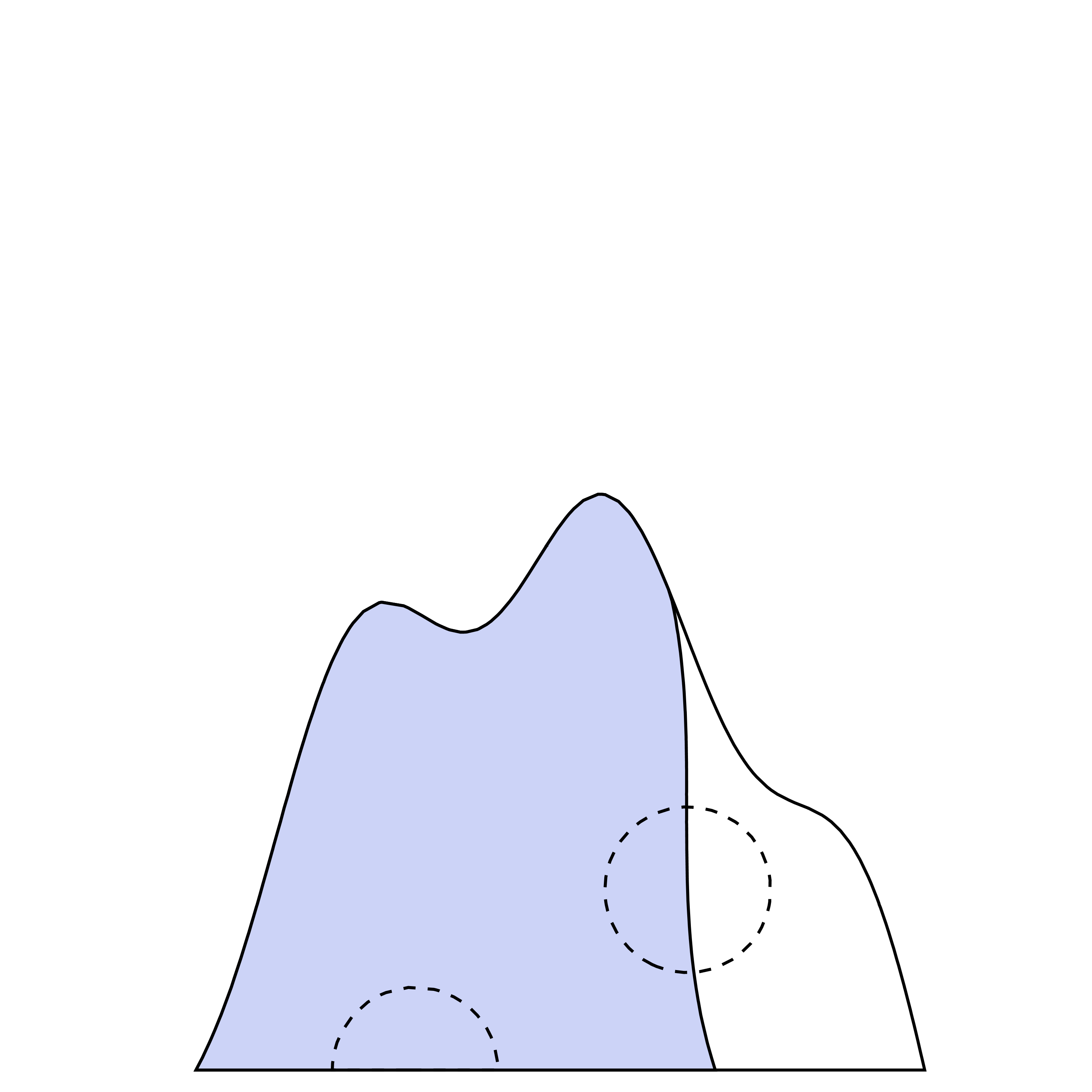}}
 \put(80,70){$\Omega$}
 \put(80,25){$B^+$}
 \put(118,56){$B$}
 \put(150,10){$\sO\less\Omega$}
 \end{picture}
  \caption[Perron method for a solution to the obstacle problem.]{Perron method for a solution to the obstacle problem, $\min\{Au-f, \ u-\psi\}=0$ a.e. on $\sO$, using balls $B\Subset\sO$ and half-balls $B^+\Subset\underline\Omega$ centered on $\partial_0\sO$, where $\Omega = \sO\cap\{u>\psi\}$ and $\sO\less\Omega = \sO\cap\{u=\psi\}$.}
 \label{fig:domain_continuation_region}
\end{figure}

\subsection{Notation and conventions}
\label{subsec:Notation}
In the definition and naming of function spaces, including spaces of continuous functions and H\"older spaces, we follow Adams \cite{Adams_1975} and alert the reader to occasional differences in definitions between \cite{Adams_1975} and standard references such as Gilbarg and Trudinger \cite{GilbargTrudinger} or Krylov \cite{Krylov_LecturesHolder}.

We let $\NN:=\left\{0,1,2,3,\ldots\right\}$ denote the set of non-negative integers. If $X\subset\RR^d$ is a subset, we let $\bar X$ denote its closure with respect to the Euclidean topology and denote $\partial X := \bar X\less X$. For $r>0$ and $x^0\in\RR^d$, we let $B_r(x^0) := \{x\in\RR^d: |x-x^0|<r\}$ denote the open ball with center $x^0$ and radius $r$. We denote $B_r^+(x^0) := B_r(x^0) \cap \HH$ when $x^0\in\partial\HH$. When $x^0$ is the origin, $O\in\RR^d$, we often denote $B_r(x^0)$ and $B_r^+(x^0)$ simply by $B_r$ and $B_r^+$ for brevity and, if $r=1$, we often denote $B_1$ and $B_1^+$ simply by $B$ and $B^+$.

If $V\subset U\subset \RR^d$ are open subsets, we write $V\Subset U$ when $U$ is bounded with closure $\bar U \subset V$. By $\supp\zeta$, for any $\zeta\in C(\RR^d)$, we mean the \emph{closure} in $\RR^d$ of the set of points where $\zeta\neq 0$.

We use $C=C(*,\ldots,*)$ to denote a constant which depends at most on the quantities appearing on the parentheses. In a given context, a constant denoted by $C$ may have different values depending on the same set of arguments and may increase from one inequality to the next.

We denote $x\vee y = \max\{x,y\}$ and $x\wedge y = \min\{x,y\}$, for any $x,y\in\RR$.

\subsection{Acknowledgments} I am very grateful to Ioannis Karatzas and the Department of Mathematics at Columbia University for their generous support during the preparation of this article and also to Emanuele Spadaro, J\"urgen Jost, and the Max Planck Institut f\"ur Mathematik in der Naturwissenschaft, Leipzig, for their support during visits in 2012 and 2013. I warmly thank Panagiota Daskalopoulos and Camelia Pop for many helpful conversations concerning obstacle problems and degenerate partial differential equations and for their comments on an earlier draft of this article. Preliminary versions of the results described in this article were presented at the Research in Options Conference 2012, in B\'uzios, Brazil, at the kind invitation of Jorge Zubelli; I greatly appreciated the comments and questions of Emmanuel Gobet following that presentation. I am very grateful to Guy Barles and Hitoshi Ishii for explanations of technical points arising in their articles \cite{Barles_1993} and \cite{Ishii_1991} and to Eduardo Teixeira for a helpful discussion of regularity theory for obstacle problems.

\section{Weak and strong maximum principles for degenerate-elliptic operators}
\label{sec:Maximum_principles}
In this section, we review the weak and strong maximum principles in \cite{Feehan_maximumprinciple_v1} and specialize them to the case of an operator $A$ of the form \eqref{eq:Generator} on a subdomain $\sO$ of the upper-half space $\HH\subset\RR^d$. We also review the a priori estimates implied by the weak maximum principle provided in \cite{Feehan_maximumprinciple_v1} and \cite[Appendix A]{Feehan_Pop_elliptichestonschauder}.

\begin{defn}[Non-negative definite characteristic form]
\label{defn:Degeneracy_locus}
\cite[Definition 2.1]{Feehan_maximumprinciple_v1}
%COMMENT We do not require a(x) to be defined for x in \overline{\partial_1\sO}
Suppose $a=(a^{ij})$ is an $\RR^{d\times d}$-valued function on $\underline\sO$ which defines a \emph{non-negative definite characteristic form}, that is
\begin{equation}
\label{eq:Non_degeneracy_near_boundary}
\langle a\xi, \xi\rangle \geq 0 \quad\hbox{on } \underline\sO, \quad\forall \xi\in\RR^d.
\end{equation}
\end{defn}

In this article, we normally assume that $\partial\sO\less\partial_1\sO$ is non-empty, though $\partial_0\sO$ will of course be empty if $\partial\sO\less\partial_1\sO$ has empty interior.

\begin{defn}[Linear, second-order, partial differential operator with non-negative definite characteristic form]
\label{defn:Partial_diff_operator_nonneg_char_form}
\cite[Definition 2.2]{Feehan_maximumprinciple_v1}
We call $A$ in \eqref{eq:Generator} a linear, second-order, partial differential operator with \emph{non-negative definite characteristic form on $\underline\sO$} if the coefficient matrix, $a$, obeys \eqref{eq:Non_degeneracy_near_boundary}.
\end{defn}

\begin{defn}[Second-order boundary condition]
\label{defn:Boundary_regularity_property}
\cite[Definition 2.3]{Feehan_maximumprinciple_v1}
We say that $u\in C^2(\sO)\cap C^1(\underline\sO)$ obeys a \emph{second-order boundary condition} along $\partial_0\sO$ if
\begin{align}
\label{eq:WMP_C2s_regularity}
x_dD^2u &\in C(\underline\sO;\RR^{d\times d}),
\\
\label{eq:WMP_xdsecond_order_derivatives_zero_boundary}
x_dD^2u &= 0 \quad\hbox{on }\partial_0\sO.
\end{align}
\end{defn}

Recall that $\underline\sO = \sO\cup\partial_0\sO$, in contrast with $\bar\sO=\sO\cup\partial\sO$. To simplify the statements of our definitions and results, we shall make the universal

\begin{assump}[Partial differential operator with non-negative characteristic form]
\label{assump:C2_operator_coefficients}
Given a domain, $\sO\subseteqq\RR^d$, the coefficients of the linear, second-order partial differential operator $A:C^2(\sO)\to C(\sO)$ as in \eqref{eq:Generator} obey \eqref{eq:Non_degeneracy_near_boundary} and
\begin{equation}
\label{eq:amatrix_locally_bounded_upto_lower_boundary}
a(x) \hbox{ is a locally bounded function of } x\in \underline\sO.
\end{equation}
\end{assump}

We can now state the key\footnote{In \cite[Definitions 2.3 \& 2.8 and Theorems 5.1 \& 5.3]{Feehan_maximumprinciple_v1}, we assumed that $u$ belongs to $C_{\loc}(\bar\sO)$ rather than just $C(\sO\cup\partial_1\sO)$ as done here, but this stronger assumption was unnecessary and was relaxed in subsequent versions.}

\begin{defn}[Weak maximum principle property for $C^2$ subsolutions]
\label{defn:Weak_max_principle_property}
\cite[Definition 2.8]{Feehan_maximumprinciple_v1}
We say that an operator $A$ in \eqref{eq:Generator} obeys the \emph{weak maximum principle property on $\underline\sO$} if whenever $u\in C^2(\sO)\cap C^1(\underline\sO)\cap C(\sO\cup\partial_1\sO)$ obeys \eqref{eq:WMP_C2s_regularity}, \eqref{eq:WMP_xdsecond_order_derivatives_zero_boundary} and
\begin{align}
\label{eq:C2_subsolution}
Au &\leq 0 \quad \hbox{on }\sO,
\\
\label{eq:u_leq_zero_boundary}
u &\leq 0 \quad \hbox{on }\partial_1\sO,
\end{align}
%COMMENT Since we no longer assume u \in C_{\loc}(\bar\sO), we always need a bounded above assumption unlike \cite[Definition 2.8]{Feehan_maximumprinciple_v1}.
and $\sup_\sO u < \infty$, then
$$
u \leq 0 \quad\hbox{on }\sO.
$$
\end{defn}

Our preference for abstracting the ``weak maximum principle property'' is well illustrated by the following applications, as discussed in \cite{Feehan_maximumprinciple_v1}.

\begin{defn}[Classical solution, subsolution, and supersolution to a boundary value problem with partial Dirichlet data]
\label{defn:BVPClassicalSolution}
\cite[Definition 2.14]{Feehan_maximumprinciple_v1}
Given functions $f\in C(\underline\sO)$ and $g\in C(\partial_1\sO)$, we call $u\in C^2(\sO)\cap C^1(\underline\sO)\cap C(\sO\cup\partial_1\sO)$ a \emph{subsolution} to a boundary value problem for an operator $A$ in \eqref{eq:Generator} with Dirichlet boundary condition along $\partial_1\sO$, if $u$ obeys \eqref{eq:WMP_C2s_regularity}, \eqref{eq:WMP_xdsecond_order_derivatives_zero_boundary}, and
\begin{align}
\label{eq:BVPC2}
Au &\leq f \quad \hbox{on }\sO,
\\
\label{eq:BVPBoundaryC2}
u&\leq g \quad \hbox{on }\partial_1\sO.
\end{align}
We call $u$ a \emph{supersolution} if $-u$ is a subsolution and we call $u$ a \emph{solution} if it is both a subsolution and supersolution. \qed
\end{defn}

\begin{prop}[Weak maximum principle estimates for classical subsolutions, supersolutions, and solutions]
\label{prop:Elliptic_C2_weak_max_principle_apriori_estimates}
\cite[Proposition 2.19]{Feehan_maximumprinciple_v1}
Let $\sO\subseteqq\HH$ be a possibly unbounded domain. Let $A$ in \eqref{eq:Generator} have the weak maximum principle property on $\underline\sO$ in the sense of Definition \ref{defn:Weak_max_principle_property} and assume $c$ obeys \eqref{eq:Nonnegative_c}. Let $f\in C(\underline\sO)$ and $g\in C(\partial_1\sO)$. Suppose $u\in C^2(\sO)\cap C^1(\underline\sO)\cap C(\sO\cup\partial_1\sO)$ obeys \eqref{eq:WMP_C2s_regularity}, \eqref{eq:WMP_xdsecond_order_derivatives_zero_boundary}.
\begin{enumerate}
\item\label{item:C2_subsolution_f_leq_zero} If $f\leq 0$ on $\sO$ and $u$ is a subsolution for $f$ and $g$, and $\sup_\sO u < \infty$, then
$$
u\leq 0 \vee \sup_{\partial_1\sO}g \quad\hbox{on }\sO.
$$
\item\label{item:C2_subsolution_f_arb_sign} If $f$ has arbitrary sign and $u$ is a subsolution for $f$ and $g$, and $\sup_\sO u < \infty$, but, in addition, $c$ obeys \eqref{eq:Positive_lower_bound_c_domain}, so $c\geq c_0$ on $\underline\sO$ for a positive constant $c_0$, then
$$
u\leq 0 \vee \frac{1}{c_0}\sup_\sO f \vee \sup_{\partial_1\sO}g \quad\hbox{on }\sO.
$$
\item\label{item:C2_supersolution_f_geq_zero} If $f\geq 0$ on $\sO$ and $u$ is a supersolution for $f$ and $g$, and $\inf_\sO u > -\infty$, then
$$
u\geq 0 \wedge \inf_{\partial_1\sO}g \quad\hbox{on }\sO.
$$
\item\label{item:C2_supersolution_f_arb_sign} If $f$ has arbitrary sign, $u$ is a supersolution for $f$ and $g$, and $\inf_\sO u > -\infty$, and $c$ obeys \eqref{eq:Positive_lower_bound_c_domain}, then
$$
u\geq 0 \wedge \frac{1}{c_0}\inf_ \sO f \wedge \inf_{\partial_1\sO}g \quad\hbox{on }\sO.
$$
\item\label{item:C2_solution_f_zero} If $f=0$ on $\sO$ and $u$ is a solution for $f$ and $g$, and $\sup_\sO |u| < \infty$, then
$$
\|u\|_{C(\bar\sO)} \leq \|g\|_{C(\overline{\partial_1\sO})}.
$$
\item\label{item:C2_solution_f_arb_sign} If $f$ has arbitrary sign, $u$ is a solution for $f$ and $g$, and $\sup_\sO |u| < \infty$, and $c$ obeys \eqref{eq:Positive_lower_bound_c_domain}, then
$$
\|u\|_{C(\bar\sO)} \leq \frac{1}{c_0}\|f\|_{C(\bar\sO)}\vee\|g\|_{C(\overline{\partial_1\sO})}.
$$
\end{enumerate}
The terms $\sup_{\partial_1\sO}g$, and $\inf_{\partial_1\sO}g$, and $\|g\|_{C(\overline{\partial_1\sO})}$ in the preceding inequalities are omitted when $\partial_1\sO$ is empty.
\end{prop}

\begin{thm}[Weak maximum principle on bounded domains]
\label{thm:Elliptic_weak_maximum_principle_bounded_domain}
\cite[Theorem 5.1]{Feehan_maximumprinciple_v1}
Let $\sO\subset\HH$ be a \emph{bounded} domain and $A$ be as in \eqref{eq:Generator}. Require that the coefficients of $A$ be defined everywhere on $\bar\sO$, obey \eqref{eq:Non_degeneracy_near_boundary} and
\begin{align}
\label{eq:CoeffBD}
b^d &\geq 0 \quad\hbox{on }\partial_0\sO,
\\
\label{eq:NonnegativecDomainCupSigma}
c &\geq 0 \quad\hbox{on }\underline\sO.
\end{align}
Assume further that at least \emph{one} of the following holds,
\begin{equation}
\label{eq:NonDegeneracyInteriorRatio}
\begin{cases}
c > 0 \hbox{ on }\underline\sO,\hbox{ or}
\\
b^d > 0 \hbox{ on }\partial_0\sO \hbox{ and } \displaystyle{\inf_{\sO}\frac{b^d}{x_da^{dd}}} > -\infty.
\end{cases}
\end{equation}
Suppose that $u \in C^2(\sO)\cap C^1(\underline\sO)\cap C(\sO\cup\partial_1\sO)$ obeys \eqref{eq:WMP_C2s_regularity}, \eqref{eq:WMP_xdsecond_order_derivatives_zero_boundary}. If $u$ obeys \eqref{eq:C2_subsolution}, that is, $Au \leq 0$ on $\sO$, and is bounded above, $\sup_\sO u < \infty$, then
\begin{equation}
\label{eq:C2MaxPrincipleBound}
\sup_\sO u \leq 0\vee\sup_{\partial_1\sO} u,
\end{equation}
and, if $c=0$ on $\underline\sO$, then
\begin{equation}
\label{eq:C2MaxPrincipleBoundcZero}
\sup_\sO u = \sup_{\partial_1\sO} u.
\end{equation}
Moreover, $A$ has the \emph{weak maximum principle property on $\underline\sO$} in the sense of Definition \ref{defn:Weak_max_principle_property}.
\end{thm}

\begin{rmk}[Continuity and boundedness conditions]
\label{rmk:Continuity_boundedness_conditions}
The hypotheses in Theorem \ref{thm:Elliptic_weak_maximum_principle_bounded_domain} assert that $u$ belongs to $C(\underline\sO\cup\partial_1\sO) = C(\sO\cup\partial_0\sO\cup\partial_1\sO)$, but not necessarily $C_{\loc}(\bar\sO)$, so $u$ may be discontinuous at the ``corner points'', $\overline{\partial_0\sO}\cup\overline{\partial_1\sO}$. Recall that the classical weak maximum principle \cite[Theorem 3.1]{GilbargTrudinger} for a strictly elliptic operator $L$ on a bounded domain $\sO\subset\RR^d$ does not require that $u\in C^2(\sO)$ be continuous on $\bar\sO$ and asserts more generally that, when $c=0$ on $\sO$ and $Lu \leq 0$ on $\sO$ (see \cite[Equation (3.6)]{GilbargTrudinger}),
$$
\sup_\sO u = \limsup_{x\to\partial\sO}u(x)
$$
and, when $c\geq 0$ (and replacing the assertion of \cite[Theorem 3.1]{GilbargTrudinger} by the preceding generalization in the proof of \cite[Corollary 3.2]{GilbargTrudinger})
$$
\sup_\sO u \leq \limsup_{x\to\partial\sO}u^+(x).
$$
We may write
$$
\limsup_{x\to\partial\sO}u(x) = \sup_{x_0\in\partial\sO}\left(\limsup_{x\to x_0}u(x)\right) = \sup_{\partial\sO}u^*,
$$
where $u^*:\bar\sO\to\RR$ is the upper semicontinuous envelope of $u$ on $\sO$, where $u^*$ necessarily attains its maximum on $\bar\sO$ when $\sO$ is bounded, and $\sup_\sO u^*<\infty$ since $\sup_\sO u^* = \sup_\sO u$ and $\sup_\sO u<\infty$ by hypothesis.

Hence, the conclusion in Theorem \ref{thm:Elliptic_weak_maximum_principle_bounded_domain} can be strengthened to $u\leq 0$ on $\underline\sO\cup\partial_1\sO$ but not necessarily on $\bar\sO$; the conclusion could be replaced by the equivalent statement, $u^*\leq 0$ on $\bar\sO$.
\qed
\end{rmk}

Next, we consider the case of bounded $C^2$ functions on \emph{unbounded} domains and recall the following version of \cite[Theorem 5.3]{Feehan_maximumprinciple_v1}.

\begin{lem}[Extension of the weak maximum principle property for bounded functions to unbounded domains]
\label{lem:Elliptic_weak_maximum_principle_extension_unbounded_domain}
\cite[Theorem 5.3]{Feehan_maximumprinciple_v1}
Let $\sO\subseteqq\HH$ be a possibly unbounded domain. Suppose $A$ in \eqref{eq:Generator} has the weak maximum principle property on \emph{bounded} subdomains of $\underline\sO$ in the sense of Definition \ref{defn:Weak_max_principle_property} and assume, in addition, that its coefficients obey \eqref{eq:Positive_lower_bound_c_domain} and \eqref{eq:Quadratic_growth}. Then $A$ has the weak maximum principle property on $\underline\sO$ in the sense of Definition \ref{defn:Weak_max_principle_property}.
\end{lem}

Combining Theorem \ref{thm:Elliptic_weak_maximum_principle_bounded_domain} and Lemma \ref{lem:Elliptic_weak_maximum_principle_extension_unbounded_domain} yields

\begin{thm}[Weak maximum principle for bounded functions on unbounded domains]
\label{thm:Elliptic_weak_maximum_principle_unbounded_domain}
\cite[Theorem 5.3]{Feehan_maximumprinciple_v1}
Let $\sO\subseteqq\HH$ be a possibly unbounded domain. Assume that the coefficients of $A$ in \eqref{eq:Generator} obey the hypotheses of Theorem \ref{thm:Elliptic_weak_maximum_principle_bounded_domain}, except that the condition \eqref{eq:NonDegeneracyInteriorRatio} on $a, b$ or $c$ is replaced by the condition \eqref{eq:Positive_lower_bound_c_domain} on $c$ for some positive constant, $c_0$, and, in addition, we require that the coefficients $a, b$ obey the growth condition \eqref{eq:Quadratic_growth}. Suppose that $u \in C^2(\sO)\cap C^1(\underline\sO)\cap C(\sO\cup\partial_1\sO)$ obeys \eqref{eq:WMP_C2s_regularity} and \eqref{eq:WMP_xdsecond_order_derivatives_zero_boundary}. If $\sup_\sO u < \infty$ and $u\leq 0$ on $\partial_1\sO$ (when non-empty), then
$$
\sup_\sO u \leq 0\vee\frac{1}{c_0}\sup_\sO Au.
$$
Moreover, $A$ has the \emph{weak maximum principle property on $\underline\sO$} in the sense of Definition \ref{defn:Weak_max_principle_property}.
\end{thm}

It will be useful to have a version of the weak maximum principle and the corresponding estimate for operators which include those of the form $A$ in \eqref{eq:Generator} when $\inf_\sO c = 0$ and the domain, $\sO$, is unbounded though of finite height.  Notice that when the coefficient $c$ obeys \eqref{eq:Nonnegative_c} but not \eqref{eq:Positive_lower_bound_c_domain}, so does not have a uniform positive lower bound on $\sO$, the weak maximum principle, Theorem \ref{thm:Elliptic_weak_maximum_principle_unbounded_domain}, does not immediately apply when $\sO$ is unbounded. For this situation, we recall the following special cases of \cite[Lemma A.1 \& Corollary A.2]{Feehan_Pop_elliptichestonschauder}.

\begin{lem}[Weak maximum principle for bounded functions on a domain of finite height]
\label{lem:Elliptic_weak_maximum_principle_finite_height_domain}
\cite[Lemma A.1]{Feehan_Pop_elliptichestonschauder}
Let $\sO\subseteqq\HH$ be a domain with finite height, $\height(\sO)\leq \nu$, for some positive constant $\nu$. Let $A$ be as in \eqref{eq:Generator} and require that its coefficients, $a, b, c$, obey \eqref{eq:Nonnegative_c}, \eqref{eq:Finite_upper_bound_add_domain} (with $\Lambda$ replaced by $\Lambda/\nu$) and \eqref{eq:Positive_lower_bound_bd_domain} for some positive constants $b_0$ and $\Lambda$, and \eqref{eq:Non_degeneracy_near_boundary}. Then $A$ has the \emph{weak maximum principle property on $\underline\sO$} in the sense of Definition \ref{defn:Weak_max_principle_property}.
\end{lem}

\begin{cor}[Weak maximum principle estimate for bounded functions on a domain of finite height]
\label{cor:Elliptic_weak_maximum_principle_finite_height_domain}
\cite[Corollary A.2]{Feehan_Pop_elliptichestonschauder}
Assume the hypotheses of Lemma \ref{lem:Elliptic_weak_maximum_principle_finite_height_domain}. Let $f \in C(\underline\sO)$, and $g \in C(\partial_1 \sO)$. If $u \in C^2(\sO)\cap C^1(\underline\sO)\cap C_b(\sO\cup\partial_1\sO)$ obeys \eqref{eq:WMP_C2s_regularity} and \eqref{eq:WMP_xdsecond_order_derivatives_zero_boundary} and is a solution to the boundary value problem \eqref{eq:Elliptic_equation}, \eqref{eq:Elliptic_boundary_condition}, then
\begin{equation}
\label{eq:Elliptic_weak_maximum_principle_finite_height_domain}
\|u\|_{C(\bar \sO)}
\leq e^{b_0\nu/2\Lambda}\left(\frac{4\Lambda}{b_0^2}\|f\|_{C(\bar \sO)} \vee \|g\|_{C(\overline{\partial_1 \sO})}\right).
\end{equation}
\end{cor}

\begin{rmk}[On the proofs of the weak maximum principle for bounded functions on domains of finite height]
\label{rmk:Elliptic_weak_maximum_principle_finite_height_domain}
It will also be useful to recall the proofs of Lemma \ref{lem:Elliptic_weak_maximum_principle_finite_height_domain} and Corollary \ref{cor:Elliptic_weak_maximum_principle_finite_height_domain}. We define
$$
v(x) := e^{\sigma x_d} u(x) \quad\hbox{and}\quad \tilde f(x) := e^{\sigma x_d} f(x), \quad\forall\, x \in \sO, \quad\hbox{and}\quad \tilde g(x) := e^{\sigma x_d} g(x), \quad\forall\, x \in \partial_1\sO,
$$
and observe that $v$ solves the boundary value problem
$$
\tilde Av = \tilde f \quad\hbox{on } \sO, \quad v = \tilde g  \quad\hbox{on } \partial_1\sO,
$$
where $\tilde A$ has coefficients obeying $\tilde a^{ij} := a^{ij}$ on $\bar\sO$, and $\tilde b^i := b^i-2\sigma x_da^{id}$ with $\tilde b^d = b^d\geq b_0$ on $\partial_0\sO$ and $\tilde c := c+\sigma b^d - \sigma^2 x_da^{dd} \geq \tilde c_0 := b_0^2/4\Lambda$ on $\underline\sO$, provided $\sigma \geq b_0/2\Lambda$, while
$$
\sup_\sO|\tilde f| \leq e^{\sigma\nu}\sup_\sO|f| \quad\hbox{and}\quad \sup_{\partial_1\sO}|\tilde g| \leq e^{\sigma\nu}\sup_{\partial_1\sO}|g|.
$$
The conclusions in Lemma \ref{lem:Elliptic_weak_maximum_principle_finite_height_domain} and Corollary \ref{cor:Elliptic_weak_maximum_principle_finite_height_domain} now follow by applying to Theorem \ref{thm:Elliptic_weak_maximum_principle_unbounded_domain} and Proposition \ref{prop:Elliptic_C2_weak_max_principle_apriori_estimates} \eqref{item:C2_solution_f_arb_sign} to the boundary value problem for $v$ with $c_0$ replaced by $b_0^2/4\Lambda$.

Note also that if the coefficients of $A$ had the properties $c\geq c_0$ on $\sO$ for some positive constant $c_0$ but $\inf_{\partial_0\sO}b^d = 0$, then we could apply the preceding change of dependent variable in reverse to produce an equivalent operator $\tilde A$ with $\tilde b^d \geq \tilde b_0$ on $\partial_0\sO$ for some positive constant $b_0$. In summary, if $\sO\subset\HH$ is a domain with finite height and at least one of $\inf_{\partial_0\sO}b^d$ or $\inf_\sO c$ is positive, then we may assume without loss of generality (by virtue of the change of dependent variable) that both are positive.\qed
\end{rmk}

\begin{rmk}[On the hypotheses of Theorem \ref{thm:Elliptic_weak_maximum_principle_bounded_domain} and Lemma \ref{lem:Elliptic_weak_maximum_principle_finite_height_domain}.]
The condition \eqref{eq:Positive_bd_boundary} ensures that $b^d>0$ on $\partial_0\sO$ while the conditions \eqref{eq:Finite_upper_bound_add_domain} (with $\Lambda$ replaced by $\Lambda/\nu$) and \eqref{eq:Positive_lower_bound_bd_domain} ensure that
$$
\inf_\sO \frac{b^d}{x_d a^{dd}} \geq \frac{b_0}{x_d a^{dd}} \geq \frac{b_0}{\nu (\Lambda/\nu)} = \frac{b_0}{\Lambda} > 0,
$$
and hence that \eqref{eq:NonDegeneracyInteriorRatio} (second alternative) holds.
\end{rmk}

We introduce an analogue of Definition \ref{defn:Weak_max_principle_property}.

\begin{defn}[Strong maximum principle property]
\label{defn:Strong_max_principle_property}
Let $\sO\subset\HH$ be a domain. We say that an operator $A$ in \eqref{eq:Generator} has the \emph{strong maximum principle property on $\underline\sO$} if whenever $u\in C^2(\sO)\cap C^1(\underline\sO)\cap C(\sO\cup\partial_1\sO)$ obeys \eqref{eq:WMP_C2s_regularity} and \eqref{eq:WMP_xdsecond_order_derivatives_zero_boundary} and $Au \leq 0$ on $\sO$, then one of the following holds.
\begin{enumerate}
\item If $c=0$ on $\underline\sO$ and $u$ attains a maximum in $\underline\sO$, then $u$ is constant on $\sO$.
\item If $c\geq 0$ on $\underline\sO$ and $u$ attains a non-negative maximum in $\underline\sO$, then $u$ is constant on $\sO$.
\end{enumerate}
\end{defn}

\begin{lem}[Strong implies weak maximum principle property on bounded domains]
\label{lem:Elliptic_strong_implies_weak_maximum_principle_property_bounded_domain}
Let $\sO\subset\HH$ be a \emph{bounded} domain. Suppose $A$ in \eqref{eq:Generator} has the strong maximum principle property on $\underline\sO$ in the sense of Definition \ref{defn:Strong_max_principle_property}. Then $A$ has the weak maximum principle property on $\underline\sO$ in the sense of Definition \ref{defn:Weak_max_principle_property}.
\end{lem}

\begin{proof}
Suppose that $u$ obeys the conditions (though not necessarily the conclusion) of Definition \ref{defn:Weak_max_principle_property}. Since $\bar\sO$ is compact, the upper semicontinuous envelope, $u^*:\bar\sO\to\RR$, of $u$ attains its maximum, $u^*(x_0) = \sup_\sO u^* = \sup_\sO u$, at some point $x_0\in \bar\sO$. If $x_0\in\overline{\partial_1\sO}$, we are done, since $u^*(x_0)=0$ for such points. If $x_0\in \underline\sO$, then $u^*(x_0)=u(x_0)$ (since $u$ is continuous on $\underline\sO$) and thus $u$ is constant on $\sO$ by Definition \ref{defn:Strong_max_principle_property} and hence constant on $\sO\cup\partial_1\sO$. But $u\leq 0$ on $\partial_1\sO$ by Definition \ref{defn:Weak_max_principle_property} and so we again obtain $u\leq 0$ on $\sO$.
\end{proof}

By combining Lemmas \ref{lem:Elliptic_weak_maximum_principle_extension_unbounded_domain}, \ref{lem:Elliptic_weak_maximum_principle_finite_height_domain}, and \ref{lem:Elliptic_strong_implies_weak_maximum_principle_property_bounded_domain}, we obtain

\begin{cor}[Strong implies weak maximum principle property on unbounded domains]
\label{cor:Elliptic_strong_implies_weak_maximum_principle_property_unbounded_domain}
Let $\sO\subseteqq\HH$ be a possibly unbounded domain. Suppose $A$ in \eqref{eq:Generator} has the strong maximum principle property on $\underline\sO$ in the sense of Definition \ref{defn:Strong_max_principle_property} and assume, in addition, that its coefficients obey
\begin{enumerate}
\item when $\height(\sO)=\infty$, then \eqref{eq:Positive_lower_bound_c_domain}, or
\item when $\height(\sO)<\infty$, then \eqref{eq:Positive_lower_bound_c_domain} or both \eqref{eq:Finite_upper_bound_add_domain} and \eqref{eq:Positive_lower_bound_bd_domain},
\end{enumerate}
and, in addition, \eqref{eq:Quadratic_growth} when $\sO$ is unbounded. Then $A$ has the weak maximum principle property on $\underline\sO$ in the sense of Definition \ref{defn:Weak_max_principle_property}.
\end{cor}

Finally, we recall the following special case of \cite[Theorem 4.6]{Feehan_maximumprinciple_v1}.

\begin{thm}[Strong maximum principle]
\label{thm:Elliptic_strong_maximum_principle}
\cite[Theorem 4.6]{Feehan_maximumprinciple_v1}
Suppose that $\sO\subset\HH$ is a domain\footnote{Recall that by a ``domain'' in $\RR^d$, we always mean a \emph{connected}, open subset.}. Require that the operator $A$ in \eqref{eq:Generator} obey \eqref{eq:Nonnegative_c}, \eqref{eq:amatrix_locally_bounded_upto_lower_boundary}, and
\begin{enumerate}
\item For each $\xi\in\RR^d\less\{0\}$ and $B\Subset\sO$, one has
$$
\langle a\xi, \xi\rangle > 0 \hbox{ on }B, \quad \inf_B\frac{\langle b, \xi\rangle}{\langle a\xi, \xi\rangle} > -\infty, \quad \sup_B\frac{c}{\langle a\xi, \xi\rangle} < \infty;
$$
\item For each $x_0 \in \partial_0\sO$, there is a ball $B\Subset\underline\sO$ such that $\partial B\cap\partial_0\sO=\{x_0\}$ and\footnote{For example, the coefficient $b^d$ has this property if it is continuous and positive along $\partial_0\sO$.}
$$
\inf_{B\cup\{x_0\}} b^d > 0 \quad\hbox{and}\quad \sup_B c < \infty.
$$
\end{enumerate}
Then $A$ has the strong maximum principle property on $\underline\sO$ in the sense of Definition \ref{defn:Strong_max_principle_property}.
\end{thm}

\section{Schauder theory for degenerate-elliptic operators}
\label{sec:Schauder}
In \S \ref{subsec:Holder}, we review the construction of the Daskalopoulos-Hamilton-Koch families of H\"older norms and Banach spaces \cite{DaskalHamilton1998}, \cite{Koch}. In \S \ref{subsec:Apriori_Schauder_Dirichlet_solution_slab}, we recall the a priori interior Schauder estimates and the solution to the partial Dirichlet problem on a slab given in \cite{Feehan_Pop_elliptichestonschauder}.

\subsection{Daskalopoulos-Hamilton-Koch H\"older spaces}
\label{subsec:Holder}
We review the construction of the
%COMMENT Hyphenation
Dask-alopoulos-Hamilton-Koch families of H\"older norms and Banach spaces \cite[p. 901 \& 902]{DaskalHamilton1998}, \cite[Definition 4.5.4]{Koch}. For standard H\"older spaces, we follow the notation of Gilbarg and Trudinger \cite[\S 4.1]{GilbargTrudinger} for the Banach spaces but then use the Banach space to label the corresponding norms following the pattern we describe here.

Given a constant $\alpha\in(0,1)$ and function $u$ on an open subset $U\subset\HH$, the Daskalopoulos-Hamilton-Koch H\"older seminorm, $[u]_{C^\alpha_s(\bar U)}$, is defined by mimicking the definition of the standard H\"older seminorm, $[u]_{C^\alpha(\bar U)}$ in \cite[p. 52]{GilbargTrudinger}, except that the usual Euclidean distance between points, $|x-y| = (\sum_{i=1}^d|x_i-y_i|^2)^{1/2}$ for $s, y\in\bar\HH$, corresponding to the standard Riemannian metric $ds^2 = \sum_{i=1}^d dx_i^2$ on $\HH$ is replaced by the distance function, $s(x,y)$ for $x,y\in\HH$, corresponding to the \emph{cycloidal} Riemannian metric\footnote{The Riemannian metric is called cycloidal in \cite{DaskalHamilton1998} since its geodesics, when $d=2$, are the standard cycloidal curve, $(x_1(t), x_2(t))=(t-\sin t, 1-\cos t)$ for $t\in\RR$, curves obtained from the standard cycloid by translations $(x_1,x_2)\mapsto (x_1+b,x_2)$, $b\in\RR$, or dilations $(x_1,x_2)\mapsto (cx_1, x_2)$, $c\in\RR_+$, or are vertical lines, $x_1=a$, $a\in\RR$ \cite[Proposition I.2.1]{DaskalHamilton1998}.}
on $\HH$ defined by \cite[pp. 901--905]{DaskalHamilton1998}, \cite[p. 62]{Koch},
\begin{equation}
\label{eq:Cycloidal_metric}
ds^2 = \frac{1}{x_d}\sum_{i=1}^d dx_i^2,
\end{equation}
where $x = (x_1,\ldots,x_d)$. When the coefficient matrix $(a^{ij}(x))$, for $x\in\HH$, in the definition \eqref{eq:Generator} of the operator $A$ is strictly and uniformly elliptic on $\HH$ in the sense of \cite[p. 31]{GilbargTrudinger}, then the cycloidal metric \eqref{eq:Cycloidal_metric} is equivalent to that defined by the inverse $(g_{ij}(x))$ of the coefficient matrix $(g^{ij}(x)) = (x_da^{ij}(x))$, for $x\in\HH$. It is convenient to choose a more explicit \emph{cycloidal distance function} on $\bar\HH$ which is equivalent to the distance function defined by the Riemannian metric \eqref{eq:Cycloidal_metric}, such as that in \cite[p. 901]{DaskalHamilton1998},
$$
s(x,y) = \frac{\sum_{i=1}^d|x_i-y_i|}{\sqrt{x_d} + \sqrt{y_d} + \sqrt{|x-y|}},\quad\forall\, x, y \in \bar\HH,
$$
or the equivalent choice on $\bar\HH$ given in \cite[p. 11 \& \S 4.3]{Koch} and which we adopt in this article,
\begin{equation}
\label{eq:Cycloidal_distance}
s(x, y) := \frac{|x-y|}{\sqrt{x_d+y_d+|x-y|}},\quad\forall\, x, y \in \bar\HH.
\end{equation}
Following \cite[\S 1.26]{Adams_1975}, for a domain $U\subset\HH$, we let $C(U)$ denote the vector space of continuous functions on $U$ and let $C(\bar U)$ denote the Banach space of functions in $C(U)$ which are bounded and uniformly continuous on $U$, and thus have unique bounded, continuous extensions to $\bar U$, with norm
$$
\|u\|_{C(\bar U)} := \sup_{U}|u|.
$$
Noting that $U$ may be \emph{unbounded}, we let $C_{\loc}(\bar U)$ denote the linear subspace of functions $u\in C(U)$ such that $u\in C(\bar V)$ for every precompact open subset $V\Subset \bar U$. Following \cite[\S 5.2]{Adams_1975}, we let $C_b(U) \subset C(U)$ denote the Banach space of bounded, continuous functions on $U$ with norm again denoted by $\|u\|_{C(\bar U)}$, without implying that the function $u$ is uniformly continuous on $U$.

We note the distinction between $C(\bar\HH)$ (respectively, $C(\bar\RR^d)$), the Banach space of functions which are uniformly continuous and bounded on $\bar\HH$ (respectively, $\RR^d$), and $C(\underline\HH)=C_{\loc}(\bar\HH)$ (respectively, $C(\RR^d)$), the vector space of functions which are continuous on $\bar\HH$ (respectively, $\RR^d$).

By $C(U\cup T)$, for $T\subseteqq\partial U$, we shall always mean $C_{\loc}(U\cup T)$ and thus we may have $C(\bar U)\subsetneqq C(U\cup T)=C_{\loc}(\bar U)$ when $T=\partial U$, unless $U$ is bounded in which case, as usual, $C(\bar U) = C(U\cup \partial U)=C_{\loc}(\bar U)$.

Daskalopoulos and Hamilton provide the

\begin{defn}[$C^\alpha_s$ norm and Banach space]
\label{defn:Calphas}
\cite[p. 901]{DaskalHamilton1998}
Given $\alpha \in (0,1)$ and an open subset $U\subset\HH$, we say that $u\in C^\alpha_s(\bar U)$ if $u\in C(\bar U)$ and
$$
\|u\|_{C^\alpha_s(\bar U)} < \infty,
$$
where
\begin{equation}
\label{eq:CalphasNorm}
\|u\|_{C^\alpha_s(\bar U)} := [u]_{C^\alpha_s(\bar U)} + \|u\|_{C(\bar U)},
\end{equation}
and
\begin{equation}
\label{eq:CalphasSeminorm}
[u]_{C^\alpha_s(\bar U)} := \sup_{\stackrel{x, y\in U}{x\neq y}}\frac{|u(x)-u(y)|}{s(x, y)^\alpha}.
\end{equation}
We say that $u\in C^\alpha_s(\underline{U})$ if $u\in C^\alpha_s(\bar V)$ for all precompact open subsets $V\Subset \underline{U}$, recalling that $\underline{U} := U\cup\partial_0 U$. We let $C^\alpha_{s,\loc}(\bar U)$ denote the linear subspace of functions $u \in C^\alpha_s(U)$ such that $u\in C^\alpha_s(\bar V)$ for every precompact open subset $V\Subset \bar U$.
\end{defn}

It is known that $C^\alpha_s(\bar U)$ is a Banach space \cite[\S I.1]{DaskalHamilton1998} with respect to the norm \eqref{eq:CalphasNorm}. We recall the definition of the higher-order $C^{k,\alpha}_s$ H\"older norms and Banach spaces.

\begin{defn}[$C^{k,\alpha}_s$ norms and Banach spaces]
\label{defn:DHspaces}
\cite[p. 902]{DaskalHamilton1998}
Given an integer $k\geq 0$, $\alpha \in (0,1)$, and an open subset $U\subset\HH$, we say that $u\in C^{k,\alpha}_s(\bar U)$ if
$u \in C^k(\bar U)$ and
$$
\|u\|_{C^{k,\alpha}_s(\bar U)} < \infty,
$$
where
\begin{equation}
\label{eq:CkalphasNorm}
\|u\|_{C^{k,\alpha}_s(\bar U)} := \sum_{|\beta|\leq k}\|D^\beta u\|_{C^\alpha_s(\bar U)},
\end{equation}
where $\beta := (\beta_1,\ldots,\beta_d) \in \NN^d$, and $|\beta| := \beta_1 + \cdots + \beta_d$, and
$$
D^\beta u := \frac{\partial^{|\beta|}u}{\partial_{x_1}^{\beta_1}\cdots \partial_{x_d}^{\beta_d}}.
$$
When $k=0$, we denote $C^{0,\alpha}_s(\bar U) = C^\alpha_s(\bar U)$.
\end{defn}

For example, when $k=1$ one has
$$
\|u\|_{C^{1,\alpha}_s(\bar U)} = \|u\|_{C^\alpha_s(\bar U)} + \|Du\|_{C^\alpha_s(\bar U)}.
$$
Finally, we recall the definition of the higher-order $C^{k,2+\alpha}_s$ H\"older norms and Banach spaces.

%TODO Update per parabolic version
%TODO Could say u in C^{k+2}(U) in first part to shorten
\begin{defn}[$C^{k,2+\alpha}_s$ norms and Banach spaces]
\label{defn:DH2spaces}
\cite[pp. 901--902]{DaskalHamilton1998}
Given an integer $k\geq 0$, a constant $\alpha \in (0,1)$, and an open subset $U\subset\HH$, we say that $u \in C^{k,2+\alpha}_s(\bar U)$ if $u\in C^{k+1,\alpha}_s(\bar U)$, the derivatives, $D^\beta u$, $\beta\in\NN^d$ with $|\beta| = k+2$, of order $k+2$ are continuous on $U$, and the functions, $x_dD^\beta u$, $\beta\in\NN^d$ with $|\beta|= k+2$, extend continuously up to the boundary, $\partial U$, and those extensions belong to $C^\alpha_s(\bar U)$. We define
\begin{equation}
\label{eq:Ck+2alphasNorm}
\|u\|_{C^{k, 2+\alpha}_s(\bar U)} :=  \|u\|_{C^{k+1,\alpha}_s(\bar U)} + \sum_{|\beta|=k+2}\|x_dD^\beta u\|_{C^{\alpha}_s(\bar U)}.
\end{equation}
We say that\footnote{In \cite[pp. 901--902]{DaskalHamilton1998}, when defining the spaces $C^{k,\alpha}_s(\sA)$ and $C^{k,2+\alpha}_s(\sA)$, it is assumed that $\sA$ is a compact subset of the \emph{closed} upper half-space, $\bar\HH$.} $u\in C^{k,2+\alpha}_s(\underline{U})$ if $u\in C^{k,2+\alpha}_s(\bar V)$ for all precompact open subsets $V\Subset \underline{U}$. When $k=0$, we denote $C^{0,2+\alpha}_s(\bar U) = C^{2+\alpha}_s(\bar U)$.
\end{defn}

For example, when $k=0$ one has
$$
\|u\|_{C^{2+\alpha}_s(\bar U)} = \|u\|_{C^\alpha_s(\bar U)} + \|Du\|_{C^\alpha_s(\bar U)} + \|x_dD^2u\|_{C^\alpha_s(\bar U)}.
$$
A parabolic version of following result is included in \cite[Proposition I.12.1]{DaskalHamilton1998} when $d=2$ and proved in \cite{Feehan_Pop_mimickingdegen_pde} when $d\geq 2$ for parabolic weighted H\"older spaces. We restate the result here for the elliptic weighted H\"older spaces used in this article.

\begin{lem}[Boundary properties of functions in weighted H{\"o}lder spaces]
\label{lem:Properties_second_order_derivatives}
\cite[Lemma 3.1]{Feehan_Pop_mimickingdegen_pde}
If $u \in C^{2+\alpha}_s(\underline{\HH})$ then, for all $x_0 \in\partial \HH$,
\begin{equation}
\label{eq:PropSecondOrderDeriv}
\lim_{\HH \ni x \rightarrow x_0} x_d D^2u(x) = 0.
\end{equation}
\end{lem}

For an open subset $U\Subset\HH$, the standard and Daskalopoulos-Hamilton-Koch H\"older spaces coincide by Definitions \ref{defn:Calphas}, \ref{defn:DHspaces}, and \ref{defn:DH2spaces}, and thus, for example, $C^\alpha_s(\bar U) = C^\alpha(\bar U)$ and $C^{2+\alpha}_s(\bar U) = C^{2,\alpha}(\bar U)$. More refined relationships between these H\"older spaces arise from the observation that, by \eqref{eq:Cycloidal_distance},
\begin{equation}
\label{eq:CycloidLessEuclidDistance}
s(x,y) \leq |x - y|^{1/2}, \quad \forall\, x, \, y \in \bar\HH.
\end{equation}
The reverse inequality, on a domain $V\subset\HH$ with $\height(V)=\nu$ and $\diam(V)=D$, takes the form
\begin{equation}
\label{eq:SimpleEuclidLessCycloidDistance}
|x-y| \leq \sqrt{2\nu+D}\,s(x,y), \quad \forall\, x, \,y \in\bar V.
\end{equation}
Therefore, as noted in \cite[Remark 2.4]{Daskalopoulos_Feehan_optimalregstatheston}, we have
$$
\|u\|_{C^{\alpha/2}(U)} \leq \|u\|_{C^\alpha_s(U)} \quad\hbox{and}\quad \|u\|_{C^\alpha_s(V)} \leq C\|u\|_{C^\alpha(V)},
$$
for any open subset $U\subseteqq\HH$ and any open subset $V\subset\HH$ with $\diam(V)\leq D$ and $\height(V)\leq\nu$, with constant $C=C(D,\alpha,\nu)$, and thus
\begin{equation}
\label{eq:DH_and_standard_Holder_relationships}
C^\alpha_s(\bar U) \subset C^{\alpha/2} (\bar U) \quad\hbox{and}\quad C^\alpha(\bar V) \subset  C^\alpha_s(\bar V).
\end{equation}
With the preceding background in place, we can now proceed to recall the key Schauder regularity and existence results we shall need from \cite{Feehan_Pop_elliptichestonschauder}.

\subsection{A priori Schauder estimates and solution to the partial Dirichlet problem on a slab}
\label{subsec:Apriori_Schauder_Dirichlet_solution_slab}
If $d\geq 2$ and $A$ is as in \eqref{eq:Generator} with coefficients $a,b,c$, and $U\subseteqq \HH$ is any subset, we denote
\begin{equation}
\label{eq:CoefficientNorms}
\|a\|_{C^{k, \alpha}_s(\bar U)} := \sum_{i,j=1}^d\|a^{ij}\|_{C^{k, \alpha}_s(\bar U)}
\quad\hbox{and}\quad
\|b\|_{C^{k, \alpha}_s(\bar U)} := \sum_{i=1}^d\|b^i\|_{C^{k, \alpha}_s(\bar U)}.
\end{equation}
From \cite{Feehan_Pop_elliptichestonschauder}, we recall the

\begin{thm}[A priori interior Schauder estimate]
\label{thm:Elliptic_apriori_Schauder_interior_domain}
\cite[Theorem 1.1]{Feehan_Pop_elliptichestonschauder}
For any $\alpha\in(0,1)$ and positive constants $b_0$, $d_0$, $\lambda_0$, $\Lambda$, $\nu$, there is a positive constant, $C=C(b_0,d,d_0,\alpha,\lambda_0,\Lambda,\nu)$, such that the following holds. Suppose $\height(\sO)\leq \nu$ and the coefficients $a,b,c$ of $A$ in \eqref{eq:Generator} belong to $C^{\alpha}_s(\underline\sO)$ and obey \eqref{eq:Strict_ellipticity_domain}, \eqref{eq:Positive_lower_bound_bd_boundary}, and
\begin{equation}
\label{eq:Coeff_holder_continuity_estimate_domain}
\|a\|_{C^{\alpha}_s(\bar\sO)} + \|b\|_{C^{\alpha}_s(\bar\sO)} + \|c\|_{C^{\alpha}_s(\bar\sO)} \leq \Lambda.
\end{equation}
If $u \in C^{2+\alpha}_s(\underline\sO)$ and $\sO'\subset\sO$ is a subdomain such that $\dist(\partial_1\sO',\partial_1\sO)\geq d_0$, then
\begin{equation}
\label{eq:Elliptic_apriori_Schauder_interior_domain}
\|u\|_{C^{2+\alpha}_s(\bar\sO')} \leq C\left(\|Au\|_{C^{\alpha}_s(\bar\sO)} + \|u\|_{C(\bar\sO)}\right).
\end{equation}
\end{thm}

It is considerably more difficult to prove a global a priori estimate for a solution, $u\in C^{k, 2+\alpha}_s(\bar\sO)$, when the intersection $\overline{\partial_0\sO}\cap\overline{\partial_1\sO}$ is non-empty. However, the global estimate in Theorem \ref{thm:Global_elliptic_Schauder_estimate_slab} has useful applications when $\partial_1\sO$ does not meet $\partial_0\sO$. For a constant $\nu>0$, we call
\begin{equation}
\label{eq:Slab}
S^d_\nu = \RR^{d-1}\times(0,\nu),
\end{equation}
a \emph{slab} of height $\nu$ (in the terminology of \cite[\S 3.3]{GilbargTrudinger}). We often simply write $S$ for $S^d_\nu$ when there is no ambiguity and note that $\partial_0 S =\RR^{d-1}\times\{0\}$ and $\partial_1 S =\RR^{d-1}\times\{\nu\}$.

\begin{thm}[A priori global Schauder estimate on a slab]
\label{thm:Global_elliptic_Schauder_estimate_slab}
\cite[Corollary 1.3]{Feehan_Pop_elliptichestonschauder}
For any $\alpha\in(0,1)$ and positive constants $b_0$, $\lambda_0$, $\Lambda$, $\nu$, there is a positive constant, $C=C(b_0,d,\alpha,\lambda_0,\Lambda,\nu)$, such that the following holds. Suppose the coefficients of $A$ in \eqref{eq:Generator} belong to $C^\alpha_s(\bar S)$, where $S=\RR^{d-1}\times(0, \nu)$ as in
\eqref{eq:Slab}, and obey
\begin{gather}
\label{eq:Elliptic_coeff_Holder_continuity_slab}
\|a\|_{C^\alpha_s(\bar S)} + \|b\|_{C^\alpha_s(\bar S)} + \|c\|_{C^{\alpha}_s(\bar S)} \leq \Lambda,
\\
\label{eq:Strict_ellipticity_slab}
\langle a\xi, \xi\rangle \geq \lambda_0 |\xi|^2\quad\hbox{on } \bar S,\quad \forall\,\xi \in \RR^d,
\\
\label{eq:Elliptic_coeff_b_d_slab}
b^d \geq b_0 \quad \hbox{on } \partial_0 S.
\end{gather}
If $u\in C^{2+\alpha}_s(\bar S)$ and $u = 0$ on $\partial_1 S$, then
\begin{equation}
\label{eq:Global_ellipticSchauder_estimate_slab}
\| u\|_{C^{2+\alpha}_s(\bar S)} \leq C\left(\|A u\|_{C^\alpha_s(\bar S)} + \|u\|_{C(\bar S)}\right),
\end{equation}
and, when $c\geq 0$ on $S$,
\begin{equation}
\label{eq:Global_ellipticSchauder_estimate_slab_nonnegc}
\| u\|_{C^{2+\alpha}_s(\bar S)} \leq C \|A u\|_{C^\alpha_s(\bar S)},
\end{equation}
\end{thm}

\begin{thm}[Existence and uniqueness of a smooth solution to an boundary value problem on a slab]
\label{thm:Elliptic_existence_uniqueness_slab}
\cite[Theorem 1.6]{Feehan_Pop_elliptichestonschauder}
Let $\alpha\in(0,1)$ and let $\nu>0$ and $S=\RR^{d-1}\times(0, \nu)$ be as in \eqref{eq:Slab}. Let $A$ be an operator as in \eqref{eq:Generator}. If $f$ and the coefficients of $A$ in \eqref{eq:Generator} belong to $C^\alpha_s(\bar S)$ and obey \eqref{eq:Strict_ellipticity_slab} and \eqref{eq:Elliptic_coeff_b_d_slab} and $g \in C^{2+\alpha}_s(\bar S)$, then there is a unique solution, $u \in C^{2+\alpha}_s(\bar S)$, to the boundary value problem,
\begin{align}
\label{eq:Elliptic_equation_slab}
Au &= f \quad \hbox{on } S,
\\
\label{eq:Equation_slab_boundarycondition}
u &= g \quad \hbox{on } \partial_1 S.
\end{align}
\end{thm}

\begin{rmk}[Existence and uniqueness of a solution to the elliptic Heston equation on a slab or half-plane]
When $d=2$ and $A$ is the Heston operator as in \eqref{eq:Heston_generator} on a strip $S=\RR\times(0, \nu)$ as in \eqref{eq:Slab} and $f\in C^\alpha_s(\bar S)$, one may prove existence and uniqueness of a solution $u\in C^{2+\alpha}_s(\bar S)$ to \eqref{eq:Elliptic_equation_slab}, \eqref{eq:Equation_slab_boundarycondition} by a slight modification of the proof of \cite[Theorem B.3 \& Corollary B.4]{Feehan_Pop_elliptichestonschauder}, where the coefficients $a,b,c$ of $A$ in \eqref{eq:Generator} are assumed to be constant. Indeed, the proof of \cite[Theorem B.3 \& Corollary B.4]{Feehan_Pop_elliptichestonschauder} proceeds by taking the Fourier transform of equation \eqref{eq:Elliptic_equation_slab} with respect to $x$ and then solving the resulting Kummer ordinary differential equation in $y\in(0,\nu)$ in terms of confluent hypergeometric functions \cite[\S 13]{AbramStegun}.
The same procedure (Fourier transform with respect to $x$ and solving the Kummer equation in $y$) yields a solution $u\in C^{2+\alpha}_s(\underline \HH)$ when $f\in C^\alpha_s(\underline\HH)$ and, for example, is bounded on $\HH$.
\end{rmk}

\section{Partial Dirichlet problems on a half-ball and a slab}
\label{sec:Diffeomorphism}
Our goal in this section is to prove existence of a solution to the partial Dirichlet problem on a half-ball centered along the boundary of the upper half-space (Theorem \ref{thm:Existence_uniqueness_elliptic_Dirichlet_halfball}). In \S \ref{subsec:Diffeomorphism_halfball_slab}, we generalize the definition of the conformal map from a half-disk to a horizontal strip in the complex plane to higher dimensions and explore its properties and effect on the coefficients of the operator $A$ in \eqref{eq:Generator}. In \S \ref{subsec:Existence_dirichlet_boundary_value_problem_halfball}, we exploit the existence of a solution to the partial Dirichlet problem on a slab (Theorem \ref{thm:Elliptic_existence_uniqueness_slab}) and properties of the map from a half-ball to a slab to prove Theorem \ref{thm:Existence_uniqueness_elliptic_Dirichlet_halfball}.

\subsection{A diffeomorphism from the unit half-ball to the slab}
\label{subsec:Diffeomorphism_halfball_slab}
We begin with the case $d=2$. A conformal map from the half-disk without the corner points, $\{z\in\CC: |z|\leq 1, \ \Imag z \geq 0, \ z\neq \pm 1\}$, onto the quadrant, $\{\xi\in\CC: \Real\xi\geq 0, \ \Imag\xi \geq 0, \ \xi\neq 0\}$ without the corner, can be achieved with
\begin{equation}
\label{eq:Defn_z_to_xi_complex_plane}
\CC\less\{1\} \ni z \mapsto \xi := \frac{1+z}{1-z} \in \CC.
\end{equation}
A conformal map from the quadrant, $\{\xi\in\CC: \Real\xi\geq 0, \ \Imag\xi \geq 0, \ \xi\neq 0\}$ without the corner, onto the infinite horizonal strip, $\{w\in\CC:0\leq \Imag w\leq \pi/2\}$, can be achieved with
\begin{equation}
\label{eq:Defn_xi_to_w_complex_plane}
\CC\less\{\xi\in\CC:\Real\xi \leq 0, \ \Imag \xi = 0\} \ni \xi \mapsto w :=\Log\xi \in \CC,
\end{equation}
where the principal value of the complex logarithm is defined as usual with
$$
\Log \xi := \ln |\xi| + i\Arg \xi,
$$
and $\Arg \xi \in (-\pi, \pi]$. Hence, noting that $\xi = (1-|z|^2+2\Imag z)/|1-z|^2$, the composition,
\begin{equation}
\label{eq:Defn_z_to_w_complex_plane}
\CC\less\{z\in\CC: |z| \geq 1, \ \Imag z = 0\} \ni z \mapsto w = \Log\left(\frac{1+z}{1-z}\right) \in \CC,
\end{equation}
maps the
\begin{enumerate}
\item open disk, $\{z\in\CC: |z|<1, \ \Imag z>0\}$, onto the open strip $0 < \Imag w < \pi/2$;
\item point $z=1$ in the closed unit half-disk onto the point at infinity, $w=\infty$;
\item semicircle, $\{z\in\CC: |z|=1, \ \Imag z>0\}$, onto the line, $\Imag w = \pi/2$;
\item unit interval, $\{z\in\CC: |z|<1, \ \Imag z=0\}$, onto the line, $\Imag w = 0$;
\item point $z=-1$ in the closed unit half-disk onto the point at infinity, $w=-\infty$.
\end{enumerate}
See Figure \ref{fig:conformalmaps} for an illustration of these maps.

The relationship between a finite-width rectangle in the strip of height $\pi/2$ in the $w$-plane and corresponding region in the unit half-disk in the $z$-plane defined by the map $w \mapsto z$ will be especially important in our analysis. For a fixed $k\in\RR$, the line segment, $w = k + i\theta$, $\theta \in (0, \pi/2)$, in the $w$-plane is mapped from the quarter-circle, $\xi = e^ke^{i\theta}$, $\theta \in (0, \pi/2)$, in the $\xi$-plane. Moreover, for a fixed $r\in\RR_+$, the quarter-circle, $\xi = re^{i\theta}$, $\theta \in (0, \pi/2)$, in the $\xi$-plane is mapped from the curve in the $z$-plane given by
$$
z = \frac{re^{i\theta}-1}{re^{i\theta}+1} = \frac{r^2-1 + i2r\sin\theta}{(r\cos\theta+1)^2 + r^2\sin^2\theta},
\quad \theta \in (0, \pi/2),
$$
as we see by solving the equation $(1+z)/(1-z) = re^{i\theta}$ for $z$. See Figure \ref{fig:conformalmaprectangle} for an illustration of the effect of this map.

We now generalize the holomorphic map defined in \eqref{eq:Defn_z_to_w_complex_plane}, identifying the half-disk with a strip in dimension $d=2$, to a map identifying the half-ball with a slab in dimension $d\geq 2$. First observe that, by analogy with our definition in \eqref{eq:Defn_z_to_xi_complex_plane},
$$
\xi = \frac{1+z}{1-z} = \frac{(1+z)(1-\bar z)}{|1-z|^2} = \frac{(1-|z|^2) + 2\Imag z}{|1-z|^2},
$$
we now set
\begin{equation}
\label{eq:Defn_xi_dgeq2_explicit}
\xi(x) := \frac{(1-|x|^2)e_1 + 2(x-\langle x,e_1 \rangle e_1)}{|e_1-x|^2}, \quad x\in B^+,
\end{equation}
where $B^+$ is the unit half-ball,
$$
B^+ := \{x\in \RR^d: |x|<1, \ x_d > 0\} \subset\RR^{d-1}\times\RR_+,
$$
and observe that the definition \eqref{eq:Defn_xi_dgeq2_explicit} of $\xi(x)$ gives a diffeomorphism,
\begin{equation}
\label{eq:Defn_xi_dgeq2}
\xi: B^+ \cong \RR_+ \times \RR^{d-2} \times \RR_+, \quad x \mapsto \xi(x).
\end{equation}
Indeed, if $x=e_1+\vec h\in B^+$ for $\vec h\in\RR^d$ with $h=|\vec h|$ and we write $x=(x_1,x'')$, then as $B^+ \ni x\to e_1$, so $h\downarrow 0$, while $x''\neq 0$ we see that
$$
\xi(x) = \frac{h^2 e_1 + 2x''}{h^2} = e_1 + \frac{2x''}{h^2}, \quad x'' \in \RR^{d-2} \times \RR_+.
$$
The transformation \eqref{eq:Holomorphic_map_subquadrant_onto_substrip} extends continuously to map the
\begin{enumerate}
\item point $e_1$ in $\bar B^+$ onto a point at infinity labeled as $(e_1,\infty)\in\{e_1\}\times ((\RR^{d-2} \times \RR_+) \cup\{\infty\})$;
\item points $x$ in the open $(d-1)$-dimensional hemisphere, $\{x\in\bar B^+: |x|=1, \ x_d>0\}$, onto points $\xi(x) \in \{0\}\times \RR^{d-2} \times \RR_+$;
\item points $x$ in the punctured, open $(d-1)$-dimensional unit ball, $\{x\in\bar B^+: |x|<1, \ x_d=0, \ x\neq e_1\}$,
onto points $\xi(x) \in \RR_+\times \RR^{d-2} \times \{0\}$;
\item points $x$ in the punctured $(d-2)$-dimensional unit sphere, $\{x\in\bar B^+: |x|=1, \ x_d=0, \ x\neq e_1\}$, onto points $\xi(x) \in \{0\}\times\RR^{d-2}\times\{0\}$.
\end{enumerate}
We now set
\begin{equation}
\label{eq:Defn_wk}
w_k:=\xi_k, \quad \hbox{for } k=2,\ldots,d-1,
\end{equation}
and, identifying
$$
\RR_+\times\RR_+ = \{(\xi_1, \xi_d)\in\RR^2: \xi_1>0, \ \xi_d > 0\}
$$
with an open quadrant in the complex plane,
$$
Q := \{\xi_1 + i\xi_d: \xi_1>0, \ \xi_d > 0\} \subset \CC,
$$
we set
\begin{equation}
\label{eq:Defn_w1_plus_iwd}
w_1 + iw_d := \ln|\xi_1+i\xi_d| + i \Arg(\xi_1+i\xi_d),
\end{equation}
and thus,
$$
w_1 + iw_d = \frac{1}{2}\ln\left(\xi_1^2 + \xi_d^2\right) + i\arctan\left(\frac{\xi_d}{\xi_1}\right).
$$
As in \eqref{eq:Slab}, we denote the open two-dimensional strip of height $\pi/2$ by
$$
S^2_{\pi/2} = \{w_1+iw_d: w_1\in\RR, \ 0<w_d<\pi/2\} = \RR\times (0,\pi/2) \subset \CC,
$$
with closure, $\bar S^2_{\pi/2} = \RR\times [0,\pi/2]$. The transformation \eqref{eq:Defn_w1_plus_iwd} defines a holomorphic map,
\begin{equation}
\label{eq:Holomorphic_map_subquadrant_onto_substrip}
Q \cong S^2_{\pi/2}, \quad \xi_1 + i\xi_d\mapsto w_1+iw_d,
\end{equation}
which extends continuously to map the
\begin{enumerate}
\item point $(\xi_1,\xi_d)=(0,0)$ onto the interval at infinity, $\{-\infty\}\times (0,\pi/2)$;
\item open quarter-circle, $\{(\xi_1,\xi_d)\in\RR^2: \xi_1^2+\xi_d^2=1, \ \xi_1>0, \ \xi_d>0\}$, onto the open interval, $\{0\}\times (0,\pi/2)$;
\item open half-line, $\{(\xi_1,\xi_d)\in\RR^2: \xi_1>0, \ \xi_d=0\}$, onto the line, $\RR \times \{0\}$;
\item open half-line, $\{(\xi_1,\xi_d)\in\RR^2: \xi_1=0, \ \xi_d>0\}$, onto the line, $\RR \times \{\pi/2\}$;
\item quarter-circle at infinity, $\{(\xi_1,\xi_d)\in\RR^2: \xi_1^2+\xi_d^2=\infty, \ \xi_1>0, \ \xi_d>0\}$ onto the interval at infinity, $\{\infty\}\times (0,\pi/2)$.
\end{enumerate}
As in \eqref{eq:Slab}, we denote the $d$-dimensional open slab of height $\pi/2$ by
$$
S \equiv S^d_{\pi/2} = \{w\in\RR^d: 0<w_d<\pi/2\} = \RR^{d-1}\times (0,\pi/2) \subset\RR^{d-1}\times\RR_+,
$$
with closure, $\bar S = \RR^{d-1}\times [0,\pi/2]$. The composite transformation induced by \eqref{eq:Defn_xi_dgeq2}, \eqref{eq:Defn_wk}, and \eqref{eq:Defn_w1_plus_iwd},
\begin{equation}
\label{eq:Defn_Phi}
\Phi: B^+ \cong S, \quad x \mapsto w(x),
\end{equation}
maps the $d$-dimensional open unit half-ball, $B^+$, onto the $d$-dimensional open slab, $S$, of height $\pi/2$ and extends continuously to map the
\begin{enumerate}
\item point $e_1\in \bar B^+$ onto a point at infinity labeled as $\{\infty\}\times\RR^{d-2}\times(0,\pi/2)$;
\item the open $(d-1)$-dimensional hemisphere, $\{x\in\bar B^+: |x|=1, \ x_d>0\}$, onto the $(d-1)$-dimensional hyperplane, $\RR^{d-1} \times \{\pi/2\}$;
\item the open $(d-1)$-dimensional unit ball, $\{x\in\bar B^+: |x|<1, \ x_d=0\}$, onto the $(d-1)$-dimensional hyperplane, $\RR^{d-1} \times \{0\}$;
\item the punctured $(d-2)$-dimensional unit sphere, $\{x\in\bar B^+: |x|=1, \ x_d=0, \ x\neq e_1\}$, onto a point at infinity labeled as $\{-\infty\}\times\RR^{d-2}\times(0,\pi/2)$.
\end{enumerate}
The expression \eqref{eq:Defn_xi_dgeq2_explicit} for $\xi(x)$ yields
$$
\xi_1 = \frac{1-|x|^2}{|e_1-x|^2}, \quad \xi_j = \frac{2x_j}{|e_1-x|^2}, \quad 2\leq j\leq d,
$$
an so \eqref{eq:Defn_xi_dgeq2_explicit}, \eqref{eq:Defn_wk}, and \eqref{eq:Defn_w1_plus_iwd} give an expression for the components of $w=\Phi(x)$ in \eqref{eq:Defn_Phi},
\begin{equation}
\label{eq:Defn_Phi_explicit}
\begin{aligned}
w_1 &= \frac{1}{2}\ln\left(\frac{(1-|x|^2)^2 + 4x_d^2}{|e_1-x|^4}\right)
= \frac{1}{2}\ln\left(\frac{1 + (\sum_{j=1}^dx_j^2)^2 + 2x_d^2 - 2\sum_{j=1}^{d-1}x_j^2}{((1-x_1)^2 + \sum_{j=2}^dx_j^2)^2}\right),
\\
w_d &= \arctan\left(\frac{2x_d}{1-|x|^2}\right) = \arctan\left(\frac{2x_d}{1-\sum_{j=1}^dx_j^2}\right),
\\
w_j &= \frac{2x_j}{|e_1-x|^2} = \frac{2x_j}{(1-x_1)^2 + \sum_{i=2}^dx_i^2}, \quad 2\leq j\leq d-1.
\end{aligned}
\end{equation}
We can partially solve for $x_d\in(0,1)$ in terms of $w\in\RR^{d-1}\times(0,\pi/2)$ using the expression for $w_d$ in \eqref{eq:Defn_Phi_explicit} to give
$$
x_d = \frac{1}{2}\left(1-|x|^2\right)\tan w_d.
$$
Note that $w_d\downarrow 0$ if and only if $\tan w_d \downarrow 0$ and thus, when $|x|<1$, the preceding expression shows that $w_d\downarrow 0$ if and only if $x_d\downarrow 0$. Moreover, when $|x|<1$, we see that $w_d=0$ if and only if $x_d=0$.

On the other hand, observe that $w_d\uparrow \pi/2$ if and only if $\tan w_d\uparrow\infty$. But
$$
\tan w_d = \frac{2x_d}{1-|x|^2},
$$
and hence, when $x_d>0$, the preceding expression shows that $w_d\uparrow \pi/2$ if and only if $|x|\uparrow 1$. Moreover, when $x_d>0$, we see that $w_d=\pi/2$ if and only if $|x|=1$.

We now record the properties of the map from the half-ball to the slab, and consequently its regularity-preserving properties for functions defined on the half-ball or slab. The proof is clear by inspection\footnote{When $d=2$, the effect of the map $\Phi^{-1}:S\to B^+$ on finite-width rectangles within the strip can be visualized in Figure \ref{fig:conformalmaprectangle}. For our application in this article, it suffices to know that $\Phi$ is a $C^3$-diffeomorphism, a fact which may also be verified by direct, if tedious, calculation.} of the definition of the map and the preceding discussion.

\begin{lem}[Properties of the map from the half-ball to the slab]
\label{lem:Properties_map_halfball_slab}
The map $\Phi$ in \eqref{eq:Defn_Phi_explicit} is a $C^\infty$-diffeomorphism from the open unit half-ball, $B^+=\{x\in\RR^d:|x|<1, \ x_d>0\}$, onto the open slab, $S=\RR^{d-1}\times (0, \pi/2)\}$, of height $\pi/2$. Moreover, the map $\Phi$ extends to a $C^\infty$-diffeomorphism of open $C^\infty$ manifolds with boundary which identifies the following open portions of the boundary, $\partial B^+$, with the corresponding open portions of the boundary, $\partial S$:
\begin{enumerate}
\item The $(d-1)$-dimensional open hemisphere, $\{x\in\bar B^+: |x|=1, \ x_d>0\}$, with the $(d-1)$-dimensional hyperplane, $\RR^{d-1} \times \{\pi/2\}\subset\RR^d$, and
\item The $(d-1)$-dimensional open unit ball, $\{x\in\bar B^+: |x|<1, \ x_d=0\}$, with the $(d-1)$-dimensional hyperplane, $\RR^{d-1} \times \{0\}\subset\RR^d$.
\end{enumerate}
Moreover, $\Phi$ extends to a homeomorphism from $\bar B^+$ onto $\bar S\cup\{\pm\infty\}$ which identifies
\begin{enumerate}
\item[(3)]\setcounter{enumi}{3} The point $e_1\in \bar B^+$ with a point at infinity labeled as $\{\infty\}\times\RR^{d-2}\times(0,\pi/2)$;
\item The punctured $(d-2)$-dimensional unit sphere, $\{x\in\bar B^+: |x|=1, \ x_d=0, \ x\neq e_1\}$, with a point at infinity labeled as $\{-\infty\}\times\RR^{d-2}\times(0,\pi/2)$.
\end{enumerate}
\end{lem}

By defining $v(w) := u(x)$, for $x\in\underline B^+\cup\partial_1 B^+$, and substituting into the expression \eqref{eq:Generator} for $Au(x)$ and using
$$
u_{x_i} = v_{w_k}\frac{\partial w_k}{\partial x_i} \quad\hbox{and}\quad
u_{x_ix_j} = v_{w_kw_l}\frac{\partial w_k}{\partial x_i}\frac{\partial w_l}{\partial x_j}
+ v_{w_k}\frac{\partial^2 w_k}{\partial x_i\partial x_j},
$$
we obtain
\begin{align*}
Au(x) &= -x_da^{ij}(x)u_{x_ix_j}(x) - b^i(x)u_{x_i} + c(x)u(x)
\\
&= -x_d(w)a^{ij}(x(w))v_{w_kw_l}\frac{\partial w_k}{\partial x_i}\frac{\partial w_l}{\partial x_j}
- \left(b^k(x(w)) + x_d(w)a^{ij}(x(w))\frac{\partial^2 w_k}{\partial x_i\partial x_j}\right)v_{w_k} + c(x(w))v
\\
&= f(x(w)).
\end{align*}
We now define
\begin{equation}
\label{eq:Generator_slab}
\tilde Av(w) := -w_d\tilde a^{ij}(w)v_{w_kw_l} - \tilde b^k(w)v_{w_k} + \tilde c(w)v, \quad w \in S,
\end{equation}
where, for $1\leq k\leq d$,
\begin{align}
\label{eq:Generator_slab_aij}
\tilde a^{ij}(w) &:= \frac{x_d(w)}{w_d}a^{ij}(x(w))\frac{\partial w_k}{\partial x_i}\frac{\partial w_l}{\partial x_j},
\\
\label{eq:Generator_slab_bi}
\tilde b^k(w) &:= b^k(x(w)) + x_d(w)a^{ij}(x(w))\frac{\partial^2 w_k}{\partial x_i\partial x_j},
\\
\label{eq:Generator_slab_c}
\tilde c(w) &:= c(x(w)), \quad w\in S.
\end{align}
We record the properties of the coefficients of $\tilde A$ which we shall need in the following

\begin{lem}[Properties of the coefficients of the operator on the slab obtained by pull-back of the operator on a half-ball]
\label{lem:Generator_slab_coefficient_properties}
Let $A$ be as in \eqref{eq:Generator} with coefficients, $a^{ij}$, $b^i$, $c$, defined on $B^+$ and let $\tilde A$ be as in \eqref{eq:Generator_slab}, with corresponding coefficients, $\tilde a^{kl}$, $\tilde b^k$, $\tilde c$, defined on $S$. Then the following hold.
\begin{enumerate}
\item\label{item:Ellipticity_generator_slab} (Ellipticity of the matrix $(\tilde a^{kl})$.) The matrix $(a^{kl})$ obeys
$$
\langle a\xi, \xi\rangle \geq 0 \quad\left(\hbox{respectively,}\  \langle a\xi, \xi\rangle > 0\right)
\quad\hbox{on }B^+, \quad\forall\,\xi\in\RR^d,
$$
if and only if the matrix $(\tilde a^{kl})$ obeys
$$
\langle \tilde a\xi, \xi\rangle \geq 0 \quad\left(\hbox{respectively,}\  \langle a\xi, \xi\rangle > 0\right)
\quad\hbox{on }S, \quad\forall\,\xi\in\RR^d.
$$
Moreover, the matrix $(a^{kl})$ is symmetric if and only the matrix $(\tilde a^{kl})$ is symmetric.
\item\label{item:Local_strict_ellipticity_generator_slab} (Local strict ellipticity of the matrix $(\tilde a^{kl})$ on bounded subsets.)
If the matrix $(a^{kl})$ is strictly elliptic on $\underline B^+$, that is, obeys \eqref{eq:Strict_ellipticity_domain} with $\sO=B^+$ and positive constant of ellipticity $\lambda_0$, then for every $U\Subset \bar S$, the matrix $(\tilde a^{kl})$ is strictly elliptic on $\underline U$, that is, obeys \eqref{eq:Strict_ellipticity_domain} with $\sO=U$ and some positive constant of ellipticity, $\lambda_U$.
\item\label{item:Generator_slab_lower_bound_bd} (Lower bound for $\tilde b^d$ on $\partial_0S$.) One has $b^d\geq 0$ on $\partial_0B^+$ (respectively, $b^d\geq b_0$ on $\partial_0B^+$ for some positive constant, $b_0$) if and only if $\tilde b^d\geq 0$ on $\partial_0S$ (respectively, $\tilde b^d\geq b_0$ on $\partial_0S$).
\item\label{item:Generator_slab_lower_bound_c} (Lower bound for $\tilde c$ on $S$.) One has $c\geq 0$ on $B^+$ (respectively, $c\geq c_0$ on $B^+$ for some positive constant, $c_0$) if and only if  $\tilde c\geq 0$ on $S$ (respectively, $\tilde c\geq c_0$ on $S$).
\item\label{item:Generator_slab_holder_continuous_coefficients} (H\"older continuity of the coefficients.) The coefficients $a^{ij}$, $b^i$, $c$ belong to $C^\alpha_s(\underline B^+)$ (respectively, $C^\alpha_s(\underline B^+\cup\partial_1 B^+)$) if and only if $\tilde a^{ij}$, $\tilde b^i$, $\tilde c$ belong to $C^\alpha_s(\underline S)$ (respectively, $C^\alpha_{s,\loc}(\bar S)$).
\end{enumerate}
\end{lem}

\begin{proof}
Items \eqref{item:Ellipticity_generator_slab} and \eqref{item:Generator_slab_lower_bound_c} are clear. Item \eqref{item:Local_strict_ellipticity_generator_slab} follows from the fact that, for any $R>0$ and cylinder $U_R := B_R^{d-1}\times (0,\pi/2) \subset \RR^{d-1}\times(0,\pi/2)$ (where $B_R^{d-1} := \{x'\in\RR^{d-1}:|x'|<R\}$ is the open ball in $\RR^{d-1}$ with center at the origin and radius $R$), Lemma \ref{lem:Properties_map_halfball_slab} implies that $\Phi^{-1}$ is a diffeomorphism from $\bar U_R\subset\bar S$ onto its image in $\underline B^+\cup\partial_1B^+$ which identifies $\partial_0U_R$ with a subset of $\partial_0B^+$ and $\partial_1U_R$ with a subset of $\partial_1B^+$ (see Figure \ref{fig:conformalmaprectangle}). For Item \eqref{item:Generator_slab_lower_bound_bd}, we use the definition \eqref{eq:Generator_slab_bi} of $\tilde b^k$ to write
$$
\tilde b^d(w) = b^d(x(w)) + w_d\frac{x_d(w)}{w_d} a^{ij}(x(w))\frac{\partial^2 w_d}{\partial x_i\partial x_j},
$$
and, denoting $w = (w',w_d)\in \RR^{d-1}\times\RR_+$, we see that Lemma \ref{lem:Properties_map_halfball_slab} implies, by taking limits as $w_d\downarrow 0$, that
$$
\lim_{w_d\downarrow 0}\tilde b^d(w', w_d) = b^d(x(w',0)), \quad w' \in \partial_0S,
$$
by the same reasoning as for Item \eqref{item:Local_strict_ellipticity_generator_slab}, and thus Item \eqref{item:Generator_slab_lower_bound_bd} follows immediately. Finally, Item \eqref{item:Generator_slab_holder_continuous_coefficients} also follows from Lemma \ref{lem:Properties_map_halfball_slab}.
\end{proof}

Lemma \ref{lem:Generator_slab_weak_maximum_principle_property} below circumvents the fact that the coefficients $\tilde a$ and $\tilde b$ of $\tilde A$ need not obey the quadratic growth condition \eqref{eq:Quadratic_growth} normally required for the maximum principle to hold on unbounded domains.

\begin{lem}[Weak maximum principle property for an operator on a slab]
\label{lem:Generator_slab_weak_maximum_principle_property}
Let $A$ be as in \eqref{eq:Generator}, with coefficients defined on $B^+$, and let $\tilde A$ be the corresponding operator in \eqref{eq:Generator_slab}, with coefficients defined on $S$. Then $A$ has the weak maximum principle property on $\underline B^+$ in the sense of Definition \ref{defn:Weak_max_principle_property} if and only if this is true for $\tilde A$ on $\underline S$.
\end{lem}

\begin{proof}
The conclusion is immediate because of the relationship between $A$ and a function $u$ on $B^+$ and $\tilde A$ and a function $v=u\circ\Phi^{-1}$ on $S$ via the diffeomorphism $\Phi$ in \eqref{eq:Defn_Phi} and its properties in Lemma \ref{lem:Properties_map_halfball_slab}.
\end{proof}

\subsection{Existence of solutions to the partial Dirichlet boundary value problem on a half-ball}
\label{subsec:Existence_dirichlet_boundary_value_problem_halfball}
We first use Lemmas \ref{lem:Properties_map_halfball_slab}, \eqref{lem:Generator_slab_coefficient_properties}, and \eqref{lem:Generator_slab_weak_maximum_principle_property} to provide us with a partial Dirichlet problem on the slab which is equivalent to one on the half-ball and then proceed to prove Theorem \ref{thm:Existence_uniqueness_elliptic_Dirichlet_halfball}.

\begin{lem}[Equivalence between the partial Dirichlet problems on the half-ball and the slab]
\label{lem:Existence_uniqueness_elliptic_Dirichlet_halfball_slab_equivalence}
Assume the hypotheses of Theorem \ref{thm:Existence_uniqueness_elliptic_Dirichlet_halfball} with\footnote{The restrictions $r=1$ and $x_0=O\in\RR^d$ are included merely to simplify notation; naturally, the result holds for any $r>0$ and $x_0\in\partial\HH$.} $r=1$ and $x_0=O\in\RR^d$. Then the following hold.
\begin{enumerate}
\item (Partial Dirichlet problem with $C^{2+\alpha}_s$ boundary data)
A function $u \in C^{2+\alpha}_s(\underline B^+\cup \partial_1 B^+)\cap C_b(B^+)$ is a solution to the Dirichlet problem on the unit half-ball, $B^+$,
\begin{align}
\label{eq:Elliptic_equation_unithalfball}
Au &= f \quad\hbox{on } B^+,
\\
\label{eq:Elliptic_boundary_condition_unithalfball}
u &= g \quad\hbox{on }\partial_1B^+,
\end{align}
for
$$
f \in C^\alpha_s(\underline B^+\cup\partial_1B^+)\cap C_b(B^+) \quad\hbox{and}\quad g \in C^{2+\alpha}_s(\underline B^+\cup\partial_1B^+)\cap C_b(B^+),
$$
if and only if $v=u\circ\Phi^{-1}$ is a solution to the partial Dirichlet problem on the slab, $S=\RR^{d-1}\times (0,\pi/2)$,
\begin{align}
\label{eq:Elliptic_equation_halfpislab}
\tilde Av &= \tilde f \quad\hbox{on } S,
\\
\label{eq:Elliptic_boundary_condition_halfpislab}
v &= \tilde g \quad\hbox{on }\partial_1S,
\end{align}
with $\tilde A$ in \eqref{eq:Generator_slab}, and $v$ belonging to $C^{2+\alpha}_{s,\loc}(\bar S)\cap C_b(S)$, and
$$
\tilde f = f\circ\Phi^{-1} \in C^\alpha_{s,\loc}(\bar S)\cap C_b(S) \quad\hbox{and}\quad \tilde g = g\circ\Phi^{-1} \in C^{2+\alpha}_{s,\loc}(\bar S)\cap C_b(S).
$$
\item (Partial Dirichlet problem with continuous boundary data)
A function $u \in C^{2+\alpha}_s(\underline B^+)\cap C_b(B^+\cup\partial_1 B^+)$ is a solution to the partial Dirichlet problem \eqref{eq:Elliptic_equation_unithalfball}, \eqref{eq:Elliptic_boundary_condition_unithalfball} on the unit half-ball, $B^+$, for
$$
f \in C^\alpha_s(\underline B^+)\cap C_b(B^+) \quad\hbox{and}\quad g \in C_b(\partial_1B^+),
$$
if and only if $v = u\circ\Phi^{-1}$ is a solution to the partial Dirichlet problem \eqref{eq:Elliptic_equation_halfpislab}, \eqref{eq:Elliptic_boundary_condition_halfpislab} on the slab, $S$, and\footnote{In particular, $u$ belongs to $C_b(\underline B^+\cup\partial_1 B^+)$ but not necessarily $C(\bar B^+)$.} $v$ belongs to $C^{2+\alpha}_s(\underline S)\cap C_b(S\cup\partial_1S)$, and
$$
\tilde f = f\circ\Phi^{-1} \in C^\alpha_s(\underline S)\cap C_b(S) \quad\hbox{and}\quad \tilde g = g\circ\Phi^{-1} \in C_b(\partial_1S).
$$
\end{enumerate}
\end{lem}

\begin{proof}
The identifications of the function spaces for $u$, $f$, $g$ (and coefficients of $A$) with their counterparts for $v$, $\tilde f$, $\tilde g$ (and coefficients of $\tilde A$) are immediate from Lemma \ref{lem:Properties_map_halfball_slab}. The conclusions now follow from the fact that $\Phi:\underline B^+\cup \partial_1B^+ \to \bar S$ is a $C^\infty$-diffeomorphism of (open) $C^\infty$-manifolds with boundary.
\end{proof}

We can now give the

\begin{proof}[Proof of Theorem \ref{thm:Existence_uniqueness_elliptic_Dirichlet_halfball}]
To simplify notation, we shall assume that $r=1$ and $x_0=O\in\RR^d$ in Theorem \ref{thm:Existence_uniqueness_elliptic_Dirichlet_halfball}, so $B_r(x_0)=B^+$ is the unit half-ball; clearly this simplification makes no difference to the logic of the proof.

By Remark \ref{rmk:Elliptic_weak_maximum_principle_finite_height_domain} we may assume without loss of generality that both $b\geq b_0$ on $B^+$ and $c\geq c_0$ on $B^+$, for some positive constants $b_0$ and $c_0$, since $\sup_{B^+}a^{dd} >0$ (we have $a^{dd}\in C^\alpha_s(\bar B^+)$) and either $\inf_{\partial_0B^+}b^d >0$ or $\inf_{B^+}c >0$ by hypothesis and thus, if $c$ does not already have a uniform positive lower bound on $B^+$, this may be achieved by a simple change of dependent variable whenever $\inf_{B^+}b^d >0$ or $\inf_{B^+}c >0$ since $\height(B^+)$ is obviously finite.

Lemma \ref{lem:Existence_uniqueness_elliptic_Dirichlet_halfball_slab_equivalence} assures us that it is enough to consider uniqueness and existence of solutions to the partial Dirichlet problem \eqref{eq:Elliptic_equation_halfpislab}, \eqref{eq:Elliptic_boundary_condition_halfpislab} on a slab, $S$. We first consider the case $\tilde f\in C^\alpha_{s,\loc}(\bar S)\cap C_b(S)$ and $\tilde g \in C^{2+\alpha}_{s,\loc}(\bar S)\cap C_b(S)$ and verify existence of a solution $v\in C^{2+\alpha}_{s,\loc}(\bar S)\cap C_b(S)$ to \eqref{eq:Elliptic_equation_halfpislab}, \eqref{eq:Elliptic_boundary_condition_halfpislab}. By Definition \ref{defn:DH2spaces}, any such function, $v$, belongs to $C^2(S)\cap C^1(\underline S)\cap C_b(S)$ and Lemma \ref{lem:Properties_second_order_derivatives} implies that $v$ obeys \eqref{eq:WMP_C2s_regularity}, \eqref{eq:WMP_xdsecond_order_derivatives_zero_boundary} with $\sO=S$, that is, $w_dD^2v \in C(\underline S)$ and $w_dD^2v=0$ on $\partial_0S$.

By Lemma \ref{lem:Generator_slab_weak_maximum_principle_property}, the operator $\tilde A$ has the weak maximum principle property on $\underline S$ in the sense of Definition \ref{defn:Weak_max_principle_property}. The weak maximum principle (Theorem \ref{thm:Elliptic_weak_maximum_principle_unbounded_domain}) implies that there is at most one function $v \in C^2(S)\cap C^1(\underline S)\cap C_b(S)$ obeying \eqref{eq:WMP_C2s_regularity}, \eqref{eq:WMP_xdsecond_order_derivatives_zero_boundary} with $\sO=S$, that is, $w_dD^2v \in C(\underline S)$ and $w_dD^2v=0$ on $\partial_0S$, and solving the partial Dirichlet problem  \eqref{eq:Elliptic_equation_halfpislab}, \eqref{eq:Elliptic_boundary_condition_halfpislab} on the slab, $S$. Thus it remains to consider existence of solutions to \eqref{eq:Elliptic_equation_halfpislab}, \eqref{eq:Elliptic_boundary_condition_halfpislab}.

Proposition \ref{prop:Elliptic_C2_weak_max_principle_apriori_estimates} \eqref{item:C2_solution_f_arb_sign}, together with Theorem \ref{thm:Elliptic_weak_maximum_principle_unbounded_domain} and Lemma \ref{lem:Generator_slab_weak_maximum_principle_property}, provides an a priori estimate for $v$ on the slab, $S$,
\begin{equation}
\label{eq:Elliptic_weak_maximum_principle_estimate_slab}
\|v\|_{C(\bar S)}
\leq \frac{1}{c_0}\|\tilde f\|_{C(\bar S)} \vee \|\tilde g\|_{C(\overline{\partial_1S})}.
\end{equation}
Set $U_n:=B_n^{d-1}(O)\times(0,\pi/2)\subset S$, for integers $n\geq 1$, where $B_r^{d-1}(O)\subset\RR^{d-1}$ is the ball of radius $r>0$ and center at the origin. Given coefficients $a^{kl}$, $b^k$, $c$ of $A$ belonging to $C_s^\alpha(\underline B^+\cup\partial_1 B^+)$ as in the hypotheses of Theorem \ref{thm:Existence_uniqueness_elliptic_Dirichlet_halfball}
and corresponding coefficients $\tilde a^{kl}$, $\tilde b^k$, $\tilde c$ of $\tilde A$ belonging to $C_{s,\loc}^\alpha(\bar S)$, as provided by Lemma \ref{lem:Generator_slab_coefficient_properties}, we choose sequences $\tilde a^{kl}_n$, $\tilde b^k_n$, $\tilde c_n$, for $n\geq 1$, belonging to $C_s^\alpha(\bar S)$ and which obey the hypotheses of Theorem \ref{thm:Elliptic_existence_uniqueness_slab} and which coincide with $\tilde a^{kl}$, $\tilde b^k$, $\tilde c$ upon restriction to the bounded subsets $U_n \subset S$ for $n \geq 1$. For example, writing $w=(w',w_d)\in\RR^{d-1}\times\RR$, we may define
$$
\tilde a^{kl}_n(w) := \begin{cases} \tilde a^{kl}(w), & w \in U_n, \\ \tilde a^{kl}(nw'/|w'|,w_d), & w \in S\less U_n,\end{cases}
$$
and similarly for $\tilde b^k_n$ and $\tilde c_n$.

While the functions $\tilde f$ and $\tilde g$ are bounded on $S$, their respective global H\"older norms need not be finite. Therefore, we define a sequence $\{\tilde f_n\}_{n\geq 1}\subset C_s^\alpha(\bar S)$ which coincides with $\tilde f$ upon restriction to the bounded subsets $U_n \subset S$ for $n \geq 1$ in the same way that we defined $\tilde a^{kl}_n$. To define a sequence $\{\tilde g_n\}_{n\geq 1}\subset C_s^{2+\alpha}(\bar S)$, we first choose a cutoff function $\zeta \in C^\infty(\RR)$ such that $\zeta(t)=1$ for $t\leq 1$ and $\zeta(t)=0$ for $t\geq 2$ and $0\leq \zeta\leq 1$ on $\RR$ and set
$$
\tilde g_n(w) := \zeta(|w'|/n)\tilde g(w) + (1-\zeta(|w'|/n))\tilde g(nw'/|w'|, w_d), \quad w\in S.
$$
The sequence $\{\tilde g_n\}_{n\geq 1}\subset C_s^{2+\alpha}(\bar S)$ coincides with $\tilde g$ upon restriction to the bounded subsets $U_n \subset S$ for $n \geq 1$.

We can now apply Theorem \ref{thm:Elliptic_existence_uniqueness_slab} to find $v_n \in C^{2+\alpha}_s(\bar S)$ satisfying \eqref{eq:Elliptic_equation_halfpislab}, \eqref{eq:Elliptic_boundary_condition_halfpislab}, with $\tilde A$, $\tilde f$, and $\tilde g$ replaced by $\tilde A_n$, $\tilde f_n$, and $\tilde g_n$, that is
\begin{align}
\label{eq:Elliptic_equation_halfpislab_n}
\tilde A_nv_n &= \tilde f_n \quad\hbox{on } S,
\\
\label{eq:Elliptic_boundary_condition_halfpislab_n}
v_n &= \tilde g_n \quad\hbox{on }\partial_1S,
\end{align}
for all integers $n\geq 1$, where
$$
\tilde A_nv_n := -w_d\tr(\tilde a_n D^2v_n) - \langle\tilde b_n, v_n\rangle + \tilde c v_n \quad\hbox{on }\underline S.
$$
Lemma \ref{lem:Elliptic_weak_maximum_principle_finite_height_domain} implies that $\tilde A_n$ has the weak maximum principle property on $\underline S$ in the sense of Definition \ref{defn:Weak_max_principle_property} and so Proposition \ref{prop:Elliptic_C2_weak_max_principle_apriori_estimates} \eqref{item:C2_solution_f_arb_sign} implies that the solutions $v_n$ obey the weak maximum principle estimate \eqref{eq:Elliptic_weak_maximum_principle_estimate_slab}, that is,
\begin{equation}
\label{eq:Elliptic_weak_maximum_principle_estimate_slab_sequence}
\|v_n\|_{C(\bar S)}
\leq \frac{1}{c_0}\|\tilde f_n\|_{C(\bar S)} \vee \|\tilde g_n\|_{C(\overline{\partial_1S})}, \quad\forall\, n\geq 1,
\end{equation}
where we note that the right-hand side of the preceding inequality is in turn uniformly bounded, independently of $n$, since
$$
\|\tilde f_n\|_{C(\bar S)} \leq \|\tilde f\|_{C(\bar S)}
\quad\hbox{and}\quad
\|\tilde g_n\|_{C(\overline{\partial_1S})} \leq \|\tilde g\|_{C(\overline{\partial_1S})}, \quad\forall\, n\geq 1,
$$
by construction of these sequences. By appealing to\cite[Corollary 6.7]{GilbargTrudinger} for a ball $B_{2\rho}(x_0)\Subset\HH$ with $x_0\in\partial_1S$ and $\rho>0$, we obtain an a priori Schauder estimate of the form,
$$
\|v_n\|_{C^{2,\alpha}(\overline{B_\rho(x_0)\cap S})} \leq C\left(\|f_n\|_{C^\alpha(\overline{B_\rho(x_0)\cap S})} + \|g_n\|_{C^{2,\alpha}(\overline{B_\rho(x_0)\cap S})} + \|v_n\|_{C(\overline{B_\rho(x_0)\cap S})}\right),
$$
for all $n\geq 1$ and constant $C$ independent of $n$, and appealing to Theorem \ref{thm:Elliptic_apriori_Schauder_interior_domain} for precompact open subsets $U\Subset \underline S$ and $U'\Subset \underline U$, we obtain an a priori Schauder estimate of the form,
$$
\|v_n\|_{C^{2+\alpha}_s(\bar U')} \leq C\left(\|f_n\|_{C^\alpha_s(\bar U)} + \|v_n\|_{C(\bar U)}\right),
$$
again for all $n\geq 1$ and constant $C$ independent of $n$. Hence, for precompact open subsets $V\Subset\bar S$ and $V'\Subset\bar V$, we may combine the preceding estimates via a standard covering argument to give
\begin{equation}
\label{eq:Elliptic_apriori_Schauder_interior_slab}
\|v_n\|_{C^{2+\alpha}_s(\bar V')} \leq C\left(\|f_n\|_{C^\alpha_s(\bar V)} + \|g_n\|_{C^{2+\alpha}_s(\bar V')} + \|v_n\|_{C(\bar V)}\right),
\end{equation}
for all $n\geq 1$ and constant $C$ independent of $n$. We have $C^\alpha_s(\bar U_k) \subset C^{\alpha/2}(\bar U_k)$ by \eqref{eq:DH_and_standard_Holder_relationships} and so, using
$$
\|v_n\|_{C^{\alpha/2}\bar U_k)} \leq \|v_n\|_{C^\alpha_s(\bar U_k)}
\quad\hbox{and}\quad
\|v_n\|_{C^\alpha_s(\bar U_k)} \leq \|v_n\|_{C^{2+\alpha}_s(\bar U_k)},
$$
and the a priori estimate \eqref{eq:Elliptic_apriori_Schauder_interior_slab} (with $U_k\subset U_{k+1}$ replacing $V'\subset V$), the Arzel\`a-Ascoli Theorem implies that, after passing to a subsequence, $v_{n_k}$ converges in $C(\bar U_k)$ to a limit in $C(\bar U_k)$. Therefore, the diagonal subsequence, relabeled as $\{v_n\}_{n\geq 1}$, converges to a limit $v\in C_{\loc}(\bar S)$ on precompact subsets $U_k\Subset\bar S$ as $n\to\infty$, for each fixed $k\geq 1$, and because of the bound \eqref{eq:Elliptic_weak_maximum_principle_estimate_slab_sequence}, we must have $v\in C_{\loc}(\bar S) \cap C_b(S) = C(\bar S)$.

In particular, the subsequence, $\{v_n\}_{n\geq 1}$, is Cauchy in $C(\bar V)$ for precompact open subsets $V\Subset\bar S$ and the a priori estimate \eqref{eq:Elliptic_apriori_Schauder_interior_slab} then implies that the subsequence, $\{v_n\}_{n\geq 1}$, is Cauchy in $C^{2+\alpha}_s(\bar V')$ for precompact open subsets $V'\Subset\bar S$. Thus, $\{v_n\}_{n\geq 1}$ converges in $C^{2+\alpha}_{s,\loc}(\bar S)$ as $n\to\infty$ to a limit $v\in C^{2+\alpha}_{s,\loc}(\bar S)$, necessarily coinciding with the limit $v\in C(\bar S)$ already discovered, and so $v\in C^{2+\alpha}_{s,\loc}(\bar S)\cap C_b(S)$. (Compare, the proof of \cite[Lemma 6.10]{GilbargTrudinger}.) In addition, by taking limits in \eqref{eq:Elliptic_equation_halfpislab_n}, \eqref{eq:Elliptic_boundary_condition_halfpislab_n}, we also see that $v$ solves the partial Dirichlet problem \eqref{eq:Elliptic_equation_halfpislab}, \eqref{eq:Elliptic_boundary_condition_halfpislab} on the slab, $S$.

%COMMENT Remember that we will still assume Holder regularity of the coefficients of A and \tilde A up to the upper boundary here!
%COMMENT C(\bar S) \neq C_{\loc}(\bar S)\cap C_b(S)
It remains to consider the case $\tilde f\in C^\alpha_s(\underline S)\cap C_b(S)$ and $\tilde g \in C_b(\partial_1S)$ and verify existence of $v\in C^{2+\alpha}_s(\underline S)\cap C_b(S\cup\partial_1S)$. We accomplish this by adapting the proof of \cite[Theorem 8.30]{GilbargTrudinger}. By Definitions \ref{defn:Calphas} and \ref{defn:DH2spaces}, we have $C^1_{\loc}(\bar S) \subset C^\alpha_{s,\loc}(\bar S)$ and $C^3_{\loc}(\bar S) \subset C^{2+\alpha}_{s,\loc}(\bar S)$. Choose a sequence $\{\tilde f_m\}_{m\in\NN}\subset C^1_{\loc}(\bar S)\cap C_b(S)$ which converges in $C^\alpha_s(\underline S)\cap C_b(S)$ to $\tilde f \in C^\alpha_s(\underline S)\cap C_b(S)$ and a sequence $\{\tilde g_m\}_{m\in\NN}\subset C^3_{\loc}(\bar S)\cap C_b(S)$ which converges in $C_b(\partial_1S)$ to $\tilde g \in C_b(\partial_1S)$. (For example, one may apply \cite[Corollary 1.29]{Adams_1975} and a partition of unity.)
%COMMENT Add approximation and extension lemmas for DH spaces

Let $\{v_m\}_{m\in\NN}\subset C^{2+\alpha}_{s,\loc}(\bar S)\cap C_b(S)$ be the corresponding sequence of solutions to the partial Dirichlet problem
\eqref{eq:Elliptic_equation_halfpislab}, \eqref{eq:Elliptic_boundary_condition_halfpislab} on the slab, $S$, with $\tilde f_m$ and $\tilde g_m$, that is
$$
\tilde Av_m = \tilde f_m \quad\hbox{on }S, \quad v_m = \tilde g_m \quad\hbox{on }\partial_1 S,
$$
provided by our preceding analysis.

The weak maximum principle estimate \eqref{eq:Elliptic_weak_maximum_principle_estimate_slab} implies that $\{v_m\}_{m\in\NN}$ is a Cauchy sequence in the Banach space $C_b(\bar S):=C_{\loc}(\bar S)\cap C_b(S)$ \cite[\S 5.2]{Adams_1975} and thus converges to $v$ in $C_b(\bar S)$. Moreover, by appealing to our a priori interior local Schauder estimate, Theorem \ref{thm:Elliptic_apriori_Schauder_interior_domain}, the sequence $\{v_m\}_{m\in\NN}$ necessarily converges in $C^{2+\alpha}_s(\underline S)$ to the limit $v$, in the sense that the sequence converges in $C^{2+\alpha}_s(\bar V)$ for each precompact open subset $V\Subset\underline S$. Therefore, $v$ belongs to $C^{2+\alpha}_s(\underline S)\cap C_b(S\cup\partial_1S)$ and solves the \eqref{eq:Elliptic_equation_halfpislab}, \eqref{eq:Elliptic_boundary_condition_halfpislab} on the slab, $S$, when we only assume $f\in C^\alpha_s(\underline S)\cap C_b(S)$ and $g\in C_b(\partial_1S)$. This completes the proof of Theorem \ref{thm:Existence_uniqueness_elliptic_Dirichlet_halfball}.
\end{proof}

\section{Perron methods for existence of solutions to degenerate-elliptic boundary value and obstacle problems}
\label{sec:Perron}
We apply the results of the preceding sections to prove existence of solutions to the boundary value and obstacle problems described in \S \ref{sec:Introduction} using a modification of the classical Perron method \cite[\S 2.8 \& \S 6.3]{GilbargTrudinger} for smooth solutions, as distinct from Ishii's version of the Perron method for existence of viscosity solutions \cite[\S 4]{Crandall_Ishii_Lions_1992}, \cite{Ishii_1989}. In \S \ref{sec:Perron_boundary_value_problem}, we develop a Perron method to prove existence of a solution to the boundary value problem \eqref{eq:Elliptic_equation}, \eqref{eq:Elliptic_boundary_condition}, yielding Theorem \ref{thm:Perron_elliptic_bvp_solution} and hence Theorem \ref{thm:Existence_uniqueness_elliptic_Dirichlet}. In \S \ref{sec:Perron_obstacle_problem}, we develop a Perron method to prove existence of a solution to the obstacle problem \eqref{eq:Elliptic_obstacle_problem}, \eqref{eq:Elliptic_boundary_condition}, yielding Theorems \ref{thm:Perron_elliptic_obstacle_problem_solution} and \ref{thm:Perron_elliptic_obstacle_problem_solution_lipschitz_obstacle} and hence Theorems \ref{thm:Existence_uniqueness_elliptic_obstacle} and \ref{thm:Existence_uniqueness_elliptic_obstacle_lipschitz}.

\subsection{A Perron method for existence of solutions to a degenerate-elliptic boundary value problem}
\label{sec:Perron_boundary_value_problem}
By analogy with the definitions of continuous subharmonic and superharmonic functions \cite[\S 2.8]{GilbargTrudinger} or continuous subsolutions and supersolutions to linear, second-order elliptic partial differential equations \cite[pp. 102--103]{GilbargTrudinger}, we make the

\begin{defn}[Continuous subsolution and supersolution to an elliptic equation and boundary problem]
\label{defn:Continuous_super_sub_solution_bvp}
%COMMENT We changed A\bar u = f to A\bar u \geq f and ask that \sup\bar u < \infty
Let $\sO\subseteqq\HH$ be a domain and $A$ be as in \eqref{eq:Generator}. Given $f\in C(\underline\sO)$, we call $u:\underline\sO\to\RR$ a \emph{continuous subsolution} to the elliptic equation \eqref{eq:Elliptic_equation} if $u$ is continuous on $\sO$, locally bounded on $\underline\sO$, and for every open ball $B\Subset\sO$ or half-ball $B^+\Subset\underline\sO$ with center in $\partial_0\sO$ and for every $\bar u\in C^2(U)\cap C^1(\underline U)$, with $U=B$ or $B^+$, obeying \eqref{eq:WMP_C2s_regularity}, \eqref{eq:WMP_xdsecond_order_derivatives_zero_boundary}, and $\inf_U\bar u > -\infty$, and
$$
\begin{cases}
A\bar u \geq f &\hbox{on } U,
\\
\bar u \geq u &\hbox{on } \partial_1U,
\end{cases}
$$
we then have
$$
u\leq \bar u \quad \hbox{on } U.
$$
Given $g\in C(\partial_1\sO)$, we call $u\in C(\sO\cup\partial_1\sO)$ a \emph{continuous subsolution} to the boundary value problem \eqref{eq:Elliptic_equation}, \eqref{eq:Elliptic_boundary_condition} if $u$ is a continuous subsolution to \eqref{eq:Elliptic_equation} and $u\leq g$ on $\partial_1\sO$.

We call $v\in C(\sO)$ a \emph{continuous supersolution} to the elliptic equation \eqref{eq:Elliptic_equation} if $-v$ is a subsolution to \eqref{eq:Elliptic_equation}; we call $v\in C(\sO\cup\partial_1\sO)$ a \emph{continuous supersolution} to the boundary value problem \eqref{eq:Elliptic_equation}, \eqref{eq:Elliptic_boundary_condition} if $-v$ is a continuous subsolution to \eqref{eq:Elliptic_equation}, \eqref{eq:Elliptic_boundary_condition}. \qed
\end{defn}

\begin{rmk}[Continuity of subsolutions and supersolutions]
It is important to note that in Definition \ref{defn:Continuous_super_sub_solution_bvp} subsolutions and supersolutions to \eqref{eq:Elliptic_equation} are defined to only be continuous on $\sO$ and not $\underline\sO$, since continuity is not necessarily preserved along $\partial_0\sO$ by the ``harmonic lifting'' process. Fortunately, this has no impact on the application of Perron's method.\qed
\end{rmk}

It will be convenient to isolate local solvability of the partial Dirichlet problem from specific conditions on $A$ which ensure solvability.

\begin{hyp}[Local existence and uniqueness of solutions to the partial Dirichlet problem]
\label{hyp:Dirichlet_problem_local_solvability}
Suppose that $U\subset\HH$ is a ball $B\Subset\HH$ with center in $\HH$ or half-ball $B^+\Subset\underline\HH$ with center in $\partial\HH$, the coefficients of $A$ in \eqref{eq:Generator} belong to $C^\alpha_s(\underline U)$, and $f\in C^\alpha_s(\underline U)\cap C_b(U)$ for some $\alpha\in (0,1)$, and $g\in C_b(\partial_1 U)$, and $A$ has the weak maximum principle property on $\underline U$ in the sense of Definition \ref{defn:Weak_max_principle_property}. Then there is a unique solution $u\in C^{2+\alpha}_s(\underline U)\cap C_b(U\cup\partial_1U)$ to the partial Dirichlet problem\footnote{Existence and uniqueness of $\bar u$ is given by Theorem \ref{thm:Existence_uniqueness_elliptic_Dirichlet_halfball} when $U=B^+$ and \cite[Lemma 6.10]{GilbargTrudinger} when $U=B$.},
$$
\begin{cases}
Au=f &\hbox{on } U,
\\
u = g &\hbox{on } \partial_1 U.
\end{cases}
$$
\end{hyp}

\begin{defn}[Local solvability of the partial Dirichlet problem]
\label{defn:Dirichlet_problem_local_solvability}
Let $\sO\subseteqq\HH$ be a domain, $\alpha\in (0,1)$, and let $A$ be as in \eqref{eq:Generator} with coefficients belonging to $C^\alpha_s(\underline \sO)$. We say that the \emph{partial Dirichlet problem for $A$ is locally solvable in $\underline\sO$} if Hypothesis \ref{hyp:Dirichlet_problem_local_solvability} holds for each ball $B\Subset\sO$ with center in $\sO$ or half-ball $B^+\Subset\underline\sO$ with center in $\partial_0\sO$.
\end{defn}

We have the following analogue of \cite[Properties (i)--(iv), p. 103]{GilbargTrudinger} for continuous subsolutions and supersolutions in the sense of Definition \ref{defn:Continuous_super_sub_solution_bvp}. Properties \eqref{item:elliptic_regular_continuous_subsolution_is_classical_subsolution} and \eqref{item:elliptic_classical_subsolution_is_continuous_subsolution} are included in Theorem \ref{thm:Elliptic_equation_continuous_sub_supersolution_properties} for completeness but are not used subsequently.

\begin{thm}[Properties of continuous subsolutions and supersolutions]
\label{thm:Elliptic_equation_continuous_sub_supersolution_properties}
If $\sO\subseteqq\HH$ is a domain\footnote{Recall that by a ``domain'' in $\HH$, we always mean a \emph{connected}, open subset.}, $A$ is as in \eqref{eq:Generator}, and $f\in C(\underline\sO)$, then the following hold.
\begin{enumerate}
\item \label{item:elliptic_regular_continuous_subsolution_is_classical_subsolution} (A continuous subsolution with sufficient regularity is a smooth subsolution) Suppose $f$ and the coefficients of $A$ belong to $C^\alpha_s(\underline\sO)$ and $A$ obeys Hypothesis \ref{hyp:Dirichlet_problem_local_solvability} for balls $B\Subset\sO$. If $u\in C^2(\sO)\cap C^1(\underline\sO)$ obeys \eqref{eq:WMP_C2s_regularity}, \eqref{eq:WMP_xdsecond_order_derivatives_zero_boundary} and is a continuous subsolution to \eqref{eq:Elliptic_equation}, then $Au\leq f$ on $\sO$, that is, $u$ is a smooth subsolution to \eqref{eq:Elliptic_equation}.

\item \label{item:elliptic_classical_subsolution_is_continuous_subsolution} (A smooth subsolution is a continuous subsolution) Assume $A$ has the weak maximum principle property on $\underline\sO$ in the sense of Definition \ref{defn:Weak_max_principle_property}. If $u\in C^2(\sO)\cap C^1(\underline\sO)$ obeys \eqref{eq:WMP_C2s_regularity}, \eqref{eq:WMP_xdsecond_order_derivatives_zero_boundary} and is a subsolution to \eqref{eq:Elliptic_equation}, that is, $Au\leq f$ on $\sO$, then $u$ is a continuous subsolution to \eqref{eq:Elliptic_equation}.

\item \label{item:elliptic_strong_maximum_principle} (Continuous subsolutions obey the strong maximum principle) Assume that $A$ has the strong maximum principle property on $\underline\sO$ in the sense of Definition \ref{defn:Strong_max_principle_property}. If $u\in C(\sO)$ is a continuous subsolution to \eqref{eq:Elliptic_equation} with $f = 0$, then $u$ obeys the strong maximum principle, in the sense of the conclusion to Definition \ref{defn:Strong_max_principle_property};

\item \label{item:elliptic_weak_maximum_principle} (Continuous subsolutions obey the weak maximum principle) Assume that $A$ has the strong maximum principle property on $\underline\sO$ in the sense of Definition \ref{defn:Strong_max_principle_property}. If $\sO$ is unbounded, assume in addition that the coefficients of $A$ obey \eqref{eq:Positive_lower_bound_c_domain}, \eqref{eq:Quadratic_growth}, and \eqref{eq:Positive_lower_bound_c_geq_2K} and that $\sO$ obeys \eqref{eq:Connectedness_domain_and_ball}. If $u\in C(\sO)$ is a continuous subsolution to \eqref{eq:Elliptic_equation} with $f = 0$ and $\sup_\sO u < \infty$, then $u$ obeys the weak maximum principle: If $u \leq 0$ on $\partial_1\sO$, then\footnote{Our hypotheses allow for the possibility that $\sO=\HH$, in which case $\partial_1\sO=\emptyset$ and the boundary condition is omitted.}
    $$
    u\leq 0 \quad\hbox{on } \sO.
    $$
\item \label{item:elliptic_strong_comparison_principle} (Strong comparison principle for continuous subsolutions and supersolutions) Assume the conditions on $A$ and $\sO$ for Property \eqref{item:elliptic_weak_maximum_principle}. If $u\in C(\sO)$ is a continuous subsolution to \eqref{eq:Elliptic_equation} with $\sup_\sO u < \infty$ and $v\in C(\sO)$ is a continuous supersolution to \eqref{eq:Elliptic_equation} with $\inf_\sO v > -\infty$ such that $u \leq v$ on $\partial_1\sO$, then\footnotemark[\value{footnote}]
    either $u<v$ throughout $\sO$ or $u\equiv v$ on $\sO$;

\item \label{item:elliptic_subsolution_harmonic_lift} (Construction of larger continuous subsolutions) Suppose $U\subset\sO$ is an open ball $B\Subset\sO$ or half-ball $B^+\Subset\underline\sO$ with center in $\partial_0\sO$ and $f$ and the coefficients of $A$ belong to $C^\alpha_s(\underline U)$ and $A$ obeys Hypothesis \ref{hyp:Dirichlet_problem_local_solvability} on $U$. Let $u\in C(\sO)$ be a continuous subsolution to \eqref{eq:Elliptic_equation}. Define $\bar u \in C^{2+\alpha}_s(\underline U)\cap C_b(U\cup\partial_1U)$ to be the unique solution to
$$
A\bar u = f \quad\hbox{on } U, \quad \bar u = u \quad\hbox{on }\partial_1 U.
$$
Then the function $\hat v\in C(\sO)$ defined by
\begin{equation}
\label{eq:Elliptic_equation_harmonic_lift}
\hat u(x) := \begin{cases} \bar u(x), & x\in U, \\ u(x), & x \in \sO\less U,\end{cases}
\end{equation}
is a continuous subsolution to \eqref{eq:Elliptic_equation}.

\item \label{item:elliptic_max_set_subsolutions} (Pointwise maximum of a finite set of continuous subsolutions is a continuous subsolution) If $u_1,\ldots, u_N\in C(\sO)$ are continuous subsolutions to \eqref{eq:Elliptic_equation}, then the function $u\in C(\sO)$ defined by $u(x) := \max_{1\leq i\leq N}u_i(x)$, $x\in \underline\sO$, is a continuous subsolution to \eqref{eq:Elliptic_equation}.
\end{enumerate}
With the exception of Property \eqref{item:elliptic_strong_comparison_principle}, the analogous properties for continuous supersolutions, $v$, are obtained by substituting $u=-v$.
\end{thm}

\begin{proof}
For \eqref{item:elliptic_regular_continuous_subsolution_is_classical_subsolution}, suppose that $u$ is a continuous subsolution. If there is a point $z_0\in \sO$ such that $Au(z_0)>f(z_0)$, choose $r>0$ small enough that $Au\geq f$ on the ball $B=B_r(z_0)\Subset\sO$. By \cite[Theorem 6.13]{GilbargTrudinger}, we may define $\bar u \in C^{2,\alpha}(B)\cap C(\bar B)$ by the solution to $A\bar u = f$ on $B$, $\bar u = u$ on $\partial B$, and so $u\leq\bar u$ on $B$ by Definition \ref{defn:Continuous_super_sub_solution_bvp}. But $Au \geq f = A\bar u$ on $B$ and thus $u\geq\bar u$ on $B$ by the weak maximum principle \cite[Theorem 3.3]{GilbargTrudinger}. Hence, $u=\bar u$ on $B$ and $Au=A\bar u=f$ on $B$, contradicting the definition of $z_0$, and therefore we must have $Au\leq f$ on $\sO$.

For \eqref{item:elliptic_classical_subsolution_is_continuous_subsolution}, suppose $U\subset\sO$ is an open ball $B\Subset\sO$ or half-ball $B^+\Subset\underline\sO$ with center in $\partial_0\sO$ and that $\bar u\in C^2(U)\cap C^1(\underline U)$ obeys \eqref{eq:WMP_C2s_regularity}, \eqref{eq:WMP_xdsecond_order_derivatives_zero_boundary}, and $\inf_U\bar u > -\infty$, and
$$
A\bar u \geq f \quad\hbox{on } U, \quad \bar u \geq u \quad\hbox{on } \partial_1U.
$$
But then $A(\bar u - u) \geq 0$ on $U$ and because $A$ has the weak maximum principle property on $\underline\sO$, then $\bar u \geq u$ on $U$. Thus, $u$ is a continuous subsolution in the sense of Definition \ref{defn:Continuous_super_sub_solution_bvp}.

For \eqref{item:elliptic_strong_maximum_principle}, suppose the contrary and that $M := \sup_\sO u = u(x_0)$ at some $x_0\in\underline\sO$, but $u\not\equiv M$ throughout $\sO$. The constant $M$ may have arbitrary sign in the case $c=0$ on $\underline\sO$ while for the case $c\geq 0$ on $\underline\sO$, we assume $M\geq 0$. We first consider the case $c=0$ on $\underline\sO$. We may then choose $U\subset\sO$, where $U$ is an open ball $B_r(x_0)\Subset\sO$ or half-ball $B_r^+(x_0)\Subset\underline\sO$ with $x_0\in\partial_0\sO$, such that $u\not\equiv M$ on $\partial_1U$; to see that $U$ exists, observe that we may suppose $x_0\in \underline\sO\cap\partial\{u<M\}$. Let $\bar u \in C^{2+\alpha}_s(\underline U)\cap C_b(U\cup\partial_1 U)$ be the ``harmonic lift'' of $u$ on $U$ described in Definition \ref{defn:Continuous_super_sub_solution_bvp}, so $A\bar u = 0$ on $U$ and $\bar u=u$ on $\partial_1 U$, and thus $u\leq \bar u$ on $U$. Because $A$ has the strong maximum principle property on $\underline\sO$ by hypothesis, Lemma \ref{lem:Elliptic_strong_implies_weak_maximum_principle_property_bounded_domain} implies that it also has the weak maximum principle property --- in the sense of Definition \ref{defn:Weak_max_principle_property} --- on bounded subdomains of $\sO$. Hence, the weak maximum principle applied to $A\bar u=0$ on $U$ ensures that $\bar u \leq \sup_{\partial_1U}\bar u$ on $U$ and so (by Proposition \ref{prop:Elliptic_C2_weak_max_principle_apriori_estimates} \eqref{item:C2_subsolution_f_leq_zero})
$$
M \geq \sup_{\partial_1U}u = \sup_{\partial_1U}\bar u \geq \bar u(x_0) \geq u(x_0) = M.
$$
Thus equality holds throughout and $\bar u(x_0)=M$. Because $A$ has the strong maximum principle property on $\underline\sO$ in the sense of Definition \ref{defn:Strong_max_principle_property} by hypothesis, the strong maximum principle applied to $A\bar u=0$ on $U$ ensures that $\bar u \equiv M$ on $U$ and hence $u = \bar u = M$ on $\partial_1U$, which contradicts the choice of $U$.

For the case $c\geq 0$ on $\underline\sO$ and $M\geq 0$, we again observe that the weak maximum principle applied to $A\bar u=0$ on $U$ ensures that $\bar u \leq \sup_{\partial_1U}\bar u^+$ on $U$ and thus,
$$
M \geq 0\vee\sup_{\partial_1U}u = \sup_{\partial_1U}u^+ = \sup_{\partial_1U}\bar u^+ \geq \bar u(x_0) \geq u(x_0) = M,
$$
and so the argument proceeds just as before.

For \eqref{item:elliptic_weak_maximum_principle}, we first consider the case where $\sO$ is \emph{bounded}, so $\partial_1\sO$ is non-empty and $u\leq 0$ on $\partial_1\sO$ by hypothesis. Let
$$
M := \sup_{\sO}u = \sup_{\sO}u^*,
$$
noting that $M<\infty$ by our hypothesis on $u$.
%TODO We may assume WLOG that M > 0, as otherwise we are done
Here, $u^*:\bar\sO\to\RR$ denotes the upper semicontinuous envelope of $u$ and necessarily obeys $u^* \geq u$ on $\underline\sO\cup\partial_1\sO$ and $u^* = u$ on $\sO\cup\partial_1\sO$. We consider
%COMMENT Neater way to handle this and avoid awkward "closure definition"?
$$
\bar\sO_M := \{x\in\bar\sO: u^*(x)=M\} \quad\hbox{and}\quad \underline\sO_M := \underline\sO\cap \bar\sO_M = \{x\in\underline\sO: u^*(x)=M\}.
$$
Since $\bar\sO$ is compact and $u^*$ is upper semicontinuous on $\bar\sO$, then $\bar\sO_M$ is non-empty.
%COMMENT We only know M=u(x_0) some x_0 in \bar\sO when \sO is bounded!
%COMMENT Closure of the empty set cannot be non-empty, so definition needed!
If $\underline\sO_M$ is empty, then we must have $\bar\sO_M = \{x\in\overline{\partial_1\sO}:u^*(x)=M\}$ and thus
$$
M = \sup_{\sO}u = \sup_{\partial_1\sO}u \leq 0,
$$
where the final inequality follows from our hypothesis on $u$ on $\partial_1\sO$.
%TODO And this contradicts our assumption that M>0
On the other hand, if $\underline\sO_M$ is non-empty, then $\underline\sO_M=\underline\sO$ by the strong maximum principle for continuous subsolutions (Property \eqref{item:elliptic_strong_maximum_principle}).
%TODO Since M > 0 (need at least M \geq 0)
But then $u\equiv M$ is constant on $\sO$ and thus on $\sO\cup \partial_1\sO$ (since $u$ belongs to $C(\sO\cup \partial_1\sO)$ by hypothesis). Since $\partial_1\sO$ is non-empty, then $M = \sup_\sO u=\sup_{\partial_1\sO} u\leq 0$.

When $\sO$ is unbounded, we adapt the proof of Lemma \ref{lem:Elliptic_weak_maximum_principle_extension_unbounded_domain}. Define
\begin{equation}
\label{eq:Defn_vzero}
v_0(x) := 1+|x|^2, \quad\forall\, x \in\HH,
\end{equation}
and observe that
\begin{align*}
Av_0(x) &= -2\tr (x_da(x)) - 2\langle b(x),x\rangle + c(x)\left(1+|x|^2\right)
\\
&\geq -2K\left(1+|x|^2\right) + c_0\left(1+|x|^2\right) \quad\hbox{(by \eqref{eq:Positive_lower_bound_c_domain} and \eqref{eq:Quadratic_growth})}
\\
&\geq 0, \quad\forall\, x\in\HH \quad\hbox{(by \eqref{eq:Positive_lower_bound_c_geq_2K})},
\end{align*}
and thus
\begin{equation}
\label{eq:A_vzero_geq_zero}
Av_0 \geq 0 \quad\hbox{on } \HH.
\end{equation}
Suppose $\delta>0$. We claim that
\begin{equation}
\label{eq:Defn_w}
w := u - \delta v_0
\end{equation}
is a continuous subsolution to the elliptic equation \eqref{eq:Elliptic_equation}, when $f=0$, in the sense of Definition \ref{defn:Continuous_super_sub_solution_bvp}. If $u$ were a smooth subsolution, in the sense of having the regularity properties in Definition \ref{defn:Weak_max_principle_property} and obeying $Au \leq 0$ on $\sO$, then
$$
Aw = A\left(u - \delta v_0\right) = Au - \delta Av_0 \leq 0 \quad\hbox{on }\sO,
$$
and thus $w$ would be a smooth subsolution for \eqref{eq:Elliptic_equation} with $f=0$. When $u$ is only a continuous subsolution, suppose $U=B\Subset\sO$ or $U=B^+$, a half-ball with center in $\partial_0\sO$, and that $\bar w \in C^2(U)\cap C^1(\underline U)$ obeys \eqref{eq:WMP_C2s_regularity}, \eqref{eq:WMP_xdsecond_order_derivatives_zero_boundary}, and $\inf_U\bar w > -\infty$, and
$$
A\bar w \geq 0  \quad\hbox{on } U, \quad \bar w \geq w \quad\hbox{on } \partial_1U.
$$
We claim that $\bar w \geq w$ on $U$. To see this, observe that
$$
A\left(\bar w + \delta v_0\right) \geq 0 \quad\hbox{on } U\quad\hbox{(by \eqref{eq:A_vzero_geq_zero} and the definition of $\bar w$)}.
$$
Moreover, by \eqref{eq:Defn_w},
$$
\bar w + \delta v_0 \geq u \quad\hbox{on } \partial_1U,
$$
and because $u$ is a continuous subsolution in the sense of Definition \ref{defn:Continuous_super_sub_solution_bvp} with $f=0$, we must have
$$
\bar w + \delta v_0 \geq u \quad\hbox{on } U,
$$
that is,
$$
\bar w \geq u - \delta v_0 = w \quad\hbox{on } U.
$$
Therefore, $w$ is a continuous subsolution to the elliptic equation \eqref{eq:Elliptic_equation} with $f=0$ in the sense of Definition \ref{defn:Continuous_super_sub_solution_bvp}.

Because $\sup_\sO u < \infty$ by hypothesis, we must have
$$
w \leq 0 \quad\hbox{on }\sO\cap\partial_1B_R^+,
$$
by \eqref{eq:Defn_w} for all large enough $R>0$ while our hypothesis $u \leq 0$ on $\partial_1\sO$ implies that
$$
w = u - \delta v_0 \leq 0 \quad\hbox{on }B_R^+\cap\partial_1\sO,
$$
and thus, since $\partial_1(\sO\cap B_R^+) = (\bar\sO\cap\partial_1B_R^+) \cup (\bar B_R^+\cap\partial_1\sO)$,
$$
w \leq 0 \quad\hbox{on }\partial_1(\sO\cap B_R^+).
$$
Consequently, noting that $\sO\cap B_R^+$ is connected for large enough $R$ by hypothesis, Property \eqref{item:elliptic_weak_maximum_principle} for the case of bounded subdomains of $\HH$ (namely, $\sO\cap B_R^+$ in this case) implies that $w \leq 0$ on $\sO\cap B_R^+$, for any sufficiently large $R>0$. Thus, $w = u - \delta v_0 \leq 0$ on $\sO$ for all $\delta>0$ and taking the limit as $\delta \downarrow 0$, we obtain
$$
u \leq 0  \quad\hbox{on } \sO,
$$
as desired. This completes the proof of Property \eqref{item:elliptic_weak_maximum_principle} for the case of unbounded domains.

For \eqref{item:elliptic_strong_comparison_principle}, observe that $u-v$ is a subsolution to \eqref{eq:Elliptic_equation}, \eqref{eq:Elliptic_boundary_condition} with $f = 0$ and $g = 0$. The weak maximum principle for continuous subsolutions (Property \eqref{item:elliptic_weak_maximum_principle}) implies that $u-v\leq 0$ on $\sO$. The strong maximum principle for continuous subsolutions (Property \eqref{item:elliptic_strong_maximum_principle}) implies that if $\sup_\sO(u-v)=(u-v)(x_0)=0$ for some $x_0\in\underline\sO$, then $u-v\equiv 0$ throughout $\sO$. Therefore, if $u-v\not\equiv 0$ throughout $\sO$, we must have $u-v<0$ on $\sO$.

For \eqref{item:elliptic_subsolution_harmonic_lift}, observe that $\hat u\in C(\sO)$ by construction. Let $U'\subset\sO$ denote an arbitrary ball $B'\Subset\sO$ or half-ball ${B'}^+\Subset\underline\sO$ with center in $\partial_0\sO$ and suppose $w\in C^2(U')\cap C^1(\underline U')$ obeys \eqref{eq:WMP_C2s_regularity}, \eqref{eq:WMP_xdsecond_order_derivatives_zero_boundary}, and $\inf_{U'}w>-\infty$, and
$$
Aw \geq f\quad\hbox{on }U', \quad w \geq \hat u \quad\hbox{on }\partial_1U'.
$$
According to Definition \ref{defn:Continuous_super_sub_solution_bvp}, it remains for us to show that $\hat u \leq w$ on $U'$. Since $u$ is a continuous subsolution to \eqref{eq:Elliptic_equation} on $\sO$, then $u\leq \bar u$ on $U$, thus $u\leq\hat u$ on $\sO$ and, in particular, $u \leq \hat u$ on $U'$. Because $u$ is a continuous subsolution to \eqref{eq:Elliptic_equation} on $\sO$ and $u\leq \hat u \leq w$ on $\partial_1U'$, we must have $u\leq w$ on $U'$ and hence $\hat u \leq w$ in $U'\less U$ since $\hat u = u$ in $U'\less U$. Furthermore, $\partial_1(U\cap U') = (U\cap\partial_1U') \cup (U'\cap\partial_1U)$ and $U'\cap\partial_1U \subset U'\less U$, so
$$
\hat u \leq w \quad\hbox{on }\partial_1(U\cap U').
$$
Moreover, as $A\hat u = A\bar u = f$ on $U$, then
$$
A(\hat u - w) \leq 0\quad\hbox{on }U\cap U',
$$
so the weak maximum principle implies that $\hat u \leq w$ in $U\cap U'$. Consequently $\hat u \leq w$ on $U'$ and therefore $\hat u$ is a continuous subsolution to \eqref{eq:Elliptic_equation} on $\sO$.

For \eqref{item:elliptic_max_set_subsolutions}, let $U$ denote an arbitrary ball $B\Subset\sO$ or half-ball $B^+\Subset\underline\sO$ with center in $\partial_0\sO$ and suppose that $\bar u\in C^2(U)\cap C^1(\underline U)$ obeys \eqref{eq:WMP_C2s_regularity}, \eqref{eq:WMP_xdsecond_order_derivatives_zero_boundary}, and $\inf_U\bar u > -\infty$, and
$$
A\bar u \geq f\quad\hbox{on }U, \quad \bar u \geq u \quad\hbox{on }\partial_1U.
$$
Therefore, $u_i \leq \bar u$ on $\partial_1U$ for $i=1,\ldots,N$ and Definition \ref{defn:Continuous_super_sub_solution_bvp} implies that $u_i \leq \bar u$ on $U$ for $i=1,\ldots,N$, so we must have $u \leq \bar u$ on $U$.
\end{proof}

We can provide a priori estimates for continuous supersolutions and subsolutions via the

\begin{prop}[Maximum principle estimates for continuous subsolutions and supersolutions]
\label{prop:Continuous_super_sub_solution_bvp_maximum_principle_estimates}
\cite[Propositions 2.19 and 3.18 and Remark 3.19]{Feehan_maximumprinciple_v1}
Let $\sO\subset\HH$ be a domain, $\alpha\in(0,1)$, and let $A$ in \eqref{eq:Generator} have the strong maximum principle property on $\underline\sO$ in the sense of Definition \ref{defn:Strong_max_principle_property}, obey Hypothesis \ref{hyp:Dirichlet_problem_local_solvability} for all balls $B\Subset\sO$ and half-balls $B^+\Subset\underline\sO$ with centers in $\partial_0\sO$ (Definition \ref{defn:Dirichlet_problem_local_solvability}), and have $c$ obeying \eqref{eq:Nonnegative_c}. Let $f\in C(\underline\sO)$ and $g\in C(\partial_1\sO)$. Suppose $u$ is a continuous subsolution and $v$ a continuous supersolution to the boundary value problem \eqref{eq:Elliptic_equation}, \eqref{eq:Elliptic_boundary_condition} for $f$ and $g$ in the sense of Definition \ref{defn:Continuous_super_sub_solution_bvp}.
\begin{enumerate}
\item\label{item:Continuous_subsolution_f_leq_zero} If $f\leq 0$ on $\sO$ and $u$ is a continuous subsolution for $f$ and $g$ and $\sup_\sO u < \infty$, then
$$
u\leq 0 \vee \sup_{\partial_1\sO}g \quad\hbox{on }\sO.
$$
\item\label{item:Continuous_subsolution_f_arb_sign} If $f$ has arbitrary sign and $u$ is a continuous subsolution for $f$ and $g$ and $\sup_\sO u < \infty$ but, in addition, there is a constant $c_0>0$ such that $c$ obeys \eqref{eq:Positive_lower_bound_c_domain}, then
$$
u\leq 0 \vee \frac{1}{c_0}\sup_\sO f \vee \sup_{\partial_1\sO}g \quad\hbox{on }\sO.
$$
\item\label{item:Continuous_supersolution_f_geq_zero} If $f\geq 0$ on $\sO$ and $v$ is a continuous supersolution for $f$ and $g$ and $\inf_\sO v > -\infty$, then
$$
v\geq 0 \wedge \inf_{\partial_1\sO}g \quad\hbox{on }\sO.
$$
\item\label{item:Continuous_supersolution_f_arb_sign} If $f$ has arbitrary sign, $v$ is a continuous supersolution for $f$ and $g$ and $\inf_\sO v > -\infty$, and $c$ obeys \eqref{eq:Positive_lower_bound_c_domain}, then
$$
v\geq 0 \wedge \frac{1}{c_0}\inf_ \sO f \wedge \inf_{\partial_1\sO}g \quad\hbox{on }\sO.
$$
\end{enumerate}
The terms $\sup_{\partial_1\sO}g$ and $\inf_{\partial_1\sO}g$ in the preceding items are omitted when $\partial_1\sO$ is empty.
\end{prop}

Let $\sS_{f,g}^-$ (respectively, $\sS_{f,g}^+$) denote the set of all subsolutions (respectively, supersolutions) to the boundary value problem \eqref{eq:Elliptic_equation}, \eqref{eq:Elliptic_boundary_condition} defined by $f$ and $g$ in the sense of Definition \ref{defn:Continuous_super_sub_solution_bvp}.

We have the following analogue of \cite[Theorems 6.11 \& 6.13]{GilbargTrudinger}.

\begin{thm}[A Perron method for existence of solutions to a degenerate-elliptic boundary value problem]
\label{thm:Perron_elliptic_bvp_solution}
Let $\sO\subseteqq\HH$ be a possibly unbounded domain, $\alpha\in(0,1)$, and $A$ be as in \eqref{eq:Generator} with coefficients belonging to $C^\alpha_s(\underline\sO)$, having the strong maximum principle property on $\underline\sO$ in the sense of Definition \ref{defn:Strong_max_principle_property}, local solvability of the partial Dirichlet problem in $\underline\sO$ in the sense of Definition \ref{defn:Dirichlet_problem_local_solvability}, and coefficient $a(x)$ obeying \eqref{eq:Strict_ellipticity_domain} together with $b^d$ obeying \eqref{eq:Positive_bd_boundary} or $c$ obeying \eqref{eq:Positive_c_boundary}. Moreover, the coefficients of $A$ should obey either
\begin{enumerate}
\item Condition \eqref{eq:Positive_lower_bound_c_domain}, or
\item Conditions \eqref{eq:Finite_upper_bound_add_domain} and \eqref{eq:Positive_lower_bound_bd_domain},
\end{enumerate}
and, in addition when $\sO$ is unbounded, \eqref{eq:Quadratic_growth}, \eqref{eq:Positive_lower_bound_c_geq_2K}, and $\sO$ should obey \eqref{eq:Connectedness_domain_and_ball}.
If $f\in C^\alpha_s(\underline\sO)\cap C_b(\sO)$ and $g\in C_b(\partial_1\sO)$, then the function
\begin{equation}
\label{eq:Perron_elliptic_equation_supsubsolution}
u(x) = \sup_{w\in \sS_{f,g}^-}w(x), \quad x\in\underline\sO,
\end{equation}
belongs to $C^{2+\alpha}_s(\underline\sO)\cap C_b(\sO)$ and is a solution to the elliptic equation \eqref{eq:Elliptic_equation}. Furthermore, if each point of $\partial_1\sO$ is \emph{regular} with respect to $A$, $f$, and $g$ in the sense of Definition \ref{defn:Regular_boundary_point_elliptic_equation}, then in addition $u$ belongs to $C(\sO\cup\partial_1\sO)$ and $u$ obeys the boundary condition \eqref{eq:Elliptic_boundary_condition}.
\end{thm}

\begin{rmk}[Perron solution as the pointwise infimum of supersolutions]
\label{rmk:Perron_elliptic_bvp_solution}
A argument which is symmetric to the proof of Theorem \ref{thm:Perron_elliptic_bvp_solution} shows that
\begin{equation}
\label{eq:Perron_elliptic_equation_infsupsolution}
v(x) = \inf_{w\in \sS_{f,g}^+}w(x), \quad x\in\underline\sO,
\end{equation}
is also a solution to the elliptic equation \eqref{eq:Elliptic_equation} and thus coincides with $u$ in \eqref{eq:Perron_elliptic_equation_infsupsolution} by the weak maximum principle (Theorems \ref{thm:Elliptic_weak_maximum_principle_bounded_domain} or \ref{thm:Elliptic_weak_maximum_principle_unbounded_domain}) when $u=v$ on $\partial_1\sO$.\qed
\end{rmk}

\begin{proof}
The proof that $u$ in \eqref{eq:Perron_elliptic_equation_supsubsolution} is a solution to \eqref{eq:Elliptic_equation} is similar to that of \cite[Theorem 6.11]{GilbargTrudinger} in the case of a strictly elliptic operator, or \cite[Theorem 2.12]{GilbargTrudinger} in the case of the Laplace operator, with the
\begin{enumerate}
\item Compactness of solutions provided by the a priori interior Schauder estimate in our Theorem \ref{thm:Elliptic_apriori_Schauder_interior_domain}, augmenting \cite[Corollary 6.3]{GilbargTrudinger} (or \cite[Theorem 2.10]{GilbargTrudinger} in the case of the Laplace operator);
\item Strong and weak maximum principles for subsolutions provided by Definition \ref{defn:Strong_max_principle_property}, Lemma \ref{lem:Elliptic_strong_implies_weak_maximum_principle_property_bounded_domain} and Corollary \ref{cor:Elliptic_strong_implies_weak_maximum_principle_property_unbounded_domain}
(replacing the roles of \cite[Corollary 3.2 \& Theorem 3.5]{GilbargTrudinger}, or \cite[Theorems 2.2 \& 2.3]{GilbargTrudinger} in the case of the Laplace operator);
\item Local solvability, provided by Hypothesis \ref{hyp:Dirichlet_problem_local_solvability}, on balls $B\Subset\sO$ (replacing the roles of \cite[Lemma 6.10]{GilbargTrudinger}, or \cite[Theorem 2.6]{GilbargTrudinger} in the case of the Laplace operator) and half-balls $B^+\Subset\underline\sO$ with centers in $\partial_0\sO$.
\end{enumerate}
Theorem \ref{thm:Elliptic_equation_continuous_sub_supersolution_properties} (for continuous subsolutions and supersolutions) provides analogues of the remaining ingredients employed in the proof of \cite[Theorem 2.12]{GilbargTrudinger}.

We first check, by analogy with the construction in \cite[Equation (6.44)]{GilbargTrudinger}, that there exists at least one subsolution, $v^-$, so the set $\sS_{f,g}^-$ is non-empty, and one subsolution, $v^+$, which thus provides an upper bound for all subsolutions in $\sS_{f,g}^-$ by Theorem \ref{thm:Elliptic_equation_continuous_sub_supersolution_properties} \eqref{item:elliptic_strong_comparison_principle}.

If $\sO\subseteqq\HH$ is a possibly unbounded domain and the coefficient $c$ of $u$ in \eqref{eq:Generator} obeys \eqref{eq:Positive_lower_bound_c_domain}, so $c\geq c_0>0$ on $\sO$, and $u\in C(\sO)$ is a continuous subsolution to the boundary value problem \eqref{eq:Elliptic_equation}, \eqref{eq:Elliptic_boundary_condition} in the sense of Definition \ref{defn:Continuous_super_sub_solution_bvp}, then Proposition \ref{prop:Continuous_super_sub_solution_bvp_maximum_principle_estimates} \eqref{item:Continuous_subsolution_f_arb_sign} implies that
$$
u \leq M_1^+ \quad\hbox{on }\sO,
\quad\hbox{for } M_1^+ := 0\vee \frac{1}{c_0}\sup_\sO f \vee \sup_{\partial_1\sO} g,
$$
and so each $u \in \sS_{f,g}^-$ is bounded above by the constant $M_1^+$. If $\sO\subset\HH$ has finite $\height(\sO)\leq\nu$ and the coefficients $b^d$ and $c$ in \eqref{eq:Generator} obey \eqref{eq:Positive_lower_bound_bd_domain} and \eqref{eq:Nonnegative_c}, so $c\geq 0$ but $b^d\geq b_0>0$ on $\underline\sO$, and $u\in C(\sO)$ is a continuous subsolution to \eqref{eq:Elliptic_equation}, \eqref{eq:Elliptic_boundary_condition}, then Lemma \ref{lem:Elliptic_weak_maximum_principle_finite_height_domain}, Remark \ref{rmk:Elliptic_weak_maximum_principle_finite_height_domain}, and Proposition \ref{prop:Continuous_super_sub_solution_bvp_maximum_principle_estimates} \eqref{item:Continuous_subsolution_f_arb_sign} imply that (compare Corollary \ref{cor:Elliptic_weak_maximum_principle_finite_height_domain})
$$
u \leq M_0^+ \quad\hbox{on }\sO,
\quad\hbox{for } M_0^+ := e^{b_0\nu/2\Lambda}\left(0\vee \frac{4\Lambda}{b_0^2} \sup_\sO f \vee \sup_{\partial_1\sO} g\right),
$$
and so each $u \in \sS_{f,g}^-$ is bounded above by the constant $M_0^+$. We let $M^+$ denote $M_1^+$ or $M_0^+$, depending on whether $c\geq c_0>0$ or $c\geq 0$ on $\sO$.

If $\sO\subseteqq\HH$ is a possibly unbounded domain and the coefficient $c$ of $u$ in \eqref{eq:Generator} obeys \eqref{eq:Positive_lower_bound_c_domain}, so $c\geq c_0>0$ on $\underline\sO$, then Proposition \ref{prop:Continuous_super_sub_solution_bvp_maximum_principle_estimates} \eqref{item:Continuous_supersolution_f_arb_sign} implies that the constant function
$$
M_1^- := 0\wedge \frac{1}{c_0}\inf_\sO f \wedge \inf_{\partial_1\sO} g,
$$
is a continuous subsolution to the boundary value problem \eqref{eq:Elliptic_equation}, \eqref{eq:Elliptic_boundary_condition} in the sense of Definition \ref{defn:Continuous_super_sub_solution_bvp}. If $\sO\subset\HH$ has finite $\height(\sO)\leq\nu$ and the coefficients $b^d$ and $c$ in \eqref{eq:Generator} obey \eqref{eq:Positive_lower_bound_bd_domain} and \eqref{eq:Nonnegative_c}, so $c\geq 0$ but $b^d\geq b_0>0$ on $\underline\sO$, then Lemma \ref{lem:Elliptic_weak_maximum_principle_finite_height_domain}, Remark \ref{rmk:Elliptic_weak_maximum_principle_finite_height_domain}, and Proposition \ref{prop:Continuous_super_sub_solution_bvp_maximum_principle_estimates} \eqref{item:Continuous_supersolution_f_arb_sign} imply that (compare Corollary \ref{cor:Elliptic_weak_maximum_principle_finite_height_domain}) the constant function
$$
M_0^- := e^{b_0\nu/2\Lambda}\left(0\wedge \frac{4\Lambda}{b_0^2}\inf_\sO f \wedge \inf_{\partial_1\sO} g\right),
$$
is a continuous subsolution to the boundary value problem \eqref{eq:Elliptic_equation}, \eqref{eq:Elliptic_boundary_condition} in the sense of Definition \ref{defn:Continuous_super_sub_solution_bvp}. We let $M^-$ denote $M_1^-$ or $M_0^-$, depending on whether $c\geq c_0>0$ or $c\geq 0$ on $\sO$.

To show that $u$ in \eqref{eq:Perron_elliptic_equation_supsubsolution} is a solution to \eqref{eq:Elliptic_equation}, we adapt the proof of \cite[Theorem 2.12]{GilbargTrudinger}. Fix $x_0 \in \underline\sO$. By definition of $u$, there is a sequence $\{u_n\}_{n\in\NN}\subset\sS_{f,g}^-$ such that $u_n(x_0)\to u(x_0)$ as $n\to\infty$. By replacing $u_n$ with $u_n\vee M^-$, which we may do by Theorem \ref{thm:Elliptic_equation_continuous_sub_supersolution_properties} \eqref{item:elliptic_max_set_subsolutions}, and noting that $u_n\leq M^+$ on $\sO$ for all $n\in\NN$ by Theorem \ref{thm:Elliptic_equation_continuous_sub_supersolution_properties} \eqref{item:elliptic_strong_comparison_principle}, we may assume that the sequence $\{u_n\}_{n\in\NN}$ is bounded on $\sO$. Now choose $r>0$ small enough such that if $U=B_r(x_0)$ when $x_0\in\sO$ or $U=B_r^+(x_0)$ when $x_0\in\partial_0\sO$, then $U\Subset\underline\sO$ and define $\hat u_n\in \sS_{f,g}^-$ to be the ``harmonic lift'' of $u_n$ on $U$ according
to Theorem \ref{thm:Elliptic_equation_continuous_sub_supersolution_properties} \eqref{item:elliptic_subsolution_harmonic_lift}. Then $\hat u_n(x_0)\to u(x_0)$ as $n\to\infty$ since
$$
u_n(x_0) \leq \hat u_n(x_0) \leq u(x_0), \quad\forall\, n\in\NN,
$$
where the first inequality follows from Theorem \ref{thm:Elliptic_equation_continuous_sub_supersolution_properties} \eqref{item:elliptic_subsolution_harmonic_lift} and the second by definition of $u$ in \eqref{eq:Perron_elliptic_equation_supsubsolution}. We have
\begin{equation}
\label{eq:Elliptic_apriori_max_principle_estimate_ballorhalfball_sequence}
M^- \leq \hat u_n \leq M^+ \quad\hbox{on }U, \quad n\in\NN,
\end{equation}
and, for any open subset $U'\Subset\underline U$, we have the a priori interior Schauder estimate,
\begin{equation}
\label{eq:Elliptic_apriori_Schauder_interior_ballorhalfball_sequence}
\|\hat u_n\|_{C^{2+\alpha}_s(\bar U')} \leq C\left(\|f\|_{C^{\alpha}_s(\bar U)} + \|\hat u_n\|_{C(\bar U)}\right),
\end{equation}
for a positive constant $C$ depending at most on $b_0, d, d_0, r, \alpha, \lambda_0, \Lambda$ (by \cite[Corollary 6.3]{GilbargTrudinger} when $U=B_r(x_0)$ or Theorem \ref{thm:Elliptic_apriori_Schauder_interior_domain} when $U=B_r^+(x_0)$, where $\dist(\partial_1U',\partial_1U) \geq d_0>0$, together with Remark \ref{rmk:Elliptic_weak_maximum_principle_finite_height_domain} in the case $b^d=0$ but $c>0$ on $\partial_0\sO$). In particular, \eqref{eq:Elliptic_apriori_max_principle_estimate_ballorhalfball_sequence} and \eqref{eq:Elliptic_apriori_Schauder_interior_ballorhalfball_sequence} yield
a uniform bound on $\|\hat u_n\|_{C^\alpha_s(\bar U')}$ which is independent of $n\in\NN$ and thus a uniform bound on $\|\hat u_n\|_{C^{\alpha/2}(\bar U')}$ which is independent of $n\in\NN$ by \eqref{eq:DH_and_standard_Holder_relationships}. The Arzel\`a-Ascoli Theorem implies, after passing to a subsequence, that $\hat u_n\to v$ in $C(\bar U')$ as $n\to\infty$, for some $v\in C(\bar U')$, and in particular, the sequence $\hat u_n$ is Cauchy in $C(\bar U')$. But then \eqref{eq:Elliptic_apriori_Schauder_interior_ballorhalfball_sequence} implies, for any open subset $U''\Subset\underline U'$, that
$$
\|\hat u_n - \hat u_m\|_{C^{2+\alpha}_s(\bar U'')} \leq C\|\hat u_n - \hat u_m\|_{C(\bar U')}, \quad m, n \in \NN.
$$
Hence, the sequence $\hat u_n$ is Cauchy in $C^{2+\alpha}_s(\bar U'')$ and thus converges in $C^{2+\alpha}_s(\bar U'')$ to a limit $v\in C^{2+\alpha}_s(\bar U'')$ which necessarily coincides with the limit $v\in C(\bar U')$ already discovered. As $U''$ and $U'$ were arbitrary, we obtain $v\in C^{2+\alpha}_s(\underline U)$. Since $\hat u_n = \bar u_n$ on $U$ and thus by construction in Theorem \ref{thm:Elliptic_equation_continuous_sub_supersolution_properties} \eqref{item:elliptic_subsolution_harmonic_lift} obeys
$$
A\hat u_n = f \quad\hbox{on } U, \quad n\in\NN,
$$
we may take the limit in the preceding equation as $n\to\infty$ to give
$$
Av = f \quad\hbox{on } U.
$$
Clearly, $v \leq u$ on $U$ and $v(x_0) = u(x_0)$. To see that $v=u$ on $U$, suppose $v(z) < u(z)$ at some point $z\in U$. Then there exists a ``harmonic lift'' $\hat u \in \sS^-_{f,g}$ such that $v(z) < \bar u(z)$. Defining $w_n := \max\{\bar u, \hat u_n\}$ and its ``harmonic lifts'' $\hat w_n$ as in \eqref{eq:Elliptic_equation_harmonic_lift}, we obtain as before a subsequence of the sequence $\{w_n\}_{n\in\NN}$ which converges to a function $w\in C^{2+\alpha}_s(\underline U)$ and solves $Aw=f$ on $U$ and satisfies $v\leq w \leq u$ on $U$ and $v(x_0)=w(x_0)=u(x_0)$. Because $A$ has the strong maximum principle property on $\underline\sO$ in the sense of Definition \ref{defn:Strong_max_principle_property}, and $A(w-v)$ = 0 on $U$ with $(v-w)(x_0) = 0$ and $c \geq 0$ on $\underline U$, we must have $v\equiv w$ on $U$. This contradicts the definition of $\bar u$, since $v(z) < \bar u(z)$ while $\bar u(z) \leq w(z)$, and hence $u$ is a solution to \eqref{eq:Elliptic_equation} on $U$ and thus on $\sO$.

Finally, the proof that $u$ also belongs to $C(\sO\cup\partial_1\sO)$ and obeys the boundary condition \eqref{eq:Elliptic_boundary_condition} when each point of $\partial_1\sO$ is \emph{regular} with respect to $A$, $f$, and $g$ in the sense of Definition \ref{defn:Regular_boundary_point_elliptic_equation} is identical to that of \cite[Lemma 6.12]{GilbargTrudinger}.
\end{proof}

It remains to give the

\begin{proof}[Proof of Theorem \ref{thm:Existence_uniqueness_elliptic_Dirichlet}]
We first observe that $A$ has the strong maximum principle property on $\underline\sO$ in the sense of Definition \ref{defn:Strong_max_principle_property} thanks to Theorem \ref{thm:Elliptic_strong_maximum_principle}. Local solvability of the partial Dirichlet problem in $\underline\sO$, in the sense of Definition \ref{defn:Dirichlet_problem_local_solvability}, on balls $B\Subset\sO$ or half-balls $B^+\Subset\underline\sO$ centered on $\partial_0\sO$, is provided by \cite[Lemma 6.10]{GilbargTrudinger} and Theorem \ref{thm:Existence_uniqueness_elliptic_Dirichlet_halfball} (taking note of Remark \ref{rmk:Boundary_property_bd}), respectively. The conclusion now follows from Theorem \ref{thm:Perron_elliptic_bvp_solution}.
\end{proof}

\subsection{A Perron method for existence of solutions to a degenerate-elliptic obstacle problem}
\label{sec:Perron_obstacle_problem}
The Sobolev Embedding Theorem implies that $W^{2,p}(V) \subset C_b(V)$, when $2<p<\infty$ and $V\subset\RR^d$ is a domain which obeys an interior cone condition \cite[\S 4.3 \& Theorem 5.4 (C)]{Adams_1975}, while $W^{2,p}(V) \hookrightarrow C^{1,\gamma}(\bar V)$ for any $\gamma\in(0,1-d/p]$, when $d<p<\infty$, and $V\subset\RR^d$ is a domain whose boundary $\partial V$ has the strong local Lipschitz property \cite[\S 4.5 \& Theorem 5.4 ($\hbox{C}'$)]{Adams_1975}. To formulate our results for uniqueness and existence of solutions, we begin with the

\begin{defn}[Solution to an obstacle problem]
\label{defn:Solution_obstacle_problem}
Let $\sO\subseteqq\HH$ be a domain, $2<p<\infty$, and $A$ be as in \eqref{eq:Generator}. Given $f\in C(\underline\sO)$ and $\psi\in C(\underline\sO)$, we call $u\in W^{2,p}_{\loc}(\sO)\cap C(\underline\sO)$ a \emph{solution} to the obstacle problem if $u$ obeys \eqref{eq:Elliptic_obstacle_problem}, that is,
$$
\min\{Au-f,u-\psi\} = 0 \quad \hbox{a.e. on }\sO,
$$
and if $\Omega = \{x\in\sO: (u-\psi)(x)>0\}$, then $u\in C^2(\Omega)\cap C^1(\underline\Omega)$ and obeys \eqref{eq:WMP_C2s_regularity}, \eqref{eq:WMP_xdsecond_order_derivatives_zero_boundary} on $\Omega$.

Furthermore, given $g\in C(\partial_1\sO)$ and $\psi$ also belonging to $C(\sO\cup\partial_1\sO)$ and obeying the compatibility condition \eqref{eq:Boundarydata_obstacle_compatibility}, that is, $\psi\leq g$ on $\partial_1\sO$, we call $u$ a \emph{solution} to the obstacle problem with partial Dirichlet boundary condition if in addition $u$ belongs to $C(\sO\cup\partial_1\sO)$ and obeys  \eqref{eq:Elliptic_boundary_condition}, that is,
$$
u = g \quad\hbox{on } \partial_1\sO.
$$
\end{defn}

\begin{rmk}[Implied regularity for the obstacle function]
While not explicitly stated in Definition \ref{defn:Solution_obstacle_problem}, there is an implicit regularity condition, $\psi \in W^{2,p}_{\loc}(\sO\less\Omega)$, on the coincidence set (or ``exercise region''), $\sO\less\Omega = \{x\in\sO: (u-\psi)(x)=0\}$, where
%TODO Duplication
$\Omega$ is as in \eqref{eq:Continuation_region}. While we could impose the stronger condition, $\psi \in W^{2,p}_{\loc}(\sO)$, it is convenient not to do this initially, in part because there are important examples of obstacle functions which only belong to $W^{1,\infty}_{\loc}(\sO)$, although they may have higher regularity on the complement of subsets of measure zero in $\sO$. The requirement that $u$ belong to $W^{2,p}_{\loc}(\sO)$ is imposed in order to be consistent with \cite[Theorem 1.3.2]{Friedman_1982}.
%TODO Why this much smoothness? Does this make any sense when we want smooth u?
However, if $a^{ij} \in C^{1,1}(\underline\sO)$ and $b^i \in C^{0,1}(\underline\sO)$, we could instead have asked that $u\in C(\underline\sO)$ obey \eqref{eq:Elliptic_obstacle_problem} in the sense of distributions\footnote{By
integrating by parts in the variational inequality \cite[Equation (1.3.24)]{Friedman_1982} associated with \eqref{eq:Elliptic_obstacle_problem}, so derivatives are only applied to the test function $\varphi$.} with respect to test functions $\varphi\in C^\infty_0(\underline\sO)$ with $\varphi\geq \psi$ on $\sO$. In this second approach there is no implied regularity for $\psi\in C(\underline\sO)$.\qed
\end{rmk}

In order to motivate our definition of a continuous supersolution to the obstacle problem, we shall first give the simpler\footnote{Definition \ref{defn:Smooth_supersolution_obstacle_problem} is a refinement of \cite[Definition 3.1]{Feehan_maximumprinciple_v1}.}
%COMMENT Should we say that a smooth supersolution is also a continuous supersolution?

\begin{defn}[Smooth supersolution to an obstacle problem]
\label{defn:Smooth_supersolution_obstacle_problem}
Let $\sO\subseteqq\HH$ be a domain and $A$ be as in \eqref{eq:Generator}. Given $f\in C(\underline\sO)$ and $\psi\in C(\underline\sO)$, we call $v\in C^2(\sO)\cap C^1(\underline\sO)$ obeying \eqref{eq:WMP_C2s_regularity}, \eqref{eq:WMP_xdsecond_order_derivatives_zero_boundary}
a \emph{smooth supersolution} to the obstacle problem \eqref{eq:Elliptic_obstacle_problem} if $v$ satisfies
$$
v \geq \psi \quad\hbox{on }\sO,
$$
and is a supersolution to \eqref{eq:Elliptic_equation}, that is,
$$
Av \geq f \quad \hbox{on }\sO.
$$
Furthermore, given $g\in C(\partial_1\sO)$ and $\psi$ also belonging to $C(\sO\cup\partial_1\sO)$ and obeying \eqref{eq:Boundarydata_obstacle_compatibility}, that is, $\psi\leq g$ on $\partial_1\sO$, we call $v$ a \emph{smooth supersolution} to the obstacle problem \eqref{eq:Elliptic_obstacle_problem} with partial Dirichlet boundary condition \eqref{eq:Elliptic_boundary_condition} if $v$ is a smooth supersolution to the obstacle problem \eqref{eq:Elliptic_obstacle_problem} and in addition $v$ belongs to $C(U\cup\partial_1\sO)$ and is a supersolution to \eqref{eq:Elliptic_boundary_condition}, that is,
$$
v \geq g \quad\hbox{on } \partial_1\sO.
$$
\end{defn}

We now give the definition of a continuous supersolution\footnote{Definition \ref{defn:Continuous_supersolution_obstacle_problem} is a refinement of \cite[Definition 3.9]{Feehan_maximumprinciple_v1}.} to an obstacle problem which we use in our comparison and existence theorems.

\begin{defn}[Continuous supersolution to an obstacle problem]
\label{defn:Continuous_supersolution_obstacle_problem}
Let $\sO\subseteqq\HH$ be a domain and $A$ be as in \eqref{eq:Generator}. Given $f\in C(\underline\sO)$ and $\psi\in C(\underline\sO)$, we call $v:\underline\sO\to\RR$ a \emph{continuous supersolution} to the obstacle problem \eqref{eq:Elliptic_obstacle_problem} if $v$ is continuous on $\sO$, locally bounded on $\underline\sO$, satisfies
$$
v \geq \psi \quad\hbox{on }\underline\sO,
$$
and is a \emph{continuous supersolution} to the elliptic equation \eqref{eq:Elliptic_equation} in the sense of the following refinement of Definition \ref{defn:Continuous_super_sub_solution_bvp}: for every open ball $B\Subset\sO$ or, for $\Omega=\{x\in\sO: (v-\psi)(x)>0\}$, every half-ball $B^+\Subset\underline\Omega$ with center in $\partial_0\sO$ and for every $\underline v\in C^2(U)\cap C^1(\underline U)$, with $U=B$ or $B^+$, obeying \eqref{eq:WMP_C2s_regularity}, \eqref{eq:WMP_xdsecond_order_derivatives_zero_boundary} when $U=B^+$, and $\sup_U \underline v < \infty$, and
$$
\begin{cases}
A\underline v \leq f &\hbox{on } U,
\\
\underline v \leq v &\hbox{on } \partial_1U,
\end{cases}
$$
we then have\footnote{Note that we do \emph{not} assert that $\underline v \geq \psi$ on $U$.}
$$
v \geq \underline v \quad \hbox{on } U.
$$
Furthermore, given $g\in C(\partial_1\sO)$ and $\psi$ also belonging to $C(\sO\cup\partial_1\sO)$ and obeying \eqref{eq:Boundarydata_obstacle_compatibility}, that is, $\psi\leq g$ on $\partial_1\sO$, we call $v\in C(\sO\cup\partial_1\sO)$ a \emph{continuous supersolution} to the obstacle problem \eqref{eq:Elliptic_obstacle_problem}  with partial Dirichlet boundary condition \eqref{eq:Elliptic_boundary_condition} if $v$ is a continuous supersolution to \eqref{eq:Elliptic_equation} and \eqref{eq:Elliptic_boundary_condition}, so
$$
v \geq g \quad\hbox{on } \partial_1\sO.
$$.
\end{defn}

\begin{rmk}[Definition of continuous supersolution to a variational inequality associated with an obstacle problem]
Definition \ref{defn:Continuous_supersolution_obstacle_problem} may be compared with definitions of a supersolution to a variational inequality associated with an obstacle problem in \cite[pp. 108, 118, 215, and 236]{Rodrigues_1987}; see \cite[Theorems 4.5.7, 4.7.4, 7.3.3, and 7.3.5]{Rodrigues_1987} for related results and applications. Note that we do \emph{not} assume that $v$ is continuous on $\underline\sO$.\qed
\end{rmk}

The \emph{continuation region} (or \emph{non-coincidence set}) for a solution $u\in W^{2,p}_{\loc}(\sO)\cap C(\underline\sO)$ to the obstacle problem \eqref{eq:Elliptic_obstacle_problem} with obstacle $\psi\in C(\underline\sO)$ is defined by
\begin{equation}
\label{eq:Continuation_region}
\Omega := \{x\in\sO: (u-\psi)(x) > 0\},
\end{equation}
so $\Omega\subset\HH$ is an open subset, since both $u$ and $\psi$ are continuous on $\sO$.

\begin{rmk}[Continuation region defined by the Perron function for the obstacle problem]
\label{rmk:Continuation_region_Perron_solution_elliptic_obstacle_problem}
If $u$ is initially defined by the Perron formula \eqref{eq:Perron_elliptic_obstacle_problem_solution}, then $u$ is \emph{upper} semicontinuous on $\sO$, since each supersolution is continuous on $\sO$, and so its continuation region is \emph{not} necessarily an open subset of $\HH$, for which we would need $u$ to be \emph{lower} semicontinuous.\qed
\end{rmk}

\begin{lem}[A solution to the obstacle problem is a continuous supersolution]
\label{lem:Solution_obstacle_problem_is_continuous_supersolution}
Let $\sO\subseteqq\HH$ be a domain, 
%TODO p > d or p \geq d?
$d < p<\infty$, and $A$ as in \eqref{eq:Generator} have the weak maximum principle property in the sense of Definition \ref{defn:Weak_max_principle_property} and coefficients $a^{ij}, \ b^i \in L^\infty_{\loc}(\sO)$ and obeying \eqref{eq:Strict_ellipticity_domain}. If $f \in C(\underline\sO)$, and $\psi \in C(\underline\sO)$, and $u\in W^{2,p}_{\loc}(\sO)\cap C(\underline\sO)$ is a solution to \eqref{eq:Elliptic_obstacle_problem} in the sense of Definition \ref{defn:Solution_obstacle_problem}, then $u$ is a continuous supersolution in the sense of Definition \ref{defn:Continuous_supersolution_obstacle_problem}.
\end{lem}

\begin{proof}
Clearly, $u\geq \psi$ on $\sO$ and $u \in C(\underline\sO)$, as required by Definition \ref{defn:Continuous_supersolution_obstacle_problem}. Because $u$ obeys \eqref{eq:Elliptic_obstacle_problem}, we have $Au\geq f$ a.e. on $\sO$, while on the continuation region $\Omega$ given by \eqref{eq:Continuation_region}, we know that $u$ belongs to $C^2(\Omega)\cap C^1(\underline\Omega)$ and obeys $Au=f$ and  \eqref{eq:WMP_C2s_regularity}, \eqref{eq:WMP_xdsecond_order_derivatives_zero_boundary} on $\Omega$.

Let $U$ denote a ball $B\Subset\sO$ or half-ball $B^+\Subset\underline\Omega$ with center in $\partial_0\sO$ and suppose $\underline v\in C^2(U)\cap C^1(\underline U)$, obeys \eqref{eq:WMP_C2s_regularity}, \eqref{eq:WMP_xdsecond_order_derivatives_zero_boundary} when $U=B^+$, and
$$
\begin{cases}
A\underline v \leq f &\hbox{on } U,
\\
\underline v \leq v &\hbox{on } \partial_1U.
\end{cases}
$$
If $U=B\Subset\sO$, then $u-\underline v$ belongs to $W^{2,p}(B)\cap C(\bar B)$ and obeys $A(u-\underline v) \geq 0$ a.e. on $B$ and $u-\underline v\leq 0$ on $\partial B$. The condition \cite[p. 220]{GilbargTrudinger} on the coefficients $a^{ij}$ of $A$ is obeyed since $\det a(x)>\lambda_0^d$ for $x\in\sO$ by \eqref{eq:Strict_ellipticity_domain} and $a^{ij}\in L^\infty_{\loc}(\sO)$ by hypothesis, so $\det a(x)$, for $x\in \bar B$, is uniformly bounded above and below by positive constants; the condition \cite[Equation (9.3)]{GilbargTrudinger} on $f$ and the coefficients $b^i$ of $A$ is also obeyed since $f, \ b^i \in L^\infty_{\loc}(\sO)$; finally, $p\geq d$ by hypothesis. Hence, the Aleksandrov weak maximum principle \cite[Theorem 9.1]{GilbargTrudinger} for functions in $W^{2,d}(B)\cap C(\bar B)$ implies that
$$
v\geq \underline v \quad \hbox{on } B.
$$
If $U=B^+\Subset\underline\Omega$ with center in $\partial_0\sO$, then $u-\underline v$ belongs to $C^2(\Omega)\cap C^1(\underline\Omega)$, obeys \eqref{eq:WMP_C2s_regularity}, \eqref{eq:WMP_xdsecond_order_derivatives_zero_boundary} on $B^+$, and obeys $A(u-\underline v) \geq 0$ on $B^+$, and $\sup_{B^+}(u-\underline v)<\infty$, and $u-\underline v\leq 0$ on $\partial_1B$, so the weak maximum principle property of $A$ (Definition \ref{defn:Weak_max_principle_property}) implies that
$$
v\geq \underline v \quad \hbox{on } B^+.
$$
Therefore, $u$ is a continuous supersolution in the sense of Definition \ref{defn:Continuous_supersolution_obstacle_problem}.
\end{proof}

\begin{rmk}[Compatibility for the source and obstacle functions]
It is tempting but not necessary to impose the compatibility condition,
\begin{equation}
\label{eq:Source_and_obstacle_function_compatibility}
A\psi \geq f \quad\hbox{a.e. on } \sO,
\end{equation}
but because $u$ obeys the equation \eqref{eq:Elliptic_obstacle_problem}, we necessarily have that $A\psi = Au \geq f$ a.e. on the subset $\{x\in\sO:(u-\psi)(x)=0\}$, which is sufficient. On the other hand, if $A\psi\leq f$ a.e. on $\sO$, then a solution $u$ to $Au=f$ on $\sO$ and $u\geq\psi$ on $\partial_1\sO$ leads to $A(u-\psi)\geq 0$ a.e. on $\partial_1\sO$ and so the maximum principles employed in the proof of Lemma \ref{lem:Solution_obstacle_problem_is_continuous_supersolution} (on bounded domains) imply that $u\geq\psi$ on $\sO$ and thus $u$ is a solution to \eqref{eq:Elliptic_obstacle_problem}. A condition $A\psi>f$ a.e. on some open ball in $\sO$ ensures that the obstacle problem \eqref{eq:Elliptic_obstacle_problem} is a non-trivial generalization of the elliptic equation \eqref{eq:Elliptic_equation}. \qed
\end{rmk}

We now come to the crucial

\begin{thm}[Comparison principle for solutions and continuous supersolutions to the obstacle problem]
\label{thm:Elliptic_obstacle_problem_comparison_principle}
Let $\sO\subseteqq\HH$ be a possibly unbounded domain, and $2<p<\infty$ and $\alpha\in (0,1)$, and $A$ be as in \eqref{eq:Generator} with coefficients belonging to $C^\alpha_s(\underline\sO)$, having the strong maximum principle property on $\underline\sO$ in the sense of Definition \ref{defn:Strong_max_principle_property} and local solvability of the partial Dirichlet problem in $\underline\sO$ in the sense of Definition \ref{defn:Dirichlet_problem_local_solvability}. If $\sO$ is unbounded, assume in addition that the coefficients of $A$ obey \eqref{eq:Positive_lower_bound_c_domain}, \eqref{eq:Quadratic_growth}, and \eqref{eq:Positive_lower_bound_c_geq_2K} and that $\sO$ obeys \eqref{eq:Connectedness_domain_and_ball}.
Let $f\in C^\alpha_s(\underline\sO)\cap C_b(\sO)$ and $\psi\in C(\underline\sO)$ with $\sup_\sO\psi<\infty$. Suppose that $u$ is a solution to the obstacle problem in the sense of Definition \ref{defn:Solution_obstacle_problem} with $\sup_\sO u<\infty$ and $v$ is a continuous supersolution to the obstacle problem in the sense of Definition \ref{defn:Continuous_supersolution_obstacle_problem} with $\inf_\sO v>-\infty$. If $u$, $v$, and $\psi$ also belong to $C(\sO\cup\partial_1\sO)$ and\footnote{Our hypotheses allow for the possibility that $\sO=\HH$, in which case $\partial_1\sO=\emptyset$ and the boundary comparison condition is omitted.}
$u \leq v$ on $\partial_1\sO$, then
$$
u \leq v \quad\hbox{on }\sO.
$$
%TODO p > d or p \geq d?
In addition, suppose $p > d$ and that the coefficients $a^{ij}, \ b^i$ of $A$ belong to $L^\infty_{\loc}(\sO)$ and obey \eqref{eq:Strict_ellipticity_domain}. If $v$ is also a solution to the obstacle problem and $u = v$ on $\partial_1\sO$, then $u=v$ on $\sO$.
\end{thm}

\begin{proof}
We may assume without loss of generality that $\sO\cap\{u>\psi\}$ is non-empty, because if $u=\psi$ on $\sO$ then $u\leq v$ on $\sO$ by Definition \ref{defn:Continuous_supersolution_obstacle_problem}. We first consider the case that $\sO$ is \emph{bounded}, so $\partial_1\sO$ is non-empty and $u-v\leq 0$ on $\partial_1\sO$. Let
$$
M := \sup_{\sO}(u-v) = \sup_{\sO}(u-v_*) = \sup_{\sO}(u^*-v_*),
$$
noting that $M<\infty$ by our hypotheses on $u$ and $v$. Here $v_*:\bar\sO\to\RR$ denotes the lower semicontinuous envelope of $v$ and necessarily obeys $\psi \leq v_*\leq v$ on $\underline\sO\cup\partial_1\sO$ and $v_*=v$ on $\sO\cup\partial_1\sO$, while $u^*:\bar\sO\to\RR$ denotes the upper semicontinuous envelope of $u$ and necessarily obeys $u^* = u$ on $\underline\sO\cup\partial_1\sO$. We consider
$$
\bar\sO_M = \{x\in\bar\sO:(u^*-v_*)(x)=M\} \quad\hbox{and}\quad \underline\sO_M := \underline\sO\cap\bar\sO_M = \{x\in\underline\sO:(u-v_*)(x)=M\}.
$$
Since $\bar\sO$ is compact and $u^*-v_*$ is upper semicontinuous on $\bar\sO$, then $\bar\sO_M$ is non-empty. If $\underline\sO_M$ is empty, then we must have $\bar\sO_M = \{x\in\overline{\partial_1\sO}:(u^*-v_*)(x)=M\}$ and thus
$$
M = \sup_{\sO}(u-v) = \sup_{\partial_1\sO}(u-v) \leq 0,
$$
where the final inequality follows from our hypotheses on $u$ and $v$ on $\partial_1\sO$.

We consider the case where $\underline\sO_M$ is non-empty. If $M=0$ then $u\leq v$ on $\sO$, as desired and so, to obtain a contradiction, we suppose that $M>0$. Clearly, $\underline\sO_M$ is a relatively closed subset of $\underline\sO$ since $u-v_*$ is upper semicontinuous on $\underline\sO$. We wish to show that $\underline\sO_M$ is also a relatively open subset of $\underline\sO$, in which case (because $\sO$ and thus $\underline\sO$ are connected) either $\underline\sO_M$ is empty --- a case which we already ruled out --- or $\underline\sO_M=\underline\sO$.

We now proceed by adapting the proof\footnote{Lemma 6.2 in \cite{Han_Lin_2011} asserts that if $\sO\subset\RR^d$ is a bounded domain, $u\in C(\bar\sO)$ is subharmonic, $v\in C(\bar\sO)$ is superharmonic, and $u\leq v$ on $\partial\sO$, then $u\leq v$ on $\bar\sO$.} of \cite[Lemma 6.2]{Han_Lin_2011}. Because $\underline\sO_M$ is non-empty, we may choose $x_0\in \underline\sO_M$ and so
$$
u(x_0) = v_*(x_0) + M > v_*(x_0) \geq \psi(x_0),
$$
and thus $(u-\psi)(x_0)>0$. Let $U\Subset\sO$ be a ball $B_r(x_0)$ or half-ball $B_r^+(x_0)\Subset\underline\sO$ (for the case of $x_0\in\partial_0\sO$) with $r>0$ chosen small enough that $U\Subset\Omega=\{x\in\sO:(u-\psi)(x)>0\}$, which we may do since $u-\psi$ belongs to $C(\underline\sO)$ by Definition \ref{defn:Solution_obstacle_problem}. Since $f\in C^\alpha_s(\underline U)\cap C_b(U)$, Hypothesis \ref{hyp:Dirichlet_problem_local_solvability} ensures that there is a unique solution $\underline v\in C^{2+\alpha}_s(\underline U)\cap C_b(U\cup\partial_1U)$ to
$$
\begin{cases}
A\underline v=f &\hbox{on } U,
\\
\underline v = v &\hbox{on } \partial_1 U.
\end{cases}
$$
Then $\underline v\leq v$ on $U$ by Definition \ref{defn:Continuous_supersolution_obstacle_problem} since $v$ is a continuous supersolution, and hence
$$
u-\underline v \geq u-v \quad\hbox{on }U.
$$
By hypothesis, $A$ has the strong maximum principle property on $\underline\sO$ in the sense of Definition \ref{defn:Strong_max_principle_property} and because $\sO$ is bounded by our assumption for this case, Lemma \ref{lem:Elliptic_strong_implies_weak_maximum_principle_property_bounded_domain} implies that $A$ has the weak maximum principle property on $\underline\sO$ in the sense of Definition \ref{defn:Weak_max_principle_property}. Since $U\Subset\{u>\psi\}$, then by Definition \ref{defn:Solution_obstacle_problem} we must have $Au=f$ on $U$ and because $u$ belongs to $C^2(U)\cap C^1(\underline U)\cap C_b(U\cup \partial_1U)$ and obeys \eqref{eq:WMP_C2s_regularity}, \eqref{eq:WMP_xdsecond_order_derivatives_zero_boundary} along $\partial_0U$ (again by Definition \ref{defn:Solution_obstacle_problem}), then Hypothesis \ref{hyp:Dirichlet_problem_local_solvability} and the uniqueness afforded by the weak maximum principle property for $A$ on $\underline U$ imply that $u\in C^{2+\alpha}_s(\underline U)\cap C_b(U\cup\partial_1U)$. We have $A(u-\underline v) = Au - A\underline v = 0$ on $U$ and $u-\underline v = u-v \leq M$ on $\partial U$, that is,
$$
\begin{cases}
A(u-\underline v)=0 &\hbox{on } U,
\\
u-\underline v \leq M &\hbox{on } \partial_1 U.
\end{cases}
$$
By the weak maximum principle property for $A$ on $\underline U$, the fact that $M > 0$, and Proposition \ref{prop:Continuous_super_sub_solution_bvp_maximum_principle_estimates} \eqref{item:Continuous_subsolution_f_leq_zero}, we must therefore have $u-\underline v \leq \sup_{\partial_1 U}(u-\underline v)^+ = 0 \vee \sup_{\partial_1 U}(u-\underline v)\leq M$ on $U$; in particular,
$$
M \geq (u-\underline v)(x_0) \geq (u-v)(x_0) = M.
$$
Hence, $(u-\underline v)(x_0)=M$ and so $u-\underline v$ has a maximum at $x_0\in \underline U$.
%TODO Since M > 0 (need at least M \geq 0), then ...
The strong maximum principle property for $A$ on $\underline U$ implies that $u-\underline v\equiv M$ (constant) on $U$ and thus $U\subset \underline\sO_M$.

Since $\underline\sO$ is connected and $\underline\sO_M\subset\underline\sO$ is a non-empty (by our assumption) relatively open and closed subset, we consequently have $\underline\sO_M=\underline\sO$. But then $u-v\equiv M$ is constant on $\sO$ and thus on $\sO\cup \partial_1\sO$ (since $u$ and $v$ belong to $C(\sO\cup \partial_1\sO)$ by hypothesis) and hence $u-v\equiv M$ on $\partial_1\sO$ with $M > 0$, contradicting our hypothesis that $u-v\leq 0$ on $\partial_1\sO$. Therefore, we must have $u\leq v$ on $\sO$.

For the case of an \emph{unbounded} domain $\sO$, we adapt the proof of Lemma \ref{lem:Elliptic_weak_maximum_principle_extension_unbounded_domain}. Choose $v_0\in C^\infty(\underline\HH)$ as in \eqref{eq:Defn_vzero}, so
$$
v_0(x) = 1+|x|^2, \quad\forall\, x \in\HH,
$$
and recall that by \eqref{eq:A_vzero_geq_zero},
$$
Av_0 \geq 0 \quad\hbox{on }\sO.
$$
Suppose $\delta>0$. We claim that $v + \delta v_0$ is a continuous supersolution to the obstacle problem \eqref{eq:Elliptic_obstacle_problem} in the sense of Definition \ref{defn:Continuous_supersolution_obstacle_problem}. Clearly, $v+\delta v_0 \geq \psi$ on $\sO$ while if $v$ were a smooth supersolution in the sense of Definition \ref{defn:Smooth_supersolution_obstacle_problem}, so $Av \geq f$ on $\sO$, then
$$
A(v+\delta v_0) \geq f  \quad\hbox{on }\sO,
$$
and clearly $v+\delta v_0$ would be a smooth supersolution. When $v$ is only a continuous supersolution, suppose $U=B\Subset\sO$ or $U=B^+$, a half-ball with center in $\partial_0\sO$, and that $\underline v \in C^2(U)\cap C^1(\underline U)$ obeys \eqref{eq:WMP_C2s_regularity}, \eqref{eq:WMP_xdsecond_order_derivatives_zero_boundary}, and $\sup_U \underline v < \infty$, and
$$
A\underline v \leq f  \quad\hbox{on } U, \quad \underline v \leq v+\delta v_0 \quad\hbox{on } \partial_1U.
$$
But then $A(\underline v-\delta v_0) \leq f$ on $U$ and $\underline v-\delta v_0 \leq v$ on $\partial_1U$ and because $v$ is a continuous supersolution, we must have $\underline v-\delta v_0 \leq v$ on $U$, that is, $\underline v \leq v+\delta v_0$ on $U$. Therefore, $v+\delta v_0$ is a continuous supersolution to the obstacle problem \eqref{eq:Elliptic_obstacle_problem} in the sense of Definition \ref{defn:Continuous_supersolution_obstacle_problem}.

Because $\sup_\sO u < \infty$ and $\inf_\sO v >- \infty$, we must have
$$
u \leq v + \delta v_0 \quad\hbox{on }\sO\cap\partial_1B_R^+,
$$
for all large enough $R>0$ while our hypotheses imply that
$$
u \leq v \leq v + \delta v_0 \quad\hbox{on }B_R^+\cap\partial_1\sO,
$$
and thus, since $\partial_1(\sO\cap B_R^+) = (\bar\sO\cap\partial_1B_R^+) \cup (\bar B_R^+\cap\partial_1\sO)$,
$$
u \leq v + \delta v_0 \quad\hbox{on }\partial_1(\sO\cap B_R^+).
$$
Since $\inf_\sO v >- \infty$, we also have $\inf_\sO (v + \delta v_0) >- \infty$. Consequently,
noting that $\sO\cap B_R^+$ is connected for large enough $R$ by hypothesis, Theorem
\ref{thm:Elliptic_obstacle_problem_comparison_principle}, for the case of bounded subdomains of $\HH$, implies that $u \leq v + \delta v_0$ on $\sO\cap B_R^+$, for any sufficiently large $R>0$. Thus, $u \leq v + \delta v_0$ on $\sO$ for all $\delta>0$ and taking the limit as $\delta \downarrow 0$, we obtain $u \leq v$ on $\sO$, as desired, and this completes the proof for the case an unbounded domain, $\sO$.

Finally, if $v$ is also a solution, then because $u$ is a continuous supersolution by Lemma \ref{lem:Solution_obstacle_problem_is_continuous_supersolution}, we may reverse the roles of $u$ and $v$ in the preceding argument to also give $v\leq u$ on $\sO$ and hence $u=v$ on $\sO$.
\end{proof}

\begin{defn}[Set of continuous supersolutions to an obstacle problem]
\label{defn:Set_continuous_supersolutions_obstacle_problem}
Let $\sO\subseteqq\HH$ be a domain and $A$ be as in \eqref{eq:Generator}. Given $f\in C(\underline\sO)$, and $g\in C(\partial_1\sO)$, and $\psi\in C(\underline\sO\cup\partial_1\sO)$ obeying \eqref{eq:Boundarydata_obstacle_compatibility}, that is, $\psi\leq g$ on $\partial_1\sO$, we let $\sS_{f,g,\psi}^+$ denote the set of continuous supersolutions (in the sense of Definition \ref{defn:Continuous_supersolution_obstacle_problem}) to the obstacle problem \eqref{eq:Elliptic_obstacle_problem} with partial Dirichlet boundary condition \eqref{eq:Elliptic_boundary_condition}.
\end{defn}

\begin{thm}[A Perron method for existence of solutions to the degenerate-elliptic obstacle problem]
\label{thm:Perron_elliptic_obstacle_problem_solution}
Let $\sO\subseteqq\HH$ be a possibly unbounded domain, and $d<p<\infty$, and $\alpha\in(0,1)$. Assume the hypotheses for $A$, $f$, and $g$ in Theorem \ref{thm:Perron_elliptic_bvp_solution} and, in addition, that $c$ obeys \eqref{eq:Nonnegative_c}. If
$\psi\in C^2(\underline\sO)\cap C(\sO\cup \partial_1\sO)$ with $\sup_\sO\psi<\infty$ and obeys \eqref{eq:Boundarydata_obstacle_compatibility}, that is, $\psi\leq g$ on $\partial_1\sO$, then the \emph{Perron function},
\begin{equation}
\label{eq:Perron_elliptic_obstacle_problem_solution}
u(x) := \inf_{w\in \sS_{f,g,\psi}^+}w(x), \quad x\in\underline\sO,
\end{equation}
belongs to
$$
C^{2+\alpha}_s(\underline\Omega)\cap W^{2,p}_{\loc}(\sO)\cap C^{1,\alpha}_s(\underline\sO)\cap C_b(\sO),
$$
where $\Omega = \{x\in\sO: (u-\psi)(x)>0\}$ as in \eqref{eq:Continuation_region} and $u$ is a solution to the obstacle problem \eqref{eq:Elliptic_obstacle_problem}. Furthermore\footnote{This final conclusion also holds if each point of $\partial_1\sO$ obeys an interior sphere condition since Corollary \ref{cor:Strictly_elliptic_obstacle_problem_continuous_boundarydata} implies that $u$ belongs to $C(\sO\cup\partial_1\sO)$ and obeys the partial Dirichlet boundary condition \eqref{eq:Elliptic_boundary_condition}.} if each point of $\partial_1\sO$ is \emph{regular} with respect to $A$, $f$, and $g$ in the sense of Definition \ref{defn:Regular_boundary_point_elliptic_equation}, then in addition $u$ belongs to $C(\sO\cup\partial_1\sO)$ and $u$ obeys the boundary condition \eqref{eq:Elliptic_boundary_condition}.
\end{thm}

\begin{rmk}[Perron's method for existence of a solution to the obstacle problem]
Our definition \eqref{eq:Perron_elliptic_obstacle_problem_solution} of the Perron function for the obstacle problem \eqref{eq:Elliptic_obstacle_problem},\eqref{eq:Elliptic_boundary_condition} is the analogue of that given in \cite[Theorem 7.3.3 (iii)]{Rodrigues_1987} for the solution $u\in W^{1,\infty}_0(\sO)$ to the obstacle problem for a surface of constant mean curvature \cite[Equations (7.2.2) \& (7.2.3)]{Rodrigues_1987}, given $\psi \in C^{0,1}(\bar\sO)$ with $\sup_\sO\psi<\infty$ and $\psi<0$ on $\partial\sO$, where $\sO\subset\RR^d$ is a bounded domain with Lipschitz boundary.
\end{rmk}

Rather than prove Theorem \ref{thm:Perron_elliptic_obstacle_problem_solution} by adapting the proof of Theorem \ref{thm:Perron_elliptic_bvp_solution}, we shall proceed partly by generalizing the proof of \cite[Theorem 6.1]{Lindqvist_2006}. We begin with the following observation, which is an elementary consequence of the definition and analogue of Theorem \ref{thm:Elliptic_equation_continuous_sub_supersolution_properties} \eqref{item:elliptic_max_set_subsolutions}.

\begin{lem}[Minimum of a finite set of continuous supersolutions to the obstacle problem]
\label{lem:elliptic_obstacle_problem_min_set_supersolutions}
Let $\sO\subseteqq\HH$ be a domain and $A$ be as in \eqref{eq:Generator}. Suppose $f\in C(\underline\sO)$ and $\psi\in C(\underline\sO)$. If $v_1,\ldots, v_N$ are continuous supersolutions to \eqref{eq:Elliptic_obstacle_problem} in the sense of Definition \ref{defn:Continuous_supersolution_obstacle_problem}, then $v(x) := \min_{1\leq i\leq N}v_i(x)$, $x\in\underline\sO$, is a continuous supersolution to \eqref{eq:Elliptic_obstacle_problem}.
\end{lem}

Next, we have a partial analogue Theorem \ref{thm:Elliptic_equation_continuous_sub_supersolution_properties} \eqref{item:elliptic_subsolution_harmonic_lift}.

\begin{lem}[Construction of smaller continuous supersolutions to the obstacle problem]
\label{lem:Elliptic_obstacle_problem_smaller_supersolution}
Let $\sO\subseteqq\HH$ be a domain, $d< p<\infty$, and $\alpha\in(0,1)$. Let $A$ be as in \eqref{eq:Generator} with coefficients belonging to $C^\alpha(\sO)$ and obeying \eqref{eq:Nonnegative_c} and
%TODO Add this condition upfront and add references to it
\begin{equation}
\label{eq:Ellipticity_domain}
\langle a\xi, \xi\rangle > 0 \quad\hbox{on } \sO,\quad \forall\,\xi \in \RR^d \quad\hbox{\emph{(interior ellipticity)}}.
\end{equation}
Let $f\in C^\alpha(\sO)\cap C(\underline\sO)$ and $\psi\in C^2(\sO)\cap C(\underline\sO)$. Suppose that $v:\underline\sO\to\RR$ is a continuous supersolution to \eqref{eq:Elliptic_obstacle_problem} in the sense of Definition \ref{defn:Continuous_supersolution_obstacle_problem}. If $B\Subset\sO$ is a ball and $u_o \in W^{2,p}_{\loc}(B)\cap C(\bar B)$ is the unique solution\footnote{Existence and uniqueness of the solution is provided by Theorem \ref{thm:Strictly_elliptic_obstacle_problem_continuous_boundarydata}.} to
$$
\min\{Au_o-f, \ u_o - \psi\} = 0 \quad\hbox{a.e. on } B, \quad u_o = v \quad\hbox{on }\partial B,
$$
then
$$
\hat v := \begin{cases} u_o &\hbox{on } B, \\  v &\hbox{on } \sO\less B, \end{cases}
$$
is a continuous supersolution to \eqref{eq:Elliptic_obstacle_problem} and $\hat v \leq v$ on $\sO$.
\end{lem}

\begin{proof}
Observe that $\hat v\in L^\infty_{\loc}(\underline\sO)\cap C(\sO)$ by construction. We know that $u_o$ is a solution to \eqref{eq:Elliptic_obstacle_problem} on $B$ while $v\restriction_B$ is a continuous supersolution to \eqref{eq:Elliptic_obstacle_problem} on $B$ and $v=u_o$ on $\partial B$. Therefore, $v\geq u_o$ on $B$ by Theorem \ref{thm:Elliptic_obstacle_problem_comparison_principle}. Hence, $v\geq \hat v = u_o$ on $B$ while $v=\hat v$ on $\sO\less B$, and consequently, $v \geq \hat v$ on $\sO$.

Denote $\Omega = \{x\in\sO:(\hat v-\psi)(x)>0\}$. Let $U'\subset\sO$ denote an arbitrary ball $B'\Subset\sO$ or half-ball ${B'}^+\Subset\underline\Omega$ with center in $\partial_0\sO$ and suppose $\underline v\in C^2(U')\cap C^1(\underline U')$, obeying \eqref{eq:WMP_C2s_regularity}, \eqref{eq:WMP_xdsecond_order_derivatives_zero_boundary} when $U'={B'}^+$, and $\sup_{U'}\underline v<\infty$, and
$$
A\underline v \leq f\quad\hbox{on }U', \quad \underline v \leq \hat v \quad\hbox{on }\partial_1U'.
$$
According to Definition \ref{defn:Continuous_supersolution_obstacle_problem}, it remains for us to show that $\hat v \geq \underline v$ on $U'$.

Because $v$ is a continuous supersolution to \eqref{eq:Elliptic_obstacle_problem} on $\sO$ and $v\geq \hat v \geq \underline v$ on $\partial_1U'$, we must have $v\geq \underline v$ on $U'$ by Definition \ref{defn:Continuous_supersolution_obstacle_problem} and hence $\hat v \geq \underline v$ in $U'\less B$ since $\hat v = v$ in $U'\less B$. Furthermore, $\partial_1(B\cap U') = (B\cap\partial_1U') \cup (U'\cap\partial B)$ and $U'\cap\partial B \subset U'\less B$, so
$$
\hat v \geq \underline v \quad\hbox{on }\partial_1(B\cap U').
$$
Moreover, as $A\hat v = Au_o \geq f$ a.e. on $B$, then
$$
A(\hat v - \underline v) \geq 0\quad\hbox{a.e. on }B\cap U'.
$$
The condition \cite[p. 220]{GilbargTrudinger} on the coefficients $a^{ij}$ of $A$ is obeyed since $\det a(x)>0$ for $x\in\sO$ by \eqref{eq:Ellipticity_domain} and $a^{ij}\in C^\alpha(\sO)$ by hypothesis, so $\det a(x)$, for $x\in \bar B$, is uniformly bounded above and below by positive constants; the condition \cite[Equation (9.3)]{GilbargTrudinger} on $f$ and the coefficients $b^i$ of $A$ is also obeyed since $f, \ b^i \in C^\alpha(\sO)$; finally, $p\geq d$ by hypothesis. Hence, the Aleksandrov weak maximum principle\footnote{Provided the coefficients $a^{ij}$ belong to $C^{0,1}(B\cap U')$, so the operator $A$ in \eqref{eq:Generator} may be written in divergence form, one can appeal to the more elementary \cite[Theorem 8.1]{GilbargTrudinger} for  functions in $W^{1,2}(B\cap U')$.}
\cite[Theorem 9.1]{GilbargTrudinger} for functions in $W^{2,d}(B\cap U')\cap C(\overline{B\cap U'})$ implies that
$$
\hat v\geq \underline v \quad \hbox{on } B\cap U'.
$$
Consequently $\hat v \geq \underline v$ in $U'$ and therefore $\hat v$ is a continuous supersolution to \eqref{eq:Elliptic_obstacle_problem} on $\sO$.
\end{proof}

Next, we have the following generalization of \cite[Lemma 6.3]{Lindqvist_2006}; note that we do not yet assert that $u \in C(\underline\sO)$, but this (and more) will be established in the proof of Theorem \ref{thm:Perron_elliptic_obstacle_problem_solution}.

\begin{lem}[Continuity of the Perron function for the obstacle problem]
\label{lem:Perron_elliptic_problem_solution_continuity}
Let $\sO\subseteqq\HH$ be a domain,
%COMMENT We choose p > d rather than p > 2 so we can appeal to the Sobolev embedding
$d<p<\infty$, and $\alpha\in(0,1)$. Let $A$ be as in \eqref{eq:Generator} with coefficients belonging to $C^\alpha(\sO)$ and obeying \eqref{eq:Strict_ellipticity_domain} and \eqref{eq:Nonnegative_c}.
Let $f\in C^\alpha(\sO)\cap C(\underline\sO)$, and $g\in C(\partial_1\sO)$, and $\psi\in C^2(\sO)\cap C(\underline\sO)$. If $u$ in \eqref{eq:Perron_elliptic_obstacle_problem_solution} is locally bounded on $\sO$, then $u\in C(\sO)$.
\end{lem}

\begin{rmk}[Continuity of the Perron functions for the boundary value problem]
\label{rmk:Perron_elliptic_problem_solution_continuity}
The proof of Lemma \ref{lem:Perron_elliptic_problem_solution_continuity} also shows that if $u$ is the (upper or lower) Perron function \eqref{eq:Perron_elliptic_equation_supsubsolution} or \eqref{eq:Perron_elliptic_equation_infsupsolution} for the boundary value problem \eqref{eq:Elliptic_equation}, \eqref{eq:Elliptic_boundary_condition}, then $u\in C(\sO)$ under the same hypotheses.
\end{rmk}

\begin{proof}
We adapt the proof of \cite[Lemma 6.3]{Lindqvist_2006}. Let $x_0 \in \sO$ and $r>0$ be small enough that $B_r(x_0) \Subset \sO$. Given $\eps > 0$, we will find a radius $\delta > 0$ such that
$$
|u(x^1) - u(x^2)| < 2\eps, \quad\forall\, x^1, x^2 \in B_\delta(x_0).
$$
Suppose that $x^1, x^2 \in B_\delta(x_0)$. By definition of $u$ in \eqref{eq:Perron_elliptic_equation_supsubsolution}, we can find a sequence of continuous supersolutions, $\{v_n\}_{n\in\NN} \subset \sS_{f,g,\psi}^+$, such that
$$
\lim_{n\to\infty} v_n(x^1) = u(x^1), \quad \lim_{n\to\infty} v_n(x^2 ) = u(x^2).
$$
Indeed, if $v^1_n(x^1 ) \to u(x^1)$ and $v^2_n(x^2) \to u(x^2)$, then each $v_n := v^1_n \wedge v^2_n$ belongs to $\sS_{f,g,\psi}^+$ by Lemma \ref{lem:elliptic_obstacle_problem_min_set_supersolutions} and has the required limits since
$$
u(x^i) \leq v_n(x^i) \leq v^i_n(x^i), \quad\,\forall n \in \NN, \ i =1,2.
$$
Consider the ``harmonic lifts'' $\hat v_n \in \sS_{f,g,\psi}^+$ defined via Lemma \ref{lem:Elliptic_obstacle_problem_smaller_supersolution} by setting
$$
\hat v_n := \begin{cases} v_{o,n} &\hbox{on } B_r(x_0), \\  v_n &\hbox{on } \sO\less B_r(x_0), \end{cases}
$$
where $v_{o,n} \in W^{2,p}_{\loc}(B_r(x_0))\cap C(\bar B_r(x_0))$ is defined as the unique solution to
$$
\min\{Av_{o,n} - f, \ v_{o,n}-\psi\} = 0 \quad\hbox{a.e. on } B_r(x_0), \quad v_{o,n} = v_n \quad\hbox{on } \partial B_r(x_0).
$$
Lemma \ref{lem:Elliptic_obstacle_problem_smaller_supersolution} provides that $\hat v_n \leq v_n$ on $\sO$, while $u\leq \hat v_n$ on $\sO$ by definition of $u$ in \eqref{eq:Perron_elliptic_equation_supsubsolution}. Fix $n$ large enough that
$$
v_n(x^1) < u(x^1) + \eps , \quad v_n(x^2) < u(x^2) + \eps.
$$
It follows that
\begin{align*}
u(x^2 ) - u(x^1) &\leq \hat v_n(x_2) - u(x^1)
\\
&< \hat v_n(x^2) - v_n(x^1) + \eps
\\
&\leq \hat v_n(x^2) - \hat v_n(x^1) + \eps,
\end{align*}
and thus, since $\hat v_n = v_{o,n}$ on $B_r(x_0)$,
$$
u(x^2 ) - u(x^1) < v_{o,n}(x^2) - v_{o,n}(x^1) + \eps.
$$
But $v_{o,n} \in W^{2,p}_{\loc}(B_r(x_0)) \subset C^{1,\gamma}_{\loc}(B_r(x_0))$ (for $\gamma\in(0,d/p]$ since $d<p<\infty$ by hypothesis), so $v_{o,n} \in C^\gamma(\bar B_{r/2}(x_0))$ and thus for $0 < \delta \leq r/2$, we have
\begin{align*}
v_{o,n} (x^2 ) - v_{o,n}(x^1) &\leq \sup_{x,y \in B_\delta(x_0)}|v_{o,n}(x) - v_{o,n}(y)|
\\
&= \sup_{x,y \in B_\delta(x_0)}\frac{|v_{o,n}(x) - v_{o,n}(y)|}{|x-y|^\gamma}|x-y|^\gamma
\\
&\leq \delta^\gamma \|v_{o,n}\|_{C^\gamma(\bar B_\delta(x_0))}
\leq
\delta^\gamma \|v_{o,n}\|_{C^\gamma(\bar B_{r/2}(x_0))}.
\end{align*}
Choosing $\delta \in (0, r/2)$ small enough to ensure that $\delta^\gamma \|v_{o,n}\|_{C^\gamma(\bar B_{r/2}(x_0))} < \eps$ yields
$$
u(x^2) - u(x^1) < \eps + \eps = 2\eps.
$$
By symmetry, $u(x^1) - u(x^2) < 2\eps$, and thus continuity of $u$ follows.
\end{proof}

We now have an elementary consequence of existence and uniqueness of solutions to \eqref{eq:Elliptic_obstacle_problem}, \eqref{eq:Elliptic_boundary_condition} when $A$ is strictly elliptic \cite{Bensoussan_Lions, Friedman_1982, Rodrigues_1987, Troianiello}.

\begin{lem}[Perron solution to the strictly obstacle problem on a ball]
\label{lem:Perron_strictly_elliptic_obstacle_problem_solution_ball}
Let $\sO\subset\HH$ be a domain, and $d<p<\infty$ and $\alpha\in (0,1)$, and $A$ be as in \eqref{eq:Generator}. Let $B\Subset\sO$ be a ball and the coefficients of $A$ belong to $C^\alpha(\bar B)$, and the $a^{ij}$ obey \eqref{eq:Strict_ellipticity_domain} on $B$, and $c$ obeys \eqref{eq:Nonnegative_c} on $B$, so $c\geq 0$ on $B$. Suppose $f\in C^\alpha(B)\cap C_b(B)$, and $h\in C(\partial B)$, and $\psi\in C^2(B)\cap C(\bar B)$ obeys $\psi\leq h$ on $\partial B$. If
\begin{equation}
\label{eq:Perron_elliptic_obstacle_problem_solution_ball}
u(x) = \inf_{w\in \sS_{f,h,\psi}^+}w(x), \quad x\in B,
\end{equation}
then $u$ is the unique solution in $W^{2,p}_{\loc}(B)\cap C(\bar B)$ to the strictly elliptic obstacle problem,
\begin{equation}
\label{eq:Strictly_elliptic_obstacle_problem_ball}
\min\{Au-f, \ u - \psi\} = 0 \quad\hbox{a.e. on } B, \quad u = h \quad\hbox{on }\partial B.
\end{equation}
\end{lem}

\begin{proof}
Existence and uniqueness of the solution $u \in W^{2,p}_{\loc}(B)\cap C(\bar B)$ to \eqref{eq:Strictly_elliptic_obstacle_problem_ball} is assured by Theorem \ref{thm:Strictly_elliptic_obstacle_problem_continuous_boundarydata}.
By Lemma \ref{lem:Solution_obstacle_problem_is_continuous_supersolution}, the solution $u$ is also a continuous supersolution to the obstacle problem on $B$, that is, $u\in \sS_{f,h,\psi}^+$, while the comparison principle (Theorem \ref{thm:Elliptic_obstacle_problem_comparison_principle}) implies that $w\geq u$ on $B$ for all $w\in \sS_{f,h,\psi}^+$ and so $u$ is given by the Perron representation \eqref{eq:Perron_elliptic_obstacle_problem_solution_ball}.
\end{proof}

\begin{lem}[Localization of the Perron function on an interior ball]
\label{lem:Perron_elliptic_obstacle_problem_solution_localization}
Assume the hypotheses of Theorem \ref{thm:Perron_elliptic_obstacle_problem_solution}. If $u:\underline\sO\to\RR$ is given by \eqref{eq:Perron_elliptic_obstacle_problem_solution} and $B\Subset\sO$ is a ball, then
\begin{equation}
\label{eq:Perron_elliptic_obstacle_problem_solution_localization}
u(x) = \inf_{w\in \sS_{f,u\restriction B,\psi}^+}w(x), \quad x\in B.
\end{equation}
\end{lem}

\begin{proof}
Let $u_o$ denote the function on the right-hand side of \eqref{eq:Perron_elliptic_obstacle_problem_solution_localization}. Lemma \ref{lem:Perron_strictly_elliptic_obstacle_problem_solution_ball} implies that $u_o \in W^{2,p}_{\loc}(B)\cap C(\bar B)$ and is the unique solution to
$$
\min\{Au_o - f, \ u_o - \psi\} = 0 \quad\hbox{a.e. on } B, \quad u_o = u \quad\hbox{on }\partial B.
$$
If $w\in \sS_{f,g,\psi}^+$, Lemma \ref{lem:Elliptic_obstacle_problem_smaller_supersolution} implies that the function
$$
\hat w := \begin{cases} u_o &\hbox{on } B, \\  w &\hbox{on } \sO\less B, \end{cases}
$$
is also a continuous supersolution to \eqref{eq:Elliptic_obstacle_problem}, that is, $\hat w\in \sS_{f,g,\psi}^+$, and
$\hat w \leq w$ on $\sO$. Therefore,
$$
u(x) = \inf_{w\in \sS_{f,g,\psi}^+}w(x) = \inf_{w\in \sS_{f,g,\psi}^+}\hat w(x) = u_o(x), \quad x\in B,
$$
and hence $u=u_o$ on $B$, as desired.
\end{proof}

\begin{lem}[Localization of the Perron function on the continuation region]
\label{lem:Elliptic_localization_contination_region}
Assume the hypotheses of Theorem \ref{thm:Perron_elliptic_obstacle_problem_solution}. If $u:\underline\sO\to\RR$ is given by \eqref{eq:Perron_elliptic_obstacle_problem_solution} and $U\Subset\underline\Omega$, where $\Omega$ is given by \eqref{eq:Continuation_region}, then
$$
\inf_{w\in \sS_{f,u\restriction U,\psi}^+}w(x) = \inf_{w\in \sS_{f,u\restriction U}^+}w(x), \quad x\in\underline U.
$$
\end{lem}

\begin{proof}
If $w\in \sS_{f,u\restriction U,\psi}^+$, then $w\in \sS_{f,u\restriction U}^+$ by Definitions \ref{defn:Continuous_super_sub_solution_bvp} and \ref{defn:Continuous_supersolution_obstacle_problem}, and so
$$
\sS_{f,u\restriction U,\psi}^+ \subset \sS_{f,u\restriction U}^+.
$$
Let $w\in \sS_{f,u\restriction U}^+$. Then $w\geq u$ on $\partial_1 U$ and so by Definition \ref{defn:Continuous_super_sub_solution_bvp}, we must have $w\geq u$ on $U$. But $u>\psi$ on $U$ (since $U \Subset\underline\Omega$) and so $w\geq\psi$ on $\underline U$ and thus $w\in \sS_{f,u\restriction U,\psi}^+$ by Definition \ref{defn:Continuous_supersolution_obstacle_problem}. Therefore,
$$
\sS_{f,u\restriction U,\psi}^+ \supset \sS_{f,u\restriction U}^+,
$$
and consequently, $\sS_{f,u\restriction U,\psi}^+ = \sS_{f,u\restriction U}^+$, which gives the identity.
\end{proof}

We provide a priori estimates for continuous supersolutions to the obstacle problem via the following refinement of \cite[Proposition 3.18]{Feehan_maximumprinciple_v1}; Theorem \ref{thm:Elliptic_obstacle_problem_comparison_principle} provides the key ingredient in the proof.

\begin{prop}[Comparison principle estimates for solutions and continuous supersolutions to the obstacle problem]
\label{prop:Continuous_supersolution_obstacleproblem_maximum_principle_estimates}
\cite[Proposition 3.18]{Feehan_maximumprinciple_v1}
Assume the hypotheses of Theorem \ref{thm:Elliptic_obstacle_problem_comparison_principle} on $\sO$, $A$, $f$, and $\psi$, and that $c$ obeys \eqref{eq:Nonnegative_c}. Let $g\in C(\partial_1\sO)$. Suppose $u$ is a solution and $v$ is a continuous supersolution to the obstacle problem \eqref{eq:Elliptic_obstacle_problem}, \eqref{eq:Elliptic_boundary_condition} for $f$, $g$, and $\psi$ in the sense of Definitions \ref{defn:Solution_obstacle_problem} and \ref{defn:Continuous_supersolution_obstacle_problem}, where $u, v$ belong to $C(\sO\cup\partial_1\sO)$. Then the following hold.
\begin{enumerate}
\item\label{item:Obstacle_continuous_supersolution_f_geq_zero} If $f\geq 0$ on $\sO$ and $v$ is a continuous supersolution for $f$ and $g$ and $\inf_\sO v > -\infty$, then
$$
v\geq 0 \wedge \inf_{\partial_1\sO}g \quad\hbox{on }\sO.
$$
\item\label{item:Obstacle_continuous_supersolution_f_arb_sign} If $f$ has arbitrary sign, $v$ is a continuous supersolution for $f$ and $g$ and $\inf_\sO v > -\infty$ but, in addition, there is a constant $c_0>0$ such that $c$ obeys \eqref{eq:Positive_lower_bound_c_domain}, then
$$
v\geq 0 \wedge \frac{1}{c_0}\inf_ \sO f \wedge \inf_{\partial_1\sO}g \quad\hbox{on }\sO.
$$
\item\label{item:Obstacle_solution_f_leq_zero} If $f\leq 0$ on $\sO$ and $u$ is a solution for $f$, $g$, and $\psi$ and $\sup_\sO u < \infty$, then
$$
u\leq 0 \vee \sup_{\partial_1\sO}g \vee \sup_\sO \psi \quad\hbox{on }\sO.
$$
\item\label{item:Obstacle_solution_f_arb_sign} If $f$ has arbitrary sign and $u$ is a solution for $f$, $g$, and $\psi$ and $\sup_\sO u < \infty$ but, in addition, there is a constant $c_0>0$ such that $c$ obeys \eqref{eq:Positive_lower_bound_c_domain}, then
$$
u\leq 0 \vee \frac{1}{c_0}\sup_\sO f \vee \sup_{\partial_1\sO}g \vee \sup_\sO \psi \quad\hbox{on }\sO.
$$
\end{enumerate}
When $\partial_1\sO=\emptyset$, the terms $\inf_{\partial_1\sO}g$ and $\sup_{\partial_1\sO}g$ are omitted in the preceding inequalities.
\end{prop}

We can now give the

\begin{proof}[Proof of Theorem \ref{thm:Perron_elliptic_obstacle_problem_solution}]
We first show that the set $\sS_{f,g,\psi}^+$ is non-empty by exhibiting a constant $M$ which is a continuous supersolution to \eqref{eq:Elliptic_obstacle_problem}, \eqref{eq:Elliptic_boundary_condition}. When $c$ obeys \eqref{eq:Positive_lower_bound_c_domain}, so $c\geq c_0$ on $\underline\sO$ for a positive constant $c_0$, the proof of Proposition \ref{prop:Continuous_supersolution_obstacleproblem_maximum_principle_estimates} shows that
\footnote{By hypothesis, the functions $f$ and $\psi$ on $\sO$ and $g$ on $\partial_1\sO$ are bounded above.}
$$
M := 0\vee\frac{1}{c_0}\sup_\sO f \vee \sup_{\partial_1\sO} g \vee \sup_\sO\psi
$$
is a continuous (and also smooth) supersolution to \eqref{eq:Elliptic_obstacle_problem}, \eqref{eq:Elliptic_boundary_condition}, while for the case $c$ only obeying \eqref{eq:Nonnegative_c}, so $c\geq 0$ on $\underline\sO$, we may also construct a constant supersolution $M$ by Remark \ref{rmk:Elliptic_weak_maximum_principle_finite_height_domain} when $\height(\sO)$ is finite, $a^{dd}$ obeys \eqref{eq:Finite_upper_bound_add_domain}, and $b^d$ obeys \eqref{eq:Positive_lower_bound_bd_domain}, so $b^d\geq b_0$ on $\underline\sO$ for a positive constant $b_0$. Thus, $\sS_{f,g,\psi}^+$ is non-empty. By replacing any supersolution $v \in \sS_{f,g,\psi}^+$ by $v\wedge M$ (and appealing to Lemma \ref{lem:elliptic_obstacle_problem_min_set_supersolutions}), we may furthermore assume without loss of generality that $v \leq M$ on $\sO$ in the definition \eqref{eq:Perron_elliptic_obstacle_problem_solution} of $u$.

The preceding remarks show that $\sup_\sO u \leq M$. If $v \in \sS_{f,g,\psi}^+$ is any continuous supersolution then, for $c$ obeying \eqref{eq:Positive_lower_bound_c_domain}, so $c\geq c_0$ on $\underline\sO$ for a positive constant $c_0$, we can apply
Proposition \ref{prop:Continuous_supersolution_obstacleproblem_maximum_principle_estimates} \eqref{item:Obstacle_continuous_supersolution_f_arb_sign}
to conclude that\footnote{By hypothesis, the functions $f$ on $\sO$ and $g$ on $\partial_1\sO$ are bounded below.}
$$
v \geq 0\wedge\frac{1}{c_0}\inf_\sO f \wedge \inf_{\partial_1\sO}g \quad\hbox{on }\sO,
$$
and thus $u$ also satisfies this lower bound on $\sO$. For $c$ only obeying \eqref{eq:Nonnegative_c}, so $c\geq 0$ on $\underline\sO$, but $\height(\sO)$ is finite, $a^{dd}$ obeys \eqref{eq:Finite_upper_bound_add_domain}, and $b^d$ obeys \eqref{eq:Positive_lower_bound_bd_domain}, so $b^d\geq b_0$ on $\underline\sO$ for a positive constant $b_0$, Remark \ref{rmk:Elliptic_weak_maximum_principle_finite_height_domain} again yields a uniform lower bound for $u$ on $\sO$. Thus, $u$ is bounded on $\sO$.

We have $u\in C(\sO)$ by Lemma \ref{lem:Perron_elliptic_problem_solution_continuity}. By applying Lemmas \ref{lem:Perron_strictly_elliptic_obstacle_problem_solution_ball} and \ref{lem:Perron_elliptic_obstacle_problem_solution_localization} to each ball $B\Subset\sO$, we see that $u\in W^{2,p}_{\loc}(\sO)$ and that $u$ solves the obstacle problem \eqref{eq:Elliptic_obstacle_problem} on $\sO$. Since $\Omega = \{x\in\sO:(u-\psi)(x)>0\}$ is an open subset of $\sO$, then $\underline\Omega$ is an open subset of $\underline\sO$
%COMMENT It is desirable to simplify or streamline these topological assertions
and\footnote{Recall that $\underline\Omega := \Omega\cup\partial_0\Omega$ and $\underline\sO := \sO\cup\partial_0\sO$, where $\partial_0\Omega := \Int(\bar\Omega\cap\partial\HH)$ and $\partial_0\sO := \Int(\bar\sO\cap\partial\HH)$, and hence $\underline\Omega\subset\underline\sO$.}
so for each $x_0\in\underline\Omega$, we may choose an $r>0$ small enough that $U\Subset\underline\Omega$, where $U=B_r(x_0)$ with $x_0\in\sO$ or $B_r^+(x_0)$ with $x_0\in\partial_0\sO$. Theorem \ref{thm:Perron_elliptic_bvp_solution}, Remark \ref{rmk:Perron_elliptic_bvp_solution}, and Lemma \ref{lem:Elliptic_localization_contination_region} imply that $u$ solves $Au = f$ on $U$ and belongs to $C^{2+\alpha}_s(\underline U)$. Therefore, $u\in C^{2+\alpha}_s(\underline\Omega)$.

Recall that $W^{2,p}(V) \hookrightarrow C^{1,\gamma}(\bar V)$ for any $\gamma\in(0,1-d/p]$, when $d<p<\infty$, and domain $V\subset\RR^d$ whose boundary $\partial V$ has the strong local Lipschitz property by \cite[Theorem 5.4 ($\hbox{C}'$)]{Adams_1975} and therefore $Du=D\psi$ along the free boundary $\sO\cap\partial\Omega$. Since $\psi \in C^2(\underline\sO)$ and $u=\psi$ on $\underline\sO\less\Omega$, then $u \in C^2(\underline\sO\less\Omega) \subset C^{1,\alpha}(\underline\sO\less\Omega)$ and because $u\in C^{2+\alpha}_s(\underline\Omega) \subset C^{1,\alpha}_s(\underline\Omega)$, we obtain $u\in C^{1,\alpha}_s(\underline\sO)$, since $C^{1,\alpha}(\underline\sO\less\Omega) \subset C^{1,\alpha}_s(\underline\sO\less\Omega)$ by Definition \ref{defn:DHspaces} and \eqref{eq:DH_and_standard_Holder_relationships}. (See Figure \ref{fig:domain_continuation_region}.)
%TODO Recheck preceding paragraph

It remains to verify that $u$ attains the boundary value, $g$, continuously along $\partial_1\sO$. Suppose $x_0\in\partial_1\sO$. Observe that
$$
\partial_1\sO =  \left(\partial_1\sO\cap\partial_1(\sO\less\Omega)\right) \cup \Int\left(\partial_1\sO\cap\partial_1\Omega\right).
$$
If $x_0 \in \partial_1\sO\cap\partial_1(\sO\less\Omega)$, then since $u=\psi$ on $\sO\less\Omega$ and $\psi \in C(\sO\cup \partial_1\sO)$ and thus $\psi$ is continuous at $x_0$, we see that $u$ extends continuously to $x_0$ by setting $u(x_0)=\psi(x_0)$, with $u(x_0)\leq g(x_0)$ by \eqref{eq:Boundarydata_obstacle_compatibility}. The definition \eqref{eq:Perron_elliptic_obstacle_problem_solution} of $u$ implies that $u(x_0)\geq g(x_0)$ if $u$ is continuous at $x_0$ and therefore $u(x_0)=g(x_0)$. If $x_0\in \Int(\partial_1\sO\cap\partial_1\Omega)$, then there is an $r>0$ such that $B_r(x_0)\cap \sO\subset\Omega$ and $B_r(x_0)\cap \partial_1\sO \subset \partial\Omega$. Because each $x_0\in\partial_1\sO$ is a regular point for the elliptic equation \eqref{eq:Elliptic_equation} defined by $A$, $f$, $g$, and $\sO$ in the sense of Definition \ref{defn:Regular_boundary_point_elliptic_equation} by hypothesis, then $u$ is continuous at $x_0$ with $u(x_0)=g(x_0)$.
\end{proof}

\begin{proof}[Proof of Theorem \ref{thm:Existence_uniqueness_elliptic_obstacle}]
The remainder of proof is the same as that of Theorem \ref{thm:Existence_uniqueness_elliptic_Dirichlet} except that we appeal to Theorem \ref{thm:Perron_elliptic_obstacle_problem_solution} in place of Theorem \ref{thm:Perron_elliptic_bvp_solution}.
\end{proof}

Finally, we have the

\begin{thm}[A Perron method of existence of solutions to the degenerate-elliptic obstacle problem with Lipschitz obstacle function with locally finite concavity]
\label{thm:Perron_elliptic_obstacle_problem_solution_lipschitz_obstacle}
Assume the hypotheses of Theorem \ref{thm:Perron_elliptic_obstacle_problem_solution}, except that the assumption $\psi\in C^2(\underline\sO)\cap C(\sO\cup\partial_1\sO)$ is relaxed to $\psi \in C^{0,1}(\sO)\cap C(\underline\sO)$ and obeying \eqref{eq:Elliptic_obstaclefunction_locally_finite_concavity_condition} in $\sO$. Then the conclusions of Theorem \ref{thm:Perron_elliptic_obstacle_problem_solution} again hold, except that now
$$
u \in C^{2+\alpha}_s(\underline\Omega) \cap W^{2,p}_{\loc}(\sO) \cap C_b(\underline\sO).
$$
\end{thm}

\begin{proof}
It suffices to observe that the role of Theorem \ref{thm:Strictly_elliptic_obstacle_problem_continuous_boundarydata} in Lemmas \ref{lem:Elliptic_obstacle_problem_smaller_supersolution}, \ref{lem:Perron_elliptic_problem_solution_continuity}, \ref{lem:Perron_strictly_elliptic_obstacle_problem_solution_ball}, and \ref{lem:Perron_elliptic_obstacle_problem_solution_localization} is replaced by Theorem \ref{thm:Strictly_elliptic_obstacle_problem_continuous_boundarydata_lipschitz_obstacle}.
\end{proof}

\begin{proof}[Proof of Theorem \ref{thm:Existence_uniqueness_elliptic_obstacle_lipschitz}]
The proof is the same as that of Theorem \ref{thm:Existence_uniqueness_elliptic_obstacle} except that we appeal to Theorem \ref{thm:Perron_elliptic_obstacle_problem_solution_lipschitz_obstacle} in place of Theorem \ref{thm:Perron_elliptic_obstacle_problem_solution}.
\end{proof}

\appendix

\section{Obstacle problems for strictly elliptic operators}
\label{sec:Obstacle_strictly_elliptic_operator}
In this section we develop the results we shall need for our application to the construction of solutions by the Perron process to the obstacle problem \eqref{eq:Elliptic_obstacle_problem}, \eqref{eq:Elliptic_boundary_condition} for a degenerate-elliptic operator \eqref{eq:Generator}. For this purpose, we shall now consider a general elliptic operator
\begin{equation}
\label{eq:Strictly_elliptic_operator}
Av := -\tr(aD^2v) - \langle b, Dv\rangle + cv \quad\hbox{on }\sO, \quad v \in C^\infty(\sO),
\end{equation}
where $a:\sO\to\RR^{d\times d}$, and $b:\sO\to\RR^d$, and $c:\sO\to\RR$ are functions whose properties we shall further prescribe below and, in particular, that $a(x)$ is symmetric and \emph{strictly} elliptic for $x\in\sO$.

\begin{thm}[Existence and uniqueness of solutions to an obstacle problem for a strictly elliptic operator]
\footnote{Theorem \ref{thm:Strictly_elliptic_obstacle_problem} may be proved, with slightly weaker hypotheses, using the theory of variational inequalities, even when $A$ only defines a non-coercive bilinear form on $W^{1,2}(\sO)$: compare \cite[Theorem 3.1.1 \& Corollary 3.1.1]{Bensoussan_Lions}, \cite[Theorems 4.7.7, 5.2.5 (ii), \& 5.3.4 (i)]{Rodrigues_1987}, and \cite[Theorems 4.27 \& 4.38]{Troianiello}.}
\label{thm:Strictly_elliptic_obstacle_problem}
\cite[Theorems 1.3.2 \& 1.3.4]{Friedman_1982}
Let $\alpha\in(0,1)$ 
and\footnote{Friedman's proof of existence (\cite[Theorem 1.3.2]{Friedman_1982}) is valid for any $p<\infty$; his proof of uniqueness (\cite[Theorem 1.3.3]{Friedman_1982}) assumes that $p\geq 2$, but his proof appears to rely on the Aleksandrov maximum principle \cite[Theorem 9.1]{GilbargTrudinger}, which requires $p\geq d$.}
$d<p<\infty$, and $\sO\subset\RR^d$ be a bounded domain with $C^{2,\alpha}$ boundary $\partial\sO$, and $A$ be as in \eqref{eq:Strictly_elliptic_operator} with coefficients $a^{ij}, b^i ,c$ belonging to $C^\alpha(\bar\sO)$ and obeying \eqref{eq:Nonnegative_c}, so $c\geq 0$ on $\sO$, and \eqref{eq:Strict_ellipticity_domain}. Let $f\in C^\alpha(\bar\sO)$, and $g \in C^{2,\alpha}(\bar\sO)$, and $\psi \in C^2(\bar\sO)$ with $\psi\leq g$ on $\partial\sO$. Then there is a unique solution $u \in W^{2,p}(\sO)$ to
\begin{equation}
\label{eq:Strictly_elliptic_obstacle_boundary_problem}
\min\{Au - f, \ u-\psi\} = 0 \quad\hbox{a.e. on } \sO, \quad u = g \quad\hbox{on }\partial\sO.
\end{equation}
\end{thm}

Because of our use of the Perron method, we shall need a version of Theorem \ref{thm:Strictly_elliptic_obstacle_problem} for boundary data which is only assumed to be continuous, by analogy with \cite[Theorem 6.13]{GilbargTrudinger} for the Dirichlet problem for a strictly elliptic operator. In order to relax the regularity condition on $g$ in Theorem \ref{thm:Strictly_elliptic_obstacle_problem}, we shall need a priori estimates for solutions to \eqref{eq:Strictly_elliptic_obstacle_boundary_problem}. We begin with a special case of \cite[Proposition 3.18 (6)]{Feehan_maximumprinciple_v1}.

\begin{prop}[Maximum principle estimate for solutions to an obstacle problem for a strictly elliptic operator]
\label{prop:Elliptic_classical_rodrigues}
\cite[Proposition 3.18 (6)]{Feehan_maximumprinciple_v1}
Let $\sO\subseteqq\RR^d$ be a bounded domain and $A$ be as in \eqref{eq:Strictly_elliptic_operator}, with $a(x)$ obeying \eqref{eq:Strict_ellipticity_domain} for $x\in\sO$, together with
\begin{equation}
\label{eq:Coefficient_uniform_bound_domain}
\sup_\sO|a| + \sup_\sO|b| + \sup_\sO|c| \leq \Lambda,
\end{equation}
where $\Lambda$ is a positive constant. Let $f_i\in C^\alpha(\sO)\cap C_b(\sO)$, and $g_i\in C(\partial\sO)$, and $\psi_i\in C(\bar\sO)$ with $\psi_i\leq g_i$ on $\partial\sO$, for $i=1,2$. Suppose $u_1, u_2$ are solutions to the obstacle problem \eqref{eq:Strictly_elliptic_obstacle_boundary_problem} in the sense of Definition \ref{defn:Solution_obstacle_problem}, respectively, for $f_1,\psi_1$ and $f_2, \psi_2$ on $\sO$, and $g_1$ and $g_2$ on $\partial\sO$. If $c$ obeys \eqref{eq:Positive_lower_bound_c_domain}, so $c\geq c_0$ on $\sO$ for a positive constant $c_0$, then
$$
\|u_1-u_2\|_{C(\bar\sO)} \leq \frac{1}{c_0}\|f_1-f_2\|_{C(\bar\sO)} \vee \|g_1-g_2\|_{C(\overline{\partial\sO})} \vee \|\psi_1-\psi_2\|_{C(\bar\sO)},
$$
while if $f_1=f=f_2$ and $c$ just obeys \eqref{eq:Nonnegative_c}, so $c\geq 0$ on $\sO$, then
$$
\|u_1-u_2\|_{C(\bar\sO)} \leq \|g_1-g_2\|_{C(\overline{\partial\sO})} \vee \|\psi_1-\psi_2\|_{C(\bar\sO)}.
$$
Finally, if $c$ just obeys \eqref{eq:Nonnegative_c}, so $c\geq 0$ on $\sO$, but $\diam(\sO)\leq d_0$ for a positive constant $d_0$, then
$$
\|u_1-u_2\|_{C(\bar\sO)} \leq e^{2d_0\Lambda/\lambda_0}\left(\frac{\lambda_0}{2\Lambda^2}\|f_1-f_2\|_{C(\bar\sO)} \vee \|g_1-g_2\|_{C(\overline{\partial\sO})} \vee \|\psi_1-\psi_2\|_{C(\bar\sO)}\right).
$$
\end{prop}

\begin{proof}
The hypotheses ensure that $A$ has the weak and strong maximum principle properties by \cite[Theorem 3.1 \& Corollary 3.2 and Theorem 3.5]{GilbargTrudinger}. Therefore, the first two assertions are given by \cite{Feehan_maximumprinciple_v1}, while the third assertion follows from the first by employing a modification of the proofs of \cite[Theorem 3.7]{GilbargTrudinger} and \cite[Proposition A.1 \& Corollary A.2]{Feehan_Pop_elliptichestonschauder}. Indeed, suppose that $\sO$ lies in the slap $0<x_1<d_0$ and $u$ obeys $Au\geq f$ and $u\geq \psi$ on $\sO$, and $u=g$ on $\partial\sO$. Set $u=e^{\lambda x_1}v$ for $\lambda>0$ to be chosen and observe that $u_{x_1}=e^{\lambda x_1}v_{x_1} + \lambda e^{\lambda x_1} v$ and $u_{x_1x_1}=e^{\lambda x_1}v_{x_1x_1} + 2\lambda e^{\lambda x_1} v_{x_1} + \lambda^2 e^{\lambda x_1} v$, and thus $v$ obeys
$$
\tilde A v := Av + \lambda(a^{1i} + a^{i1})v_{x_i} + 2\lambda a^{11}e^{\lambda x_1} v_{x_1} + (\lambda^2 a^{11} + \lambda b^1)v \geq e^{-\lambda x_1}f =: \tilde f \quad\hbox{on }\sO,
$$
and $v \geq \tilde\psi := e^{-\lambda x_1}\psi$ on $\sO$, and $v=\tilde g := e^{-\lambda x_1}g$ on $\partial\sO$. But \eqref{eq:Strict_ellipticity_domain} implies that $a^{11}\geq \lambda_0$ on $\sO$, and so
$$
\lambda a^{11} + b^1 \geq \lambda\lambda_0 - \|b^1\|_{C(\bar\sO)} \geq \lambda\lambda_0 - \Lambda > 0,
$$
provided $\lambda > \Lambda/\lambda_0$, say $\lambda := 2\Lambda/\lambda_0$. For such a constant $\lambda$, the operator $\tilde A$ obeys \eqref{eq:Strict_ellipticity_domain} and has $\tilde c \geq \tilde c_0 > 0$ with $\tilde c_0 := \lambda(\lambda\lambda_0 - \Lambda) = 2\Lambda^2/\lambda_0$ and so the third assertion follows from the first using $v_i:=e^{-\lambda x_1}u_i$, and $\tilde f_i := e^{-\lambda x_1}f_i$, and $\tilde\psi_i := e^{-\lambda x_1}\psi_i$, and $\tilde g_i := e^{-\lambda x_1}g_i$ for $i=1,2$.
\end{proof}

We shall require the

\begin{thm}[A priori $W^{2,p}$ estimates for a solution to an obstacle problem for a strictly elliptic operator]
\label{thm:Strictly_elliptic_obstacle_problem_apriori_W2p_estimate}
Assume the hypotheses of Theorem \ref{thm:Strictly_elliptic_obstacle_problem} and, in addition, that \eqref{eq:Coefficient_uniform_bound_domain} holds for some positive constant $\Lambda$.
\begin{enumerate}
\item (Global a priori estimate.) There is a positive constant $C=C(d,p,\sO,\lambda_0,\Lambda)$ such that the following holds: If $u\in W^{2,p}(\sO)$ solves the obstacle problem \eqref{eq:Strictly_elliptic_obstacle_boundary_problem}, then
\begin{equation}
\label{eq:Strictly_elliptic_obstacle_problem_global_apriori_W2p_estimate}
\|u\|_{W^{2,p}(\sO)} \leq C\left(\|f\|_{L^p(\sO)} + \|(A\psi-f)^+\|_{L^p(\sO)} + \|g\|_{W^{2,p}(\sO)}\right).
\end{equation}
\item (Interior a priori estimate.) If $\sO'\Subset\sO$ is a precompact open subset, then there is a positive constant $C=C(d,p,\sO',\sO,\lambda_0,\Lambda)$ such that the following holds, with conditions on $\partial\sO$ and $g$ omitted: If $u\in W^{2,p}_{\loc}(\sO)\cap L^p(\sO)$ solves the obstacle problem \eqref{eq:Strictly_elliptic_obstacle_boundary_problem}, then
\begin{equation}
\label{eq:Strictly_elliptic_obstacle_problem_interior_apriori_W2p_estimate}
\|u\|_{W^{2,p}(\sO')} \leq C\left(\|f\|_{L^p(\sO)} + \|(A\psi-f)^+\|_{L^p(\sO)} + \|u\|_{L^p(\sO)}\right).
\end{equation}
\end{enumerate}
\end{thm}

\begin{rmk}[Alternative statement and proof of the global a priori estimate]
\label{rmk:Alternative_proof_global_apriori_estimate}
The a priori estimate \eqref{eq:Strictly_elliptic_obstacle_problem_global_apriori_W2p_estimate} is also proved by Rodrigues as \cite[Theorem 5.3.4 (i)]{Rodrigues_1987}, although with stronger hypotheses on the coefficients $a^{ij}$ and weaker hypotheses on $b^i, c$, source $f$, boundary data $g$, and obstacle function $\psi$, reflecting his emphasis on variational methods: $a^{ij} \in C^{0,1}(\bar\sO)$, and $b^i, c\in L^\infty(\sO)$, and $f\in L^p(\sO)$, and $g \in W^{2,p}(\sO)$, and $\psi \in W^{1,2}(\sO)$ and $A\psi\in M(\sO)$, the space of bounded Radon measures \cite[p. 61]{Rodrigues_1987}, and $(A\psi-f)^+\in L^p(\sO)$; the remaining hypotheses of \cite[Theorem 5.3.4 (i)]{Rodrigues_1987} coincide with those in Theorem \ref{thm:Strictly_elliptic_obstacle_problem_apriori_W2p_estimate}. While Rodrigues assumes that the bilinear form associated with $A$ obeys a coercivity condition \cite[Equation (5.3.3)]{Rodrigues_1987}, that assumption is not required for the derivation of the a priori estimate. \qed
\end{rmk}

\begin{proof}[Proof of Theorem \ref{thm:Strictly_elliptic_obstacle_problem_apriori_W2p_estimate}]
We first prove that \eqref{eq:Strictly_elliptic_obstacle_problem_global_apriori_W2p_estimate} holds. According to the proof of \cite[Theorem 1.3.2]{Friedman_1982}, there is a $C^\infty$ penalization function $\beta_\eps:\RR\to\RR$, for $\eps\in(0,1]$, in the sense of \cite[Equation (1.3.12)]{Friedman_1982} and a solution $u_\eps \in W^{2,p}(\sO)$ to the penalized equation \cite[Equation (1.3.13)]{Friedman_1982},
$$
Au_\eps + \beta_\eps(u_\eps-\psi) = f \quad\hbox{a.e. on }\sO, \quad u_\eps = g \quad\hbox{on }\partial\sO,
$$
with the properties that \cite[Lemma 1.3.1, Equation (1.3.17)]{Friedman_1982} (after passing to a subsequence and relabeling),
\begin{gather*}
u_\eps \rightharpoonup u \quad\hbox{(weak convergence in $W^{2,p}(\sO)$ as $\eps\downarrow 0$),}
\\
u_\eps \to u \quad\hbox{(strong convergence in $C(\bar\sO)$ as $\eps\downarrow 0$),}
\\
\beta_\eps(u_\eps-\psi) \to (u-\psi)^- \quad\hbox{(pointwise convergence on $\bar\sO$ as $\eps\downarrow 0$)}
\\
\|\beta_\eps(u_\eps-\psi)\|_{C(\bar\sO)} \leq C, \quad\forall\,\eps\in(0,1],
\end{gather*}
for a positive constant $C$ which is independent of $\eps\in(0,1]$. (For example, one may take $\beta_\eps$ to be a $C^\infty$ smoothing of $\RR\ni t\mapsto -t^-/\eps\in\RR$, where $t^- = -\min\{t,0\}$.) A standard a priori interior estimate \cite[Theorem 9.13]{GilbargTrudinger} yields
$$
\|u_\eps\|_{W^{2,p}(\sO)} \leq C\left(\|f\|_{L^p(\sO)} + \|\beta_\eps(u_\eps-\psi)\|_{L^p(\sO)} + \|g\|_{W^{2,p}(\sO)}\right),
$$
where $C=C(d,\sO,p,\lambda_0,\Lambda)$. The penalized equation yields
$$
\beta_\eps(u_\eps-\psi) = Au_\eps - f \quad\hbox{a.e. on }\sO
$$
and therefore by \cite[\S D.4]{Evans} and the Lebesgue Dominated Convergence Theorem we obtain, by taking the limit as $\eps\downarrow 0$,
\begin{align*}
\|u\|_{W^{2,p}(\sO)} &\leq \liminf_{\eps\downarrow 0}\|u_\eps\|_{W^{2,p}(\sO)}
\\
&\leq C\left(\|f\|_{L^p(\sO)} + \liminf_{\eps\downarrow 0}\|\beta_\eps(u_\eps-\psi)\|_{L^p(\sO)} + \|g\|_{W^{2,p}(\sO)}\right)
\\
&= C\left(\|f\|_{L^p(\sO)} + \|(Au - f)^+\|_{L^p(\sO)} + \|g\|_{W^{2,p}(\sO)}\right),
\end{align*}
where $C=C(d,\sO,p,\lambda_0,\Lambda)$. Note that $Au - f\geq 0$ a.e. on $\sO$ by \eqref{eq:Strictly_elliptic_obstacle_boundary_problem} and so $Au - f = (Au - f)^+$ a.e. on $\sO$. Using $Au - f = 0$ a.e. on $\sO\cap\{u>\psi\}$ and $Au - f = A\psi - f$ a.e. on $\sO\cap\{u=\psi\}$, we obtain the global estimate \eqref{eq:Strictly_elliptic_obstacle_problem_global_apriori_W2p_estimate}.

Next, we prove that \eqref{eq:Strictly_elliptic_obstacle_problem_interior_apriori_W2p_estimate} holds. Let $\zeta \in C^\infty_0(\RR^d)$ be a cutoff function with $0\leq\zeta\leq 1$ on $\RR^d$ and $\zeta=1$ on $\sO'$ and $\supp\zeta\subset\sO$, set $\tilde u := \zeta u$, and observe that $\tilde u \in W^{2,p}(\sO)$ is a solution to the obstacle problem,
$$
\min\{A\tilde u - \tilde f, \tilde u - \tilde\psi\} = 0 \quad\hbox{a.e. on }\sO, \quad \tilde u = 0\quad\hbox{on }\partial\sO,
$$
where $\tilde f := \zeta f + [A,\zeta]u\in L^p(\sO)$ and $\tilde\psi := \zeta\psi \in W^{2,p}(\sO)$. For any $C^{1,1}$ domain $U\subset\RR^d$ and $v\in W^{2,p}(U)$ and $\eps>0$, we have the interpolation inequality \cite[Theorem 7.28]{GilbargTrudinger},
$$
\|Dv\|_{L^p(U)} \leq \eps d\|v\|_{W^{2,p}(U)} + C\eps^{-1}\|v\|_{L^p(U)},
$$
for a positive constant $C=C(d,U)$. We now apply the localization procedure employed in the proof of \cite[Theorem 7.1.1]{Krylov_LecturesHolder} to obtain \eqref{eq:Strictly_elliptic_obstacle_problem_interior_apriori_W2p_estimate} from \eqref{eq:Strictly_elliptic_obstacle_problem_global_apriori_W2p_estimate} with $g=0$ on $\partial\sO$.
\end{proof}

We can now relax the regularity conditions on $g$ (and also $\psi$) in Theorem \ref{thm:Strictly_elliptic_obstacle_problem} via an argument modeled after the proof of \cite[Theorem 8.30]{GilbargTrudinger}.

%COMMENT There is subtlety here and in Theorem \ref{thm:Strictly_elliptic_obstacle_problem}, where we obtain uniqueness not just for p \geq d (given by the Aleksandrov weak maximum principle) but also for 1<p<d by analogy with Theorem 9.15 in GT. 
\begin{thm}[Existence and uniqueness of a solution to the obstacle problem with continuous partial Dirichlet boundary condition]
\label{thm:Strictly_elliptic_obstacle_problem_continuous_boundarydata}
Assume the hypotheses of Theorem \ref{thm:Strictly_elliptic_obstacle_problem}, but allow $f\in C^\alpha(\sO)\cap C(\bar\sO)$,
%COMMENT Proof seemed to need f in C(\bar\sO), so we strengthened from f\in C_b(\sO)
and $g\in C(\partial \sO)$, and $\psi\in C^2(\sO)\cap C(\bar \sO)$. Then there is a unique solution $u\in W^{2,p}_{\loc}(\sO)\cap C(\bar \sO)$ to the obstacle problem \eqref{eq:Strictly_elliptic_obstacle_boundary_problem}.
\end{thm}

\begin{proof}
Because $g\in C(\partial \sO)$, by \cite[Lemma 6.38 and Remark 1, p. 137]{GilbargTrudinger} we may extend $g$ to a function in $C(\bar \sO)$, which we shall continue to denote by $g$. Since $f\in C^\alpha(\sO)\cap C(\bar\sO)$, and $g\in C(\bar\sO)$, and $\psi\in C^2(\sO)\cap C(\bar \sO)$, by \cite[Corollary 1.29]{Adams_1975} we may choose sequences $\{f_n\}_{n\in\NN}\subset C^\alpha(\bar \sO)$, and $\{g_n\}_{n\in\NN}\subset C^{2,\alpha}(\bar \sO)$, and $\{\psi_n\}_{n\in\NN}\subset C^{2,\alpha}(\bar \sO)$ such that
$$
f_n \to f \quad\hbox{in }C^\alpha(\sO)\cap C(\bar\sO), \quad g_n \to g \quad\hbox{in }C(\partial \sO),
\quad \psi_n \to \psi \quad\hbox{in }C^2(\sO)\cap C(\bar \sO),
$$
as $n\to\infty$. Let $\{u_n\}_{n\in\NN}\subset W^{2,p}(\sO)$ be the corresponding sequence of unique solutions to the obstacle problem \eqref{eq:Strictly_elliptic_obstacle_boundary_problem} defined by $f_n, g_n, \psi_n$ via Theorem \ref{thm:Strictly_elliptic_obstacle_problem}. Proposition \ref{prop:Elliptic_classical_rodrigues} implies that the sequence $\{u_n\}_{n\in\NN}$ is Cauchy in $C(\bar \sO)$ and thus converges in $C(\bar \sO)$ to a limit $u\in C(\bar \sO)$ which obeys the boundary condition $u=g$ on $\partial\sO$ in \eqref{eq:Strictly_elliptic_obstacle_boundary_problem}.

The a priori interior estimate \eqref{eq:Strictly_elliptic_obstacle_problem_interior_apriori_W2p_estimate}, for any $\sO''\Subset\sO'\Subset\sO$,
gives
\begin{equation}
\label{eq:Strictly_elliptic_obstacle_problem_interior_apriori_W2p_estimate_un}
\|u_n\|_{W^{2,p}(\sO'')} \leq C\left(\|f_n\|_{L^p(\sO')} + \|(A\psi_n-f_n)^+\|_{L^p(\sO')} + \|u_n\|_{L^p(\sO')}\right),
\end{equation}
where now $C=C(d,p,\sO',\sO'',\lambda_0,\Lambda)$. By construction of the sequences, we have
\begin{align*}
\|f_n\|_{L^p(\sO')} &\to \|f\|_{L^p(\sO')},
\\
\|(A\psi_n-f_n)^+\|_{L^p(\sO')} &\to \|(A\psi-f)^+\|_{L^p(\sO')},
\\
\|u_n\|_{L^p(\sO')} &\to \|u\|_{L^p(\sO')},
\end{align*}
as $n\to\infty$. In particular, the right-hand side of \eqref{eq:Strictly_elliptic_obstacle_problem_interior_apriori_W2p_estimate_un} is bounded independently of $n\in\NN$ and so, after passing to a subsequence \cite[Theorem D.3]{Evans}, we may assume that the sequence $\{u_n\}_{n\in\NN}$ converges weakly in $W^{2,p}(\sO'')$ to a limit $u\in W^{2,p}(\sO'')$ (coinciding with the limit $u\in C(\bar \sO)$ already discovered) for each $\sO''\Subset\sO$. Thus, $u \in W^{2,p}_{\loc}(\sO)\cap C(\bar\sO)$.

We next show that $u$ solves the obstacle problem \eqref{eq:Strictly_elliptic_obstacle_boundary_problem}. Since each $u_n$ solves the obstacle problem \eqref{eq:Strictly_elliptic_obstacle_boundary_problem} defined by $f_n, g_n, \psi_n$, we have $u_n\geq \psi$ on $\sO$ and $u_n = g_n$ on $\partial\sO$ and thus, taking limits as $n\to\infty$ and recalling that the sequences $\{u_n\}_{n\in\NN}$, $\{g_n\}_{n\in\NN}$, and $\{\psi_n\}_{n\in\NN}$ converge in $C(\bar \sO)$ to $u$, $g$, and $\psi$, respectively, we obtain
$$
u\geq \psi \quad\hbox{on }\sO \quad\hbox{and}\quad u = g \quad\hbox{on } \partial\sO.
$$
We still need to establish that
\begin{align}
\label{eq:Au_geq_f}
Au &\geq f \quad\hbox{a.e. on } \sO,
\\
\label{eq:Au_equals_f_continuationset}
Au &= f \quad\hbox{a.e. on } \{u>\psi\}\cap\sO.
\end{align}
Because each $u_n$ solves the obstacle problem \eqref{eq:Strictly_elliptic_obstacle_boundary_problem} defined by $f_n, g_n, \psi_n$, we have
\begin{align}
\label{eq:Aun_geq_fn}
Au_n &\geq f_n \quad\hbox{a.e. on } \sO,
\\
\label{eq:Aun_equals_fn_continuationset}
Au_n &= f_n \quad\hbox{a.e. on } \{u_n>\psi_n\}\cap\sO.
\end{align}
Suppose $\sO' \subset \{u>\psi\}\cap\sO$. Because $u_n\to u$ and $\psi_n\to\psi$ in $C(\bar \sO)$ as $n\to\infty$, for a sufficiently large integer $N=N(\sO')$ we may assume that
$$
\sO' \Subset \{u_n>\psi_n\}\cap\sO, \quad\forall\, n \in \NN, \ n \geq N,
$$
and so \eqref{eq:Aun_equals_fn_continuationset} yields
\begin{equation}
\label{eq:Aun_equals_fn_precompactsubset_continuationset}
Au_n = f_n \quad\hbox{a.e. on } \sO', \quad\forall\, n \in \NN, \ n \geq N.
\end{equation}
The a priori interior estimate  \cite[Theorem 9.13]{GilbargTrudinger} for a solution to the strictly elliptic equation \eqref{eq:Aun_equals_fn_precompactsubset_continuationset} gives
$$
\|u_n\|_{W^{2,p}(\sO'')} \leq C\left(\|f_n\|_{L^p(\sO')} + \|u_n\|_{L^p(\sO')}\right), \quad\forall\, n \in \NN, \ n \geq N,
$$
for a constant $C$ depending on $d,\sO',\sO'',p,\lambda_0,\Lambda$ and the moduli of continuity of the coefficients $a^{ij}$. Thus, for all $m,n \in \NN$ with $m,n \geq N$, we have
$$
\|u_n-u_m\|_{W^{2,p}(\sO'')} \leq C\left(\|f_n-f_m\|_{L^p(\sO')} + \|u_n-u_m\|_{L^p(\sO')}\right).
$$
Therefore, $\{u_n\}_{n\in\NN}$ is Cauchy in $W^{2,p}(\sO'')$ and so converges strongly in $W^{2,p}(\sO'')$ to the previously established weak limit $u\in W^{2,p}(\sO'')$. Thus, we can take the limit in \eqref{eq:Aun_equals_fn_precompactsubset_continuationset} as $n\to\infty$ to give
$$
Au = f \quad\hbox{a.e. on } \sO''.
$$
Since $\sO'' \Subset \sO'$ and $\sO' \Subset \{u>\psi\}\cap\sO$ were arbitrary, we see that $u$ solves \eqref{eq:Au_equals_f_continuationset}.

Moreover, for any $\varphi \in C^\infty_0(\sO)$ such that $\varphi\geq 0$ on $\sO$, the inequality \eqref{eq:Aun_geq_fn} yields
$$
\int_\sO(Au_n-f_n)\varphi \,dx \geq 0, \quad\forall\, n \in \NN.
$$
Hence, recalling that for any $\sO''\Subset\sO$ we have
$$
u_n \rightharpoonup u \quad\hbox{weakly in } W^{2,p}(\sO''),
$$
and taking the limit as $n\to\infty$, the limit $u\in W^{2,p}_{\loc}(\sO)\cap C(\bar\sO)$ necessarily obeys
$$
(Au-f, \varphi)_{L^2(\sO)} \geq 0.
$$
Because the function $\varphi \in C^\infty_0(\sO)$ with $\varphi\geq 0$ on $\sO$ was arbitrary, $u$ must also solve \eqref{eq:Au_geq_f} and this completes the proof of existence.

Uniqueness\footnote{When $\sO$ is a domain, thus connected, we could also appeal to our comparison principle (Theorem \ref{thm:Elliptic_obstacle_problem_comparison_principle}) to obtain uniqueness.} of the solution, $u$, follows from the Aleksandrov weak maximum principle \cite[Theorem 9.1]{GilbargTrudinger} and the comparison principle \cite[Theorem 1.3.3]{Friedman_1982} for solutions to the obstacle problem \eqref{eq:Strictly_elliptic_obstacle_boundary_problem} when\footnote{For $1<p\leq d$, uniqueness of solutions to a Dirichlet boundary value problem follows by the proof of uniqueness in \cite[Theorem 9.15]{GilbargTrudinger} --- see \cite[p. 242]{GilbargTrudinger}.} $p\geq d$.
\end{proof}

The partial Dirichlet boundary value problem for a strictly elliptic operator $A$ as in \eqref{eq:Strictly_elliptic_operator} on a bounded domain $\sO$ with continuous boundary data $g\in C(\partial\sO)$ only requires that the boundary $\partial\sO$ obey an \emph{exterior} sphere condition \cite[Theorem 6.13]{GilbargTrudinger} or, more generally, that each boundary point $x_0\in \partial\sO$ is \emph{regular} for the given $f\in C^\alpha(\sO)\cap C(\bar\sO)$, and $g\in C(\partial\sO)$, and $A$, in that a \emph{local barrier} exists at $x_0$ in the sense of \cite[pp. 104--106]{GilbargTrudinger}. However, we can easily deduce the following corollary of Theorem \ref{thm:Strictly_elliptic_obstacle_problem_continuous_boundarydata} when $\partial\sO$ obeys an \emph{interior} sphere condition at each point by allowing the balls in $\sO$ used in our construction of the Perron solution \eqref{eq:Perron_elliptic_obstacle_problem_solution} for \eqref{eq:Elliptic_obstacle_problem}, \eqref{eq:Elliptic_boundary_condition} to touch
the boundary\footnote{A related idea is used in the proof of \cite[Lemma 6.18]{GilbargTrudinger}.} $\partial\sO$.

\begin{cor}[Existence and uniqueness of a solution to the obstacle problem with continuous boundary data]
\label{cor:Strictly_elliptic_obstacle_problem_continuous_boundarydata}
Assume the hypotheses of Theorem \ref{thm:Strictly_elliptic_obstacle_problem_continuous_boundarydata}, except that $\partial\sO$ is now only required to obey an interior sphere condition at each point. Then the conclusions of Theorem \ref{thm:Strictly_elliptic_obstacle_problem_continuous_boundarydata} continue to hold.
\end{cor}

According to \cite[Theorem 1.3.4]{Friedman_1982}, the condition $\psi \in C^2(\bar\sO)$ in Theorem \ref{thm:Strictly_elliptic_obstacle_problem} may be weakened to a requirement that $\psi\in C(\bar\sO)$ be globally Lipschitz and have globally finite concavity on $\bar\sO$ in the sense that,
\begin{gather}
\label{eq:Elliptic_obstaclefunction_global_lipschitz_condition}
\psi \in C^{0,1}(\tilde\sO),
\\
\label{eq:Lipschitz_elliptic_obstaclefunction_finite_concavity_condition}
D^2_\xi\psi \geq - C \quad\hbox{in }\sD'(\tilde\sO), \quad\forall\, \xi\in \RR^d, \ |\xi|=1,
\end{gather}
for some open neighborhood $\tilde\sO\subset\RR^d$ of $\bar\sO$ and some positive constant $C=C(\tilde\sO)$. (Compare the local conditions \eqref{eq:Elliptic_obstaclefunction_lipschitz_condition} and \eqref{eq:Elliptic_obstaclefunction_locally_finite_concavity_condition}.) We now recall the

\begin{thm}[Existence of solutions to an obstacle problem with a globally Lipschitz obstacle function having globally finite concavity]
\label{thm:Strictly_elliptic_OP_Lipschitz_obstaclefunction}
\cite[Theorem 1.3.5]{Friedman_1982}
The conclusion of Theorem \ref{thm:Strictly_elliptic_obstacle_problem} remains valid if the condition $\psi\in C^2(\bar\sO)$ is relaxed to $\psi \in C(\bar\sO)$ obeying \eqref{eq:Elliptic_obstaclefunction_global_lipschitz_condition} and \eqref{eq:Lipschitz_elliptic_obstaclefunction_finite_concavity_condition}.
\end{thm}

Theorem \ref{thm:Strictly_elliptic_OP_Lipschitz_obstaclefunction} is proved with the aid of a technical lemma which we may in turn apply to relax the regularity condition on $\psi$ in Theorem \ref{thm:Strictly_elliptic_obstacle_problem_continuous_boundarydata}. We first recall the mollification procedure in \cite[pp. 27--28]{Friedman_1982}. Let $\zeta \in C^\infty_0(\RR^d)$ be a cutoff function with $\supp\zeta \subset B_1(O)$ and $\zeta\geq 0$ on $\RR^d$ and
$$
\int_{\RR^d}\zeta\,dx = 1.
$$
Set $\zeta_\delta(x) = \delta^{-d}\zeta(x/\delta)$, for any $x\in\RR^d$ and $\delta>0$, and for any $v\in L^p(\sO)$, with $1<p<\infty$, set
$$
(J_\delta v)(x) := \int_\sO \zeta_\delta(x-y)v(y)\,dy, \quad x\in\RR^d.
$$
Then $J_\delta v \in C^\infty(\RR^d)$ and $\|J_\delta v - v\|_{L^p(\sO')} \to 0$ as $\delta \downarrow 0$, for any $\sO'\Subset\sO$ (see \cite[p. 27]{Friedman_1982} or \cite[p. 71 and Lemma 7.1]{GilbargTrudinger}).

For $\psi \in C(\bar\sO)$ obeying \eqref{eq:Elliptic_obstaclefunction_global_lipschitz_condition} and \eqref{eq:Lipschitz_elliptic_obstaclefunction_finite_concavity_condition} and for $\delta>0$, define
\begin{equation}
\label{eq:Mollified_obstacle_function}
\psi_\delta := J_\delta\left(\psi + \frac{1}{2}C|x|^2\right) - \frac{1}{2}C|x|^2.
\end{equation}
Then the mollified obstacle function $\psi_\delta$ obeys \cite[Equation (1.3.22)]{Friedman_1982}
\begin{equation}
\label{eq:mollified_obstacle_function_properties}
\begin{aligned}
|D\psi_\delta| &\leq C \quad\hbox{on }\sO,
\\
D^2_\xi\psi_\delta &\geq -C \quad\hbox{on }\sO,
\\
\|\psi_\delta - \psi\|_{C(\bar\sO)} &\to 0, \quad \delta\downarrow 0.
\end{aligned}
\end{equation}
We now recall the

\begin{lem}[Upper bound on $A\psi_\delta$]
\cite[Equation (1.3.22)]{Friedman_1982}
\label{lem:Upper_bound_A_mollified_obstacle_function}
Let $\sO\subset\RR^d$ be a bounded open subset and assume $\psi \in C(\bar\sO)$ obeys \eqref{eq:Elliptic_obstaclefunction_global_lipschitz_condition} and \eqref{eq:Lipschitz_elliptic_obstaclefunction_finite_concavity_condition}. Then
$$
A\psi_\delta \leq C \quad\hbox{on }\sO, \quad\forall\, \delta>0,
$$
for a positive constant $C$ independent of $\delta>0$.
\end{lem}

We then have the

\begin{thm}[Existence and uniqueness of a solution to the obstacle problem with continuous partial Dirichlet boundary condition and Lipschitz obstacle function with finite concavity]
\label{thm:Strictly_elliptic_obstacle_problem_continuous_boundarydata_lipschitz_obstacle}
The conclusion of Theorem \ref{thm:Strictly_elliptic_obstacle_problem_continuous_boundarydata} remains valid if the condition $\psi\in C^2(\sO)\cap C(\bar\sO)$ is relaxed to $\psi \in C(\bar\sO)$ obeying \eqref{eq:Elliptic_obstaclefunction_global_lipschitz_condition} and \eqref{eq:Lipschitz_elliptic_obstaclefunction_finite_concavity_condition}.
\end{thm}

\begin{proof}
For any $\delta>0$, let $\psi_\delta\in C^\infty(\RR^d)$ be as in \eqref{eq:Mollified_obstacle_function}, let $g_\delta := g\vee\psi_\delta \in C(\partial\sO)$, and let $u_\delta \in W^{2,p}_{\loc}(\sO)\cap C(\bar\sO)$ be the corresponding solution to \eqref{eq:Strictly_elliptic_obstacle_boundary_problem} defined by $f$, $g_\delta$, and $\psi_\delta$ via Theorem \ref{thm:Strictly_elliptic_obstacle_problem_continuous_boundarydata}. Proposition \ref{prop:Elliptic_classical_rodrigues} implies that the sequence $\{u_\delta\}$ (with $\delta\in (0,1]$ and $\delta\downarrow 0$) is Cauchy in $C(\bar \sO)$ and thus converges in $C(\bar \sO)$, as $\delta\downarrow 0$, to a limit $u\in C(\bar \sO)$ which obeys the boundary condition $u=g$ on $\partial\sO$ in \eqref{eq:Strictly_elliptic_obstacle_boundary_problem}.

The a priori interior estimate \eqref{eq:Strictly_elliptic_obstacle_problem_interior_apriori_W2p_estimate}, for any $\sO''\Subset\sO'\Subset\sO$,
gives
\begin{equation}
\label{eq:Strictly_elliptic_obstacle_problem_interior_apriori_W2p_estimate_udelta}
\|u_\delta\|_{W^{2,p}(\sO'')} \leq C\left(\|f\|_{L^p(\sO')} + \|(A\psi_\delta-f)^+\|_{L^p(\sO')} + \|u_\delta\|_{L^p(\sO')}\right),
\end{equation}
where $C=C(d,p,\sO',\sO'',\lambda_0,\Lambda)$. We have, as $\delta\downarrow 0$,
$$
\|u_\delta\|_{L^p(\sO')} \to \|u\|_{L^p(\sO')},
$$
while Lemma \ref{lem:Upper_bound_A_mollified_obstacle_function} yields, for some constant $C>0$,
$$
\|(A\psi_\delta-f)^+\|_{L^p(\sO')} \leq \|(C-f)^+\|_{L^p(\sO')}, \quad\forall\, \delta > 0.
$$
In particular, the right-hand side of \eqref{eq:Strictly_elliptic_obstacle_problem_interior_apriori_W2p_estimate_udelta} is bounded independently of $\delta>0$ and so, after passing to a subsequence \cite[Theorem D.3]{Evans}, we may assume that the sequence $\{u_\delta\}$ converges weakly in $W^{2,p}(\sO'')$, as $\delta\downarrow 0$, to a limit $u\in W^{2,p}(\sO'')$ (coinciding with the limit $u\in C(\bar \sO)$ already discovered) for each $\sO''\Subset\sO$. Thus, $u \in W^{2,p}_{\loc}(\sO)\cap C(\bar\sO)$.

The proofs that $u$ solves the obstacle problem \eqref{eq:Strictly_elliptic_obstacle_boundary_problem} and is unique are the same as those in Theorem \ref{thm:Strictly_elliptic_obstacle_problem_continuous_boundarydata}.
\end{proof}

\section{Transformation of a model operator on the half-plane by the conformal map from a half-disk to a strip}
\label{sec:Holomorphic_map_model_operator}
In dimension two, the effect of the map in \S \ref{subsec:Diffeomorphism_halfball_slab} from the half-disk to the strip on an operator $A$ of the form \eqref{eq:Generator} can be made explicit and this provides many useful insights, so we shall describe the calculation here. We observe that for $w$ as in \eqref{eq:Defn_z_to_w_complex_plane},
$$
w = \Log\left(\frac{1+z}{1-z}\right) \equiv s + i\theta,
$$
we have
$$
\frac{dw}{dz} = \frac{d}{dz}\Log\left(\frac{1+z}{1-z}\right) = \frac{2(1-z)}{(1+z)^3},
$$
and so $dw/dz\neq 0$ and is well-defined for $z\neq \pm 1$.

We first calculate the effect of the transformation on the Laplace operator. Writing $v(w) = u(z)$, with $z=x+iy$ and $w=s+i\theta = \Log z$, and $\Delta_z u(x,y) = -u_{xx} - u_{yy}$ and $\Delta_w v(s,\theta) = -v_{ss} - v_{\theta\theta}$, we have
\begin{align*}
\Delta_z u
&=
-\frac{\partial^2 u}{\partial z\partial\bar z}
=
-\frac{\partial }{\partial z}\left(\frac{\partial v}{\partial\bar w}\frac{\partial \bar w}{\partial\bar z}\right)
\\
&=
-\frac{\partial^2 v}{\partial w\partial\bar w}\frac{\partial w}{\partial z}\frac{\partial \bar w}{\partial\bar z} - \frac{\partial v}{\partial\bar w}\frac{\partial^2 \bar w}{\partial z\partial\bar z}
=
-\frac{\partial^2 v}{\partial w\partial\bar w}\frac{\partial w}{\partial z}\frac{\partial \bar w}{\partial\bar z}
\\
&= \Delta_w v \frac{\partial w}{\partial z}\frac{\partial \bar w}{\partial\bar z}
= \left|\frac{\partial w}{\partial z}\right|^2 \Delta_w v,
\end{align*}
noting that $\partial w/\partial\bar z = 0$, and thus we obtain
\begin{equation}
\label{eq:Laplace_operator_transformation_complex}
\Delta_z u = 2\frac{|1-z|^2}{|1+z|^6} \Delta_w v.
\end{equation}
We wish to use the preceding change of complex variables to transform the equation \eqref{eq:Elliptic_equation} on the unit half-disk,
$$
Au = f \quad\hbox{on } B^+,
$$
defined by a model operator,
\begin{equation}
\label{eq:Generator_model2d_halfdisk}
Au := -y(u_{xx} + u_{yy}) - b^1u_x - b^2u_y + cu,
\end{equation}
where the coefficients $b^1$, $b^2$, $c$ are constant (for simplicity in this example) and obey $b^2>0$ and $c\geq 0$, to an equation on the strip,
$$
\tilde Av = \tilde f \quad\hbox{on } S,
$$
where
\begin{equation}
\label{eq:Generator_model2d_strip}
\tilde Av := -\theta\tilde a (v_{ss} + v_{\theta \theta }) - \tilde b^1v_s - \tilde b^2v_\theta  + \tilde cv,
\end{equation}
where $v(s,\theta) = u(x,y)$ and $\tilde f(s,\theta) = f(x,y)$, and the coefficients $\tilde a$, $\tilde b^1$, $\tilde b^2$, $\tilde c$ are defined on the strip, $S \equiv S^2_{\pi/2}=\RR\times(0, \pi/2)$ and are to be determined.

For this purpose, we shall need explicit formulae for $x=x(s,\theta)$ and $y=y(s,\theta)$. Note that
$$
\frac{1+z}{1-z} = \frac{(1+z)(1-\bar z)}{(1-z)(1-\bar z)} =  \frac{1 - x^2 - y^2 + 2iy}{(1-x)^2 + y^2} = e^{s+i\theta} = e^s\cos\theta + ie^s\sin\theta,
$$
and thus
$$
e^s = \left|\frac{1+z}{1-z}\right| = \frac{\sqrt{(1+x)^2 + y^2}}{\sqrt{(1-x)^2 + y^2}} \quad\hbox{and}\quad \theta = \Arg\left(\frac{1+z}{1-z}\right),
$$
with
\begin{align*}
e^s\cos\theta = \frac{1 - x^2 - y^2}{(1-x)^2 + y^2}
\quad\hbox{and}\quad
e^s\sin\theta = \frac{2y}{(1-x)^2 + y^2}.
\end{align*}
Therefore,
\begin{align}
\label{eq:s_function_x_and_y}
s &= \frac{1}{2}\ln\left(\frac{(1+x)^2 + y^2}{(1-x)^2 + y^2}\right),
\\
\label{eq:theta_function_x_and_y}
\theta &= \arctan\left(\frac{2y}{1 - x^2 - y^2}\right).
\end{align}
We shall also need an expression for $y$ in terms of $(s,\theta)$. We first solve \eqref{eq:Defn_z_to_w_complex_plane} for $z$ in terms of $w$ to give
$$
z = \frac{e^w-1}{e^w+1}, \quad -\pi <\Imag(w)\leq \pi.
$$
Noting that
$$
e^w = e^{s+i\theta} = e^s\cos\theta + ie^s\sin\theta \quad\hbox{and}\quad |e^w| = e^s,
$$
this yields
\begin{align*}
z &= \frac{(e^w-1)(e^{\bar w}+1)}{|e^w+1|^2} = \frac{|e^w|^2 -1 + i2\Imag(e^w)}{|e^w+1|^2}
\\
&= \frac{e^{2s}-1 + i2\sin\theta}{(1+e^s\cos\theta)^2 + e^{2s}\sin^2\theta}
= \frac{e^{2s}-1 + i2\sin\theta}{1 + 2e^s\cos\theta + e^{2s}}.
\end{align*}
Since $z=x+iy$, we obtain
\begin{align}
\label{eq:x_function_s_theta}
x &= \frac{e^{2s}-1}{1 + 2e^s\cos\theta + e^{2s}},
\\
\label{eq:y_function_s_theta}
y &= \frac{2\sin\theta}{1 + 2e^s\cos\theta + e^{2s}}.
\end{align}
We have the change of variable formulae,
\begin{equation}
\label{eq:Change_of_variable_formulae}
u_x = v_s\frac{\partial s}{\partial x} + v_\theta\frac{\partial\theta}{\partial x}
\quad\hbox{and}\quad
u_y = v_s\frac{\partial s}{\partial y} + v_\theta\frac{\partial\theta}{\partial y}.
\end{equation}
A calculation using \eqref{eq:s_function_x_and_y} and \eqref{eq:theta_function_x_and_y} yields
%COMMENT Mathematica
\begin{align}
\label{eq:Derivative_s_wrt_x}
\frac{\partial s}{\partial x} &= \frac{-2 x^2+2 y^2+2}{x^4+2 x^2
   \left(y^2-1\right)+\left(y^2+1\right)^2},
\\
\label{eq:Derivative_s_wrt_y}
\frac{\partial s}{\partial y} &= -\frac{4 x y}{x^4+2 x^2
   \left(y^2-1\right)+\left(y^2+1\right)^2}.
\end{align}
Writing
$$
D^2 = (1 - x^2 - y^2)^2 + 4y^2 = x^4 + 2x^2(y^2-1) + (y^2+1)^2,
$$
we have
$$
\cos\theta = \frac{1 - x^2 - y^2}{D} \quad\hbox{and}\quad \sin\theta = \frac{2y}{D},
$$
and the useful expression
\begin{equation}
\label{eq:Expression_D}
D = \frac{2y}{\sin\theta} = \frac{4}{1 + 2e^s\cos\theta + e^{2s}}.
\end{equation}
We see that the equations \eqref{eq:Derivative_s_wrt_x} and \eqref{eq:Derivative_s_wrt_y} for the derivatives simplify to give
\begin{align}
\label{eq:Derivative_s_wrt_x_simpler}
\frac{\partial s}{\partial x} &= \frac{2(1-x^2-y^2) + 4y^2}{D^2} = \frac{2\cos\theta}{D} + \sin^2\theta,
\\
\label{eq:Derivative_s_wrt_y_simpler}
\frac{\partial s}{\partial y} &=  \frac{-4xy}{D^2} = \frac{-2x\sin\theta}{D}.
\end{align}
Since the map $z=x+iy\mapsto w=s+i\theta$ in \eqref{eq:Defn_z_to_w_complex_plane} is holomorphic, we must have
%COMMENT Checked by Mathematica using the arctangent formula
$$
\frac{\partial\theta}{\partial x} = -\frac{\partial s}{\partial y}
\quad\hbox{and}\quad
\frac{\partial\theta}{\partial y} = \frac{\partial s}{\partial x}.
$$
We apply these derivative formulae in \eqref{eq:Change_of_variable_formulae} to give
\begin{align*}
b^1u_x + b^2u_y &= b^1(v_s s_x + v_\theta \theta_x) + b^2(v_s s_y + v_\theta \theta_y)
\\
&= b^1(v_s s_x - v_\theta s_y) + b^2(v_s s_y + v_\theta s_x)
\\
&= ( b^1 s_x + b^2 s_y )v_s + (b^1 \theta_x + b^2 \theta_y)v_\theta
\\
&= ( b^1 s_x + b^2 s_y )v_s + (-b^1 s_y + b^2 s_x )v_\theta.
\end{align*}
Finally, note that the expressions for $e^s$ and $e^s\sin\theta$ in terms of $x$ and $y$
\begin{align*}
2\frac{|1-z|^2}{|1+z|^6} &= 2\frac{(1-x)^2 + y^2}{((1+x)^2 + y^2)^3}
=  2\frac{((1-x)^2 + y^2)^3}{((1+x)^2 + y^2)^3}\frac{1}{((1-x)^2 + y^2)^2}
\\
&= \frac{2}{e^{6s}}\left(\frac{e^s\sin\theta}{2y}\right)^2 = \frac{\sin^2\theta}{2y^2e^{4s}}.
\end{align*}
and combining the preceding identity with \eqref{eq:Laplace_operator_transformation_complex} yields
\begin{equation}
\label{eq:Laplace_operator_transformation_real}
u_{xx} + u_{yy} = \frac{\sin^2\theta}{2y^2e^{4s}} (v_{ss}+v_{\theta\theta}).
\end{equation}
We now apply the preceding formulae to the equation $Au=f$ on the half-disk, $B^+$. We see that on the strip, $S$, the function $v(s,\theta)=u(x,y)$ obeys
\begin{align*}
Au &\equiv -y(u_{xx} + u_{xx}) - b^1u_x - b^2u_y + cu
\\
&\quad = -\frac{\sin^2\theta}{2y^2e^{4s}}y(v_{ss}+v_{\theta\theta}) - (b^1 s_x + b^2 s_y)v_s - (-b^1 s_y + b^2s_x)v_\theta + cv =  f,
\end{align*}
and so, denoting $\tilde f(s,\theta)=f(x,y)$, the equation $\tilde Av=\tilde f$ on the strip, $S$, takes the explicit form,
\begin{equation}
\label{eq:Generator_model2d_transformed_disk_strip}
\begin{aligned}
\tilde Av &\equiv -\theta\tilde a(v_{ss} + v_{\theta \theta }) - \tilde b^1v_s - \tilde b^2v_\theta  + \tilde cv
\\
&= -\frac{\sin^2\theta}{2ye^{4s}}(v_{ss}+v_{\theta\theta}) - \left(b^1\left(\frac{2\cos\theta}{D} + \sin^2\theta\right) - b^2\frac{2x\sin\theta}{D}\right)v_s
\\
&\qquad - \left(b^1\frac{2x\sin\theta}{D} + b^2\left(\frac{2\cos\theta}{D} + \sin^2\theta\right)\right)v_\theta + cv = \tilde f.
\end{aligned}
\end{equation}
As $\theta\downarrow 0$, with $s\in\RR$ fixed, the formula \eqref{eq:Expression_D} gives
$$
D(s,\theta) \to \frac{4}{1 + 2e^s + e^{2s}} = \frac{4}{(1 + e^s)^2},
$$
and the expressions \eqref{eq:x_function_s_theta} and \eqref{eq:Expression_D} give
$$
x(s,\theta)\to \frac{e^{2s}-1}{1 + 2e^s + e^{2s}} \quad\hbox{and}\quad y(s,\theta) = \frac{D}{2}\sin\theta,
$$
so the coefficients in the equation \eqref{eq:Generator_model2d_transformed_disk_strip} converge to
$$
-\left(\frac{1}{De^{4s}}\frac{\sin\theta}{\theta}\right)\theta (v_{ss}+v_{\theta\theta}) - b^1\frac{2}{D}v_s - b^2\frac{2}{D}v_\theta + cv = \tilde f,
$$
that is, using $\sin\theta/\theta \to 1$ as $\theta\downarrow 0$,
$$
-\frac{(1 + e^s)^2}{4e^{4s}}\theta (v_{ss}+v_{\theta\theta}) - \frac{1}{2}(1 + e^s)^2b^1v_s
- \frac{1}{2}(1 + e^s)^2b^2v_\theta + cv = \tilde f.
$$
On the other hand, as $\theta \uparrow \pi/2$, with $s\in\RR$ fixed, we have from \eqref{eq:Expression_D},
$$
D(s,\theta) \to \frac{4}{1 + e^{2s}},
$$
and \eqref{eq:x_function_s_theta} and \eqref{eq:y_function_s_theta} give
$$
x(s,\theta) \to \frac{e^{2s}-1}{1 + e^{2s}} \quad\hbox{and}\quad y(s,\theta) \to \frac{2}{1 + e^{2s}},
$$
so the coefficients in the equation \eqref{eq:Generator_model2d_transformed_disk_strip} converge to
$$
-\frac{1}{2ye^{4s}} (v_{ss}+v_{\theta\theta}) - \left(b^1 - b^2\frac{2x}{D}\right)v_s - \left(b^1\frac{2x}{D} + b^2\right)v_\theta + cv = \tilde f,
$$
that is,
$$
-\frac{(1 + e^{2s})}{4e^{4s}} (v_{ss}+v_{\theta\theta}) - \left(b^1 - \frac{1}{2}(e^{2s}-1)b^2\right)v_s - \left(\frac{1}{2}(e^{2s}-1)b^1 + b^2\right)v_\theta + cv = \tilde f.
$$
In particular, the coefficients of the operator in \eqref{eq:Generator_model2d_transformed_disk_strip} extend continuously from $(s,\theta) \in S = \RR\times (0, \pi/2)$ to $(s,\theta) \in \bar S = \RR\times[0, \pi/2]$.

However, as we can see most easily in the limiting cases, $\theta\downarrow 0$ or $\theta\uparrow \pi/2$, while the coefficient $\tilde a$ in \eqref{eq:Generator_model2d_transformed_disk_strip} is uniformly bounded above on $S$, it is not strictly elliptic on $S$ (uniformly bounded below by a positive constant) as $s\to\infty$ and the coefficients $\tilde b^1, \tilde b^2$ are not uniformly bounded on $S$ as $s\to\infty$.

%%%%%%%%%%%%%%%%%%%%%%%%%%%%%%%%%%%%%%%%%%%%%%%%%%%%%%%%%%%%%%%%%%%%%%%%%%%%%%%
%
%                                bibliography
%
%%%%%%%%%%%%%%%%%%%%%%%%%%%%%%%%%%%%%%%%%%%%%%%%%%%%%%%%%%%%%%%%%%%%%%%%%%%%%%%

\bibliography{mfpde}

\def\cprime{$'$} \def\polhk#1{\setbox0=\hbox{#1}{\ooalign{\hidewidth
  \lower1.5ex\hbox{`}\hidewidth\crcr\unhbox0}}} \def\cprime{$'$}
  \def\cprime{$'$} \def\cprime{$'$}
  \def\lfhook#1{\setbox0=\hbox{#1}{\ooalign{\hidewidth
  \lower1.5ex\hbox{'}\hidewidth\crcr\unhbox0}}} \def\cprime{$'$}
  \def\cprime{$'$} \def\cprime{$'$} \def\cprime{$'$}
\providecommand{\bysame}{\leavevmode\hbox to3em{\hrulefill}\thinspace}
\providecommand{\MR}{\relax\ifhmode\unskip\space\fi MR }
\providecommand{\MRhref}[2]{%
  \href{http://www.ams.org/mathscinet-getitem?mr=#1}{#2}
}
\providecommand{\href}[2]{#2}
\begin{thebibliography}{10}

\bibitem{AbramStegun}
M.~Abramovitz and I.~A. Stegun, \emph{Handbook of mathematical functions},
  Dover, New York, 1972.

\bibitem{Adams_1975}
R.~A. Adams, \emph{Sobolev spaces}, Academic Press, Orlando, FL, 1975.

\bibitem{Barles_1993}
G.~Barles, \emph{Fully nonlinear {N}eumann type boundary conditions for
  second-order elliptic and parabolic equations}, J. Differential Equations
  \textbf{106} (1993), 90--106.

\bibitem{Bensoussan_Lions}
A.~Bensoussan and J.~L. Lions, \emph{Applications of variational inequalities
  in stochastic control}, North-Holland, New York, 1982.

\bibitem{Caffarelli_jfa_1998}
L.~A. Caffarelli, \emph{The obstacle problem revisited}, J. Fourier Anal. Appl.
  \textbf{4} (1998), 383--402.

\bibitem{Crandall_Ishii_Lions_1992}
M.~G. Crandall, H.~Ishii, and P-L. Lions, \emph{User's guide to viscosity
  solutions of second order partial differential equations}, Bull. Amer. Math.
  Soc. (N.S.) \textbf{27} (1992), 1--67.

\bibitem{Daskalopoulos_Feehan_optimalregstatheston}
P.~Daskalopoulos and P.~M.~N. Feehan, \emph{${C}^{1,1}$ regularity for
  degenerate elliptic obstacle problems in mathematical finance},
  arXiv:1206.0831.

\bibitem{Daskalopoulos_Feehan_statvarineqheston}
\bysame, \emph{Existence, uniqueness, and global regularity for variational
  inequalities and obstacle problems for degenerate elliptic partial
  differential operators in mathematical finance}, arXiv:1109.1075.

\bibitem{DaskalHamilton1998}
P.~Daskalopoulos and R.~Hamilton, \emph{{$C^\infty$}-regularity of the free
  boundary for the porous medium equation}, J. Amer. Math. Soc. \textbf{11}
  (1998), 899--965.

\bibitem{Epstein_Mazzeo_2011}
C.~L. Epstein and R.~Mazzeo, \emph{Degenerate diffusion operators arising in
  population biology},  (2011), 341 pages, \url{arXiv:1110.0032}.

\bibitem{Evans}
L.~C. Evans, \emph{Partial differential equations}, American Mathematical
  Society, Providence, RI, 1998.

\bibitem{Feehan_classical_perron_parabolic}
P.~M.~N. Feehan, \emph{A classical {P}erron method for existence of smooth
  solutions to boundary value and obstacle problems for degenerate-parabolic
  operators via holomorphic maps}, in preparation.

\bibitem{Feehan_maximumprinciple_v1}
\bysame, \emph{Partial differential operators with non-negative characteristic
  form, maximum principles, and uniqueness for boundary value and obstacle
  problems}, Communications in Partial Differential Equations, to appear,
  arXiv:1204.6613v1.

\bibitem{Feehan_perturbationlocalmaxima}
\bysame, \emph{Perturbations of local maxima and comparison principles for
  degenerate differential equations with partial boundary conditions}, in
  preparation.

\bibitem{Feehan_Pop_regularityweaksoln}
P.~M.~N. Feehan and C.~A. Pop, \emph{Degenerate elliptic operators in
  mathematical finance and {H\"o}lder continuity for solutions to variational
  equations and inequalities}, arXiv:1110.5594.

\bibitem{Feehan_Pop_higherregularityweaksoln}
\bysame, \emph{Higher-order regularity for solutions to degenerate elliptic
  variational equations in mathematical finance}, arXiv:1208.2658.

\bibitem{Feehan_Pop_elliptichestonschauder}
\bysame, \emph{Schauder a priori estimates and regularity of solutions to
  degenerate-elliptic linear second-order partial differential equations},
  arXiv:1210.6727.

\bibitem{Feehan_Pop_mimickingdegen_pde}
\bysame, \emph{A {S}chauder approach to degenerate-parabolic partial
  differential equations with unbounded coefficients}, Journal of Differential
  Equations \textbf{254} (2013), 4401--4445,
  dx.doi.org/10.1016/j.jde.2013.03.006, arXiv:1112.4824.

\bibitem{Friedman_1982}
A.~Friedman, \emph{Variational principles and free boundary problems}, Wiley,
  New York, 1982, reprinted by Dover, New York, 2010.

\bibitem{GilbargTrudinger}
D.~Gilbarg and N.~Trudinger, \emph{Elliptic partial differential equations of
  second order}, second ed., Springer, New York, 1983.

\bibitem{Han_Lin_2011}
Q.~Han and F.~Lin, \emph{Elliptic partial differential equations}, second ed.,
  Courant Lecture Notes in Mathematics, vol.~1, Courant Institute of
  Mathematical Sciences, New York, 2011.

\bibitem{Heston1993}
S.~Heston, \emph{A closed-form solution for options with stochastic volatility
  with applications to bond and currency options}, Review of Financial Studies
  \textbf{6} (1993), 327--343.

\bibitem{Ishii_1989}
H.~Ishii, \emph{On uniqueness and existence of viscosity solutions of fully
  nonlinear second-order elliptic {PDE}s}, Comm. Pure Appl. Math. \textbf{42}
  (1989), 15--45.

\bibitem{Ishii_1991}
\bysame, \emph{Fully nonlinear oblique derivative problems for nonlinear
  second-order elliptic {PDE}s}, Duke Math. J. \textbf{62} (1991), 633--661.

\bibitem{Jensen_1980}
R.~Jensen, \emph{Boundary regularity for variational inequalities}, Indiana
  Univ. Math. J. \textbf{29} (1980), 495--504.

\bibitem{Koch}
H.~Koch, \emph{Non-{E}uclidean singular integrals and the porous medium
  equation}, Habilitation Thesis, University of Heidelberg, 1999,
  \url{www.mathematik.uni-dortmund.de/lsi/koch/publications.html}.

\bibitem{Krylov_LecturesHolder}
N.~V. Krylov, \emph{Lectures on elliptic and parabolic equations in {H}\"older
  spaces}, American Mathematical Society, Providence, RI, 1996.

\bibitem{Lindqvist_2006}
P.~Lindqvist, \emph{Notes on the {$p$}-{L}aplace equation}, Report. University
  of Jyv\"askyl\"a Department of Mathematics and Statistics, vol. 102,
  University of Jyv\"askyl\"a, Jyv\"askyl\"a, 2006,
  \url{math.ntnu.no/~lqvist/p-laplace.pdf}.

\bibitem{Merton1973}
R.~Merton, \emph{A theory of rational option pricing}, Bell Journal of
  Economics and Management Science \textbf{4} (1973), 141--183.

\bibitem{Rodrigues_1987}
J-F. Rodrigues, \emph{Obstacle problems in mathematical physics},
  North-Holland, New York, 1987.

\bibitem{Shreve2}
S.~E. Shreve, \emph{Stochastic calculus for finance. {V}olume {II}:
  {C}ontinuous-time models}, Springer, New York, 2004.

\bibitem{Troianiello}
G.~M. Troianiello, \emph{Elliptic differential equations and obstacle
  problems}, Plenum Press, New York, 1987.

\end{thebibliography}
\bibliographystyle{amsplain}

\end{document}